\documentclass[10pt]{article}

\usepackage{amssymb}
\usepackage{amsmath}
\usepackage{theorem}
\usepackage{epsfig}
\usepackage{verbatim}
\usepackage{graphicx}
\usepackage{subfigure}
\usepackage{cancel}
\usepackage{epstopdf}
\usepackage{a4wide}
\usepackage{pdflscape}

\usepackage{color}

\DeclareMathOperator\arcsinh{arcsinh}

\newtheorem{theorem}{Theorem}[section]
\newtheorem{lemma}[theorem]{Lemma}

\newtheorem{corollary}[theorem]{Corollary}

\newtheorem{rem}[theorem]{Remark}

\newcommand{\RR}{\mathbb{R}}

\newcommand{\N}{\mathbb{N}}
\newcommand{\T}{\mathbb{T}}
\newcommand{\F}{\mathcal{F}}

\newcommand{\supp}{{\rm supp}\thinspace}

\newcommand{\al}{\alpha}

\newcommand{\pa}{\partial}

\newcommand{\inti}{\int_{-\infty}^\infty}
\newcommand{\intpi}{\int_{-\pi}^\pi}

\newcommand{\rti}{\tilde{r}}

\newcommand{\sinc}[1]{\sin^2\left(\frac{#1}{2}\right)}
\newcommand{\p}[1]{\left(#1\right)}
\newcommand{\ab}[1]{\left|#1\right|}
\newcommand{\A}{A'[u]}
\newcommand{\Ag}{A'[u+tg]}
\newcommand{\D}{\sin^2\left(\frac{\alpha'}{2}\right)+b^2s'^2}
\newcommand{\dA}{\p{\sin^2\left(\frac{\alpha'}{2}\right)+bs'^2}}
\newcommand{\ddA}{\p{\sin^2\left(\frac{\alpha'}{2}\right)+b\left(\ab{\sin\p{\frac{\al'}{2}}}+|s'|+s'^2\right)}}
\newcommand{\ms}{\ab{\sin\p{\frac{\al'}{2}}}}

\newenvironment{proof}{\begin{trivlist} \item[] {\em Proof:}}{\hfill $\Box$
                       \end{trivlist}}

\makeatletter
\renewcommand*\l@section{\@dottedtocline{1}{0em}{1.5em}}
\renewcommand*\l@subsection{\@dottedtocline{2}{1.5em}{2.3em}}
\renewcommand*\l@subsubsection{\@dottedtocline{3}{3.8em}{3.7em}}
\makeatother

\numberwithin{equation}{section}

\allowdisplaybreaks[4]

\begin{document}

\title{Uniformly rotating smooth solutions for the incompressible 2D Euler equations}

\author{Angel Castro, Diego C\'ordoba and Javier G\'omez-Serrano}

\maketitle

\begin{abstract}

In this paper, we show the existence of a family of compactly supported smooth vorticities, which are solutions of the 2D incompressible Euler equation and rotate uniformly in time and space.

\vskip 0.3cm


\end{abstract}

\tableofcontents

\section{Introduction}
\label{sectionuno}

%Motivacion fisica

We consider the initial value problem for the two dimensional incompressible Euler equation in vorticity form:

\begin{align}
\partial_t \omega (x,t)+ v(x,t) \cdot \nabla \omega(x,t)& = 0, \ \ (x,t) \in \mathbb{R}^2 \times \mathbb{R}_+ \label{euler} \\
 v(x,t) & = (\nabla^{\perp}(-\Delta)^{-1}\omega)(x,t) \nonumber\\
\omega(x,0)&= \omega_0(x)\nonumber.
\end{align}

The aim of this paper is to construct a family of solutions which are compactly supported and smooth, and moreover its dynamics are given by a uniform rotation, both in time and space.

The problem of finding smooth stationary solutions has been addressed by Nadirashvili in \cite{Nadirashvili:stationary-2d-euler}, where he studies the geometry and the stability of solutions, following the works of Arnold \cite{Arnold:geometrie-differentielle-dimension-infinie,Arnold:apriori-estimate-hydrodynamic-stability,Arnold-Khesin:topological-methods-hydrodynamics}. Choffrut and \v{S}ver\'ak \cite{Choffrut-Sverak:local-structure-steady-euler} showed that locally near each stationary smooth solution there exists a manifold of stationary smooth solutions transversal to the foliation, and Choffrut and Sz\'ekelyhidi \cite{Choffrut-Szekelyhihi:weak-solutions-stationary-euler} showed that there is an abundant set of stationary weak ($L^{\infty}$) solutions near a smooth stationary one. Shvydkoy and Luo \cite{Luo-Shvydkoy:2d-homogeneous-euler,Luo-Shvydkoy:addendum-homogeneous-euler} provided a classification of stationary smooth solutions of the form $v = \nabla^{\perp}(r^{\gamma}f(\theta))$, where $(r,\theta)$ are polar coordinates.

Bedrossian and Masmoudi proved in \cite{Bedrossian-Masmoudi:inviscid-damping-2d-euler} that given an initial perturbation of the Couette flow small in a suitable regularity class, the velocity converges strongly in $L^{2}$ to a shear flow which is also close to the Couette flow. Recently, motivated by the numerical simulations of Luo and Hou \cite{Luo-Hou:singularities-euler-3d}, Kiselev and \v{S}verak showed in \cite{Kiselev-Sverak:double-exponential-euler-boundary} double exponential growth of the gradient of the vorticity in the presence of a boundary.

It is well known that radial functions are stationary solutions for \eqref{euler} due to the structure of the nonlinear term. The solutions that will be constructed in this paper are a smooth, compactly supported desingularization of a vortex patch, perturbed in a suitable direction. Moreover, the solutions will enjoy $m$-fold symmetry (invariance by a rotation of a $\frac{2\pi}{m}$ angle) for $m \geq 2$.

In addition the dynamics of these solutions consist of global rotating level sets with constant angular velocity. These level sets are a perturbation of the circle. We remark that we can extend our results to the generalized Surface Quasigeostrophic (gSQG) equations in the case $0 < \alpha < 1$. For those values of the parameter $\alpha$, the question of global existence vs finite time singularities is still open in the smooth case and this would be to our knowledge, the first nontrivial family of solutions for which there is global existence.

These equations interpolate between the 2D incompressible Euler and the Surface Quasigeostrophic (SQG) equations and are indexed by a parameter $\al \in [0,2)$, where $\al = 0$ corresponds to Euler and $\al = 1$ to SQG. They read as follows:

\begin{align*}
\left\{ \begin{array}{ll}
\partial_{t}\theta+u\cdot\nabla\theta=0,\quad(t,x)\in\mathbb{R}_+\times\mathbb{R}^2, &\\
v=-\nabla^\perp(-\Delta)^{-1+\frac{\alpha}{2}}\theta,\\
\theta_{|t=0}=\theta_0,
\end{array} \right.
\end{align*}

Our construction employs bifurcation theory from a patch supported in a disk which solves weakly \eqref{euler}. Our motivation to study the current problem comes from the aforementioned patch setting, where the vorticity is a characteristic function of a time-dependent set $\Omega(t)$. This property is conserved in time. Local existence in this class of solutions (assuming $C^{1,\gamma}$ regularity of the boundary) was first shown by Chemin \cite{Chemin:persistance-structures-fluides-incompressibles} and then simplified by Bertozzi-Constantin \cite{Bertozzi-Constantin:global-regularity-vortex-patches} (see also \cite{Serfati:preuve-directe-existence-globale-patches}). A classical theorem of Yudovich \cite{Yudovich:Nonstationary-ideal-incompressible} asserts that if the initial data is in $L^{1} \cap L^{\infty}$, then there is uniqueness within that class.  In a very recent preprint, \cite{Elgindi-Jeong:symmetries-fluids}, the $L^{1}$ assumption can be dropped upon having an appropriate symmetry ($m$-fold) condition.

The existence of uniformly rotating $m$-fold patches (also known as V-states) started with the numerical work of Deem and Zabusky \cite{Deem-Zabusky:vortex-waves-stationary}. Burbea in \cite{Burbea:motions-vortex-patches} proved the existence of V-states and $C^\infty$-regularity for its boundary was proved by Hmidi at al. in \cite{Hmidi-Mateu-Verdera:rotating-vortex-patch}. It was shown in \cite{Castro-Cordoba-GomezSerrano:analytic-vstates-ellipses} that the boundary of these solutions is actually analytic. The doubly connected setting was studied in \cite{Hmidi-delaHoz-Mateu-Verdera:doubly-connected-vstates-euler}. The existence of V-states was established for the generalized surface quasi-geostrophic equations in the range $0 < \alpha < 1$ in \cite{Hassainia-Hmidi:v-states-generalized-sqg} and in $1 \leq \alpha < 2$ in \cite{Castro-Cordoba-GomezSerrano:existence-regularity-vstates-gsqg}. We note that a desingularization from a Dirac distribution into a patch allows Hmidi and Mateu to prove the existence of corotating and counter-rotating vortex pairs in \cite{Hmidi-Mateu:existence-corotating-counter-rotating}. In fact, our strategy is inspired by their paper.

In \cite{Castro-Cordoba-GomezSerrano:global-smooth-solutions-sqg}, we constructed for the case $\al = 1$ a family of 3-fold solutions that rotated uniformly by looking at perturbations from a smooth annular profile (as opposed to a desingularization of a patch). In that case, the velocity is more singular than in 2D Euler. We had to overcome the following difficulties:

\begin{itemize}
\item The study of the linear problem was a functional equation, as opposed to a scalar equation (which was in the patch case). Even the existence of nontrivial elements in the kernel of the linear part was not evident a priori.

 \item There was no algebraic formula for neither the eigenvalue nor the eigenvector, not even in an implicit way (such as in \cite{Castro-Cordoba-GomezSerrano:analytic-vstates-ellipses}). This made the proof of the dimensionality of the kernel much harder since we had to show that the eigenvalue was simple and have some control of the rest of the eigenvalues. In order to do that, we resorted to a computer-assisted proof to rigorously bound the operators arising in the calculations.
\end{itemize}

The paper is organized as follows: section \ref{sectionequations} is devoted to the reformulation of the equations \eqref{euler}  in new variables and the statement of the main theorem. In section \ref{checking} we check that our equation satisfies the hypotheses of the Crandall-Rabinowitz theorem. Finally, in Appendix \ref{appendixbasicintegrals} we compute some basic integrals used in the calculation of the linear part and Appendix \ref{Aestimations} is devoted to prove some estimates used along section \ref{checking}. We note that since $\al < 1$ and the velocity is less singular than for $\al = 1$, we have been able to remove the computer from the estimates, as opposed to \cite{Castro-Cordoba-GomezSerrano:global-smooth-solutions-sqg}.

The proofs of  the theorems  rely on the Crandall-Rabinowitz theorem. We recall here the statement of this theorem from \cite{Crandall-Rabinowitz:bifurcation-simple-eigenvalues} for expository purposes.

\begin{theorem}[Crandall-Rabinowitz] \label{CR-theorem} Let $X$, $Y$ be Banach spaces, $V$ a neighborhood of
0 in $X$ and
\begin{align*}
\F \,:\, &   V\times (-1,\,1) \rightarrow Y\\
& \,\qquad (r,\,\mu)\,\,\,\,\rightarrow \F[r,\,\mu]
\end{align*}
have the properties
\begin{enumerate}
\item $\F[0,\, \mu] = 0$ for any $|\mu| < 1$.
\item The partial derivatives $\pa_\mu \F$, $\pa_r \F$ and $\pa^2 _{\mu r}\F$ exist and are continuous.
\item  $\mathcal{N}(\pa_r \F[0, 0])$ and $Y/\mathcal{R}(\pa_r \F[0, 0])$ are one-dimensional.
\item  $\pa^2_{\mu r} \F[0, 0]r_0\not\in
 \mathcal {R}(\pa_r \F[0, 0])$, where
$\mathcal{N}(\pa_r \F[0, 0]) = \text{span }r_0.$
\end{enumerate}
(Here $\mathcal{N}$ and $\mathcal{R}$ denote the kernel and range respectively). If $Z$ is any complement of $\mathcal{N}(\pa_r \F(0, 0))$ in $X$, then there is a neighborhood $U$ of (0, 0) in
$\RR \times X$, an interval $(-b, b)$, and continuous functions
\begin{align*}
\phi\, : \, (-b, b) \rightarrow \RR && \psi\, :\, (-b, b) \rightarrow Z\end{align*}
such that $\phi(0) = 0, \psi(0) = 0$ and
$$\F^{-1}(0)\cap U=\{ \xi r_0+\xi\psi(\xi), (\phi(\xi))\,:\, |\xi|<b\}\cup \{(t,0)\,:\, (t,0)\in U\}.$$
\end{theorem}

\section{The equations}
\label{sectionequations}

In this section we will obtain the equation for the level sets of the vorticity of a global rotating solution of the 2D Euler equation. We assume that at time $t=0$ these level sets can be parameterized by \begin{align}\label{x}x(\alpha,\rho)=r(\alpha,\rho)(\cos(\alpha),\, \sin(\alpha))\end{align} where $r(\alpha,\rho)$ is a scalar function $\alpha\in \T$ and $\rho\in\RR^+$. Since the level sets rotate with constant angular velocity the level sets of the solutions, $z(\alpha,\rho,t)$ satisfies
\begin{align}\label{levelsets}
\omega(z(\alpha,\rho,t),t)=f(\rho)
\end{align}
for some scalar function $f$ (which does not depend neither on $\alpha$ nor on $t$), and can be written as follows,
\begin{align}\label{rota}
z(\alpha,\rho,t) = \mathcal{O}(t)x(\alpha, \rho),
\quad
\mathcal{O}(t) =
\left(
\begin{array}{cc}
\cos(\lambda t) & \sin(\lambda t) \\
-\sin(\lambda t) & \cos(\lambda t)
\end{array}
\right).
\end{align}
By using the Euler equation and \eqref{levelsets} one can obtain that (see \cite{Castro-Cordoba-GomezSerrano:global-smooth-solutions-sqg} for details)
\begin{align*}
\label{leveleq}
&\left(-v(z(\al,\rho,t),t) + z_{t}(\alpha,\rho,t )\right)\cdot z_{\al}^{\perp}(\alpha,\rho,t) \frac{f_\rho(\rho)}{z_{\al}^{\perp} \cdot z_{\rho}(\alpha,\rho,t)}=  0
\end{align*}
where $$v(x,t)=\frac{1}{2\pi}\int_{\RR^2}\log(|x-y|)\nabla^\perp\omega(y,t)dy,$$
and by a change of coordinates we have that
\begin{align}
v(z(\alpha,\rho,t),t)=\frac{1}{2\pi} \int_0^\infty\int_{-\pi}^\pi \log\left(|z(\al,\rho,t)-z(\al',\rho',t)|\right) f_\rho(\rho') z_{\al}(\al',\rho',t) d\al' d\rho'.
\end{align}
Thus equation \eqref{leveleq} can be written in terms of $x(\alpha,\rho)$ in the following way

\begin{align}\label{equationx}
\lambda x (\alpha,\rho)\cdot x_{\al}(\alpha,\rho) + \frac{1}{2\pi} x_{\al}^{\perp}(\alpha,\rho) \cdot \int_{0}^\infty \int_{-\pi}^\pi f_\rho(\rho')\log\left(|x(\al,\rho)-x(\al',\rho')|\right) x_{\al}(\al',\rho') d\al' d\rho'=0.
\end{align}
for $\alpha \in \T$ and $\rho \in \text{supp}\left(f_\rho\right)$.

%Next we discuss the kind of $f$ we are going to consider. First we define the profile $\Phi(\rho)$
%\begin{align*}
%\Phi(\rho)=\left\{\begin{array}{ccc} 1 && \rho \in (-\infty,-1]\\
%1+\int_{-1}^\rho \phi(\rho') d\rho' && \rho \in (-1,1) \\
%0 && \rho \in [1,\infty)\end{array}\right.
%\end{align*}
%where $\int_{-1}^1\phi(\rho')d\rho' =-1$ and $\phi \in C^\infty_c((-1,1))$. Then the function $f^a(\rho)$ (we make explicit the dependence in the parameter $a$) will be given by $$f^a(\rho)=\Phi\left(1+\frac{2}{a}\left(\rho-1\right)\right)$$
%and it is easy to check that $\text{supp}(f^a_\rho)\subset (1-a,1)$.
%
In terms of the function $r(\alpha,\rho)$ we get
\begin{align}\label{ecur}
&\F_{f}[r,\,\lambda] \\&\equiv
\lambda r(\alpha,\rho)r_{\al}(\al,\rho)  + \frac{r(\alpha,\rho)}{2\pi} \int_0^\infty \int_{-\pi}^\pi f_\rho(\rho')\log(|x(\al,\rho)-x(\al',\rho')|)\cos(\al-\al')r_{\al}(\al',\rho') d\al' d\rho' \nonumber\\
& - \frac{r_{\al}(\al,\rho)}{2\pi } \int_0^\infty \int_{-\pi}^\pi f_\rho(\rho')\log(|x(\al,\rho)-x(\al',\rho')|)\cos(\al-\al')r(\al',\rho') d\al' d\rho'\nonumber \\
& + \frac{1}{2\pi} \int_0^\infty \int_{-\pi}^\pi f_\rho(\rho')\log(|x(\al,\rho)-x(\al',\rho')|)\sin(\al-\al')(r(\al,\rho)r(\al',\rho') + r_{\al}(\al,\rho)r_{\al}(\al',\rho')) d\al' d\rho'\nonumber\\
&=0,\nonumber
\end{align}
for $\alpha\in \T$ and $\rho\in (1-a,1+a)$ and $|x(\al,\rho)-x(\al',\rho')|=\sqrt{r(\alpha,\rho)^2+r(\alpha',\rho')^2-2r(\alpha,\rho)r(\alpha',\rho')\cos(\alpha-\alpha')}.$

We discuss the kind of $f$ we will use in this paper.

First we define the profile $F(\rho)$
\begin{align*}
F(\rho)=\left\{\begin{array}{ccc} 1 && \rho \in (-\infty,-1]\\
1+\int_{-1}^\rho \phi(\rho') d\rho' && \rho \in (-1,1) \\
0 && \rho \in [1,\infty)\end{array}\right.
\end{align*}
where $\int_{-1}^1\phi(\rho')d\rho' =-1$ and $\phi \in C^3_c((-1,1))$. Then the function $f^a(\rho)$ (we make explicit the dependence in the parameter $a$) will be given by $$f^a(\rho)=F\left(\frac{\rho-1}{a}\right)$$
and it is easy to check that $\text{supp}(f^a_\rho)\subset (1-a,1+a)$. We also remark that $f^{a}_{\rho}\left(1+as \right)$ is independent of $a$, a fact that will be used later.

The parameter $a$ will be used as the bifurcation parameter in order to find solutions for \eqref{ecur}. Indeed the angular velocity $\lambda$ will depend on $a$ (we will make explicit this dependence later). Thus, we will solve the equation
\begin{align}\label{ecura}
\F[r,a]=0
\end{align}
with $\F[r,a]=\F_{f^a}[r,\lambda(a)].$

\subsection{Rescaling equation \eqref{ecur}}
In order to find solutions we start by  rescaling the equation \eqref{ecur}. Changing variables,  $s=\frac{\rho-1}{a}$, one   obtains that
\begin{align}\label{aux}
&\lambda(a) r(\alpha,1+as)r_{\al}(\al,1+as)  \nonumber\\ &+ \frac{r(\alpha,1+as)}{2\pi} \int_{-1}^1 \int_{-\pi}^\pi F_\rho(s')\log(|x(\al,1+as)-x(\al',1+as')|)\cos(\al-\al')r_{\al}(\al',1+as') d\al' ds' \nonumber\\
& - \frac{r_{\al}(\al,1+as')}{2\pi } \int_{-1}^1 \int_{-\pi}^\pi F_\rho(s')\log(|x(\al,1+as)-x(\al',1+as')|)\cos(\al-\al')r(\al',1+as') d\al' ds' \nonumber\\
& + \frac{1}{2\pi} \int_{-1}^1 \int_{-\pi}^\pi F_\rho(s')\log(|x(\al,1+as)-x(\al',1+as')|)\sin(\al-\al')r(\al,1+as)r(\al',1+as')  d\al' ds',\nonumber\\
& + \frac{1}{2\pi} \int_{-1}^1 \int_{-\pi}^\pi F_\rho(s')\log(|x(\al,1+as)-x(\al',1+as')|)\sin(\al-\al') r_{\al}(\al,1+as)r_{\al}(\al',1+as')) d\al' ds',\nonumber\\
&=0,
\end{align}
 for $\alpha\in (-\pi,\pi]$ and $s\in (-1,1)$.

The equation we will try to solve is the one we obtain dividing \eqref{aux} by $a^2$ and by replacing $r(\alpha,1+as)$  by $1+a^2(s+ \rti(\alpha,s))$ (we will solve in $a$ and $\rti$), in such a way that $r(\alpha,\rho)=1-a+a\rho+a^2\rti(\alpha,\frac{\rho-1}{a})$ will be a solution of \eqref{ecura} for any $a \neq 0$. This equation reads

\begin{align*}
&\lambda(a)\left(1+a^2(s+\rti(\alpha,s))\right)\rti_\alpha(\alpha,s)\\
&+ \frac{1+a^2\rti(\alpha,s)}{4\pi} \int_{-1}^1 \int_{-\pi}^\pi F_\rho(s')\log(A[\rti,a])\cos(\al-\al')\rti_{\al}(\al',s') d\al' ds' \nonumber\\
&- \frac{\rti_{\al}(\al,s)}{4\pi} \int_{-1}^1 \int_{-\pi}^\pi F_\rho(s')\log(A[\rti,a])\cos(\al-\al')(1+a^2(s'+\rti(\alpha',s'))) d\al' ds' \nonumber\\
& + \frac{1}{4\pi} \int_{-1}^1 \int_{-\pi}^\pi F_\rho(s')\frac{1}{a^2}\log(A[\rti,a])\sin(\alpha-\alpha') d\al' ds',\nonumber\\
&+ \frac{1}{4\pi} \int_{-1}^1 \int_{-\pi}^\pi F_\rho(s')\log(A[\rti,a])\sin(\al-\al') (s+\rti(\alpha,s)+s'+\rti(\alpha',s'))  d\al' ds',\nonumber\\
& + \frac{a^2}{4\pi} \int_{-1}^1 \int_{-\pi}^\pi F_\rho(s')\log(A[\rti,a])\sin(\al-\al')\left(\rti_\alpha(\al,s)\rti_\alpha(\al',s')+(s+\rti(\alpha,s))(s'+\rti(\alpha',s'))\right)  d\al' ds' = 0,\nonumber \\
\end{align*}
with

\begin{align*}
&A[\rti,a]=A[\rti,a](s,s',\al,\al')\\&=(1+a^2(s+\rti(\alpha,s)))^2+(1+a^2(s'+\rti(\alpha',s')))^2-2(1+a^2(s+\rti(\alpha,s)))(1+a^2(s'+\rti(\alpha',s')))\cos(\alpha-\alpha')\\
&= 4\sinc{\alpha-\alpha'}+4a^2(s+\rti(\alpha,s)+s'+\rti(\alpha',s'))\sinc{\alpha-\alpha'}\\
&+a^4\left(4(s+\rti(\alpha,s))(s'+\rti(\alpha',s'))\sinc{\alpha-\alpha'}+(s+\rti(\alpha,s)-s-\rti(\alpha',s'))^2\right).
\end{align*}
Since $A[\rti,0]=4\sinc{\alpha-\alpha'}$ we have that
\begin{align*}
&\frac{1}{4\pi} \int_{-1}^1 \int_{-\pi}^\pi F_\rho(s')\frac{1}{a}\log(A[\rti,a])\sin(\al-\al') d\al' ds'\\
&=\frac{1}{4\pi} \int_{-1}^1 \int_{-\pi}^\pi F_\rho(s')\frac{1}{a}\log\left(\frac{A[\rti,a]}{A[\rti,0]}\right)\sin(\al-\al')d\al' ds'
\end{align*}

Therefore we need to solve the equation
\begin{align}\label{ecuG}
\mathcal{G}[\rti,a]=0
\end{align}
with
\begin{align}\label{G}
&\mathcal{G}[\rti,a]\nonumber\\
&=\lambda(a)\left(1+a^2(s+\rti(\alpha,s))\right)\rti_\alpha(\alpha,s)\nonumber\\
&+ \frac{1+a^2(s+\rti(\alpha,s))}{4\pi} \int_{-1}^1 \int_{-\pi}^\pi F_\rho(s')\log(A[\rti,a])\cos(\al-\al')\rti_{\al}(\al',s') d\al' ds' \nonumber\\
&- \frac{\rti_{\al}(\al,s)}{4\pi} \int_{-1}^1 \int_{-\pi}^\pi F_\rho(s')\log(A[\rti,a])\cos(\al-\al')(1+a^2(s'+\rti(\alpha',s'))) d\al' ds' \nonumber\\
& + \frac{1}{4\pi} \int_{-1}^1 \int_{-\pi}^\pi F_\rho(s')\frac{1}{a^2}\log\left(\frac{A[\rti,a]}{A[\rti,0]}\right)\sin(\alpha-\alpha') d\al' ds',\nonumber\\
&+ \frac{1}{4\pi} \int_{-1}^1 \int_{-\pi}^\pi F_\rho(s')\log(A[\rti,a])\sin(\al-\al') (s+\rti(\alpha,s)+s'+\rti(\alpha',s')  d\al' ds',\nonumber\\
& + \frac{a^2}{4\pi} \int_{-1}^1 \int_{-\pi}^\pi F_\rho(s')\log(A[\rti,a])\sin(\al-\al')\left(\rti_\alpha(\al,s)\rti_\alpha(\al',s')+(s+\rti(\alpha,s))(s'+\rti(\alpha',s'))\right)  d\al' ds' ,\nonumber \\
\end{align}

We now state the main theorem of this paper:
\begin{theorem}\label{mainthm}
 Consider the domain $$\Omega\equiv \{ (\alpha,\, s)\,:\,\alpha\in \T,\, -1<s<1\}.$$  and let the function  $ \phi$ be as in section \ref{sectionequations}. Let $m\geq 2$ be an integer and $\lambda(a)=\lambda_m(0)+\frac{d\lambda}{da}(0) a$, with $\lambda_m(0)=\frac{m-1}{2m}$ and $\frac{d\lambda}{da}(0)\neq 0$.  Then there exists a branch of nontrivial smooth solutions, with $m-$fold symmetry, of equation \eqref{ecuG}, in $H^{4,3}(\Omega)$, bifurcating from $\rti(\alpha,s)=0$ and  $a=0$.
\end{theorem}

The proof of Theorem \ref{mainthm} consists of checking the hypotheses of Theorem \ref{CR-theorem}. This is done in the next section.

\begin{corollary}
There exist global rotating solutions for the 2D-Euler vorticity equation with $C^2$-regularity with compact support, with $m$-fold symmetry for any integer $m\geq 2$.
\end{corollary}
\begin{proof}
We fix $a>0$ small enough in such a way that $\rti(\alpha,\rho)$ is the solution of \eqref{ecuG}, given by Theorem \ref{mainthm}, for that $a$. Then $r(\alpha,\rho)=1-a+a\rho+a^2\rti(\alpha,\frac{\rho-1}{a})$ is a solution of equation \eqref{ecura}. In order to check that the vorticity defined through the equations \eqref{x}, \eqref{levelsets} and \eqref{rota} is a $C^2$ rotating solution of the 2D Euler equation we just need to check that $\pa_\rho r(\alpha,\rho)>0$ (in order to have that $z(\alpha,\rho)$ is one-to-one in $\al\in\T$, $\rho \in [1-a,1+a]$. But $\pa_\rho r(\alpha,\rho)=a(1+(\pa_s\rti)(\alpha,\frac{\rho-1}{a}))$
 and since we are bifurcating from $\tilde{r} = 0$ the result holds for small enough $a$.
\end{proof}

\begin{rem}
We notice that that the sufficient stability condition \cite[equation (3.10), p. 116]{Marchioro-Pulvirenti:incompressible-fluids} may not be satisfied by our solutions (take for example any positive $f^{a}(\rho)$).
\end{rem}

\section{Checking the hypotheses of the Crandall-Rabinowitz theorem and proof of Theorem \ref{mainthm}}\label{checking}

\subsection{Step 1. The functional Setting and the hypothesis 1}

The beginning of this section is devoted to defining the spaces we will work with. After that the main purpose   will be to prove lemma \ref{GXY}.

Let us also define the spaces $H^{4,3}(\Omega)$ as follows:
\begin{align}\label{Hkl}
r \in H^{4,3} \Leftrightarrow \left\{ r\in L^2(\Omega) \, : \, ||r||^2_{L^2(\Omega)}+ ||\pa^{3}_\rho r||_{L^2(\Omega)}^2+\sum_{j=0}^3||\pa^{4-j}_\alpha \pa_\rho^{j} r||_{L^2(\Omega)}^{2}<\infty\right\}.
\end{align}

Moreover, we will say that $r \in H^{4,3}_{m,even}$ (resp. $r\in H^{4,3}_{m,odd}$) if $r \in H^{4,3}$, has $m$-fold symmetry and is even (resp. odd).

\begin{lemma}\label{GXY}Let $\mathcal{G}$ be as in \eqref{G}. Then there exist  numbers $\delta>0$ and  $a_0>0$ small enough such that
\begin{align*}
\mathcal{G}\,:\, V^\delta\times (-a_0, a_0) & \to H^{3,3}_{m,odd}(\Omega),\\
\quad (r,a) & \to \mathcal{G}[r,a].
\end{align*}
where
\begin{align*}
V^\delta=\{ r\in H^{4,3}(\Omega)\,:\, ||r||_{H^{4,3}_{m,even}(\Omega)}< \delta\}.
\end{align*}
\end{lemma}

\begin{proof}

We will use the following notation. Given any function $f(\al,s)$ we will write
\begin{align*}
f=&f(\alpha,s)\\
f'=&f(\alpha-\alpha',s-s').
\end{align*}
In addition either we will omit the dependence of the $A[\rti,a](s,s',\al,\al')$ on $\rti$, $a$, $s$, $s'$, $\al$ and $\al'$ or we will only make it explicit with respect to the parameters we are interested in.   Also
$$A' \equiv A[\rti,a](s,s-s',\alpha-\alpha').$$

We will use the following convention to take derivatives: $\pa^k$ means either $\pa^k_\alpha$ or $\pa^k_\rho$.

It will be convenient to alleviate the notation to use the function $u(\alpha,s)=s+\rti(\alpha,s)$ in such a way that

\begin{align}\label{Gu}
&\mathcal{G}[u-s,a]\nonumber\\
&=\lambda(a)\left(1+a^2 u\right)u_\alpha\nonumber\\
&+ \frac{1+a^2u}{4\pi} \int_{-\infty}^\infty \int_{-\pi}^\pi F_\rho(s-s')\log(A')\cos(\al')u'_{\al} d\al' ds' \nonumber\\
&- \frac{u_{\al}}{4\pi} \int_{-\infty}^\infty \int_{-\pi}^\pi F_\rho(s-s')\log(A')\cos(\al')(1+a^2u')) d\al' ds' \nonumber\\
& + \frac{1}{4\pi} \int_{-\infty}^\infty \int_{-\pi}^\pi F_\rho(s-s')\frac{1}{a^2}\log\left(\frac{A'}{A'[a=0]}\right)\sin(\alpha') d\al' ds',\nonumber\\
&+ \frac{1}{4\pi} \int_{-\infty}^\infty \int_{-\pi}^\pi F_\rho(s-s')\log(A')\sin(\al') (u+u')  d\al' ds',\nonumber\\
& + \frac{a^2}{4\pi} \int_{-\infty}^\infty \int_{-\pi}^\pi F_\rho(s-s')\log(A')\sin(\al')\left(u_\alpha u_\alpha'+u u'\right)  d\al' ds',\nonumber \\
&\equiv \sum_{i=0}^5\mathcal{G}_i[\rti,a].
\end{align}

The proof of the estimates on $A$ that will be used to prove the lemma are left to the Appendix \ref{Aestimations}.

We will bound the terms $\mathcal{G}_i[\rti,a]$ in the equation \eqref{G} in $H^3(\Omega)$. The bound for $\mathcal{G}_0[\rti,a]$ is straightforward since $H^{4,3}\subset  C^2$, as it is proven in \cite[Lemma 4.1]{Castro-Cordoba-GomezSerrano:global-smooth-solutions-sqg}.

The proof for $\mathcal{G}_i[\rti,a]$, with $i=1,2,4,5$ is similar. We will give the details only for $\mathcal{G}_1[\rti,a]$. Finally we will bound $\mathcal{G}_3[\rti,a]$.

The $L^2-$ norm of $\mathcal{G}_1[u-s,a]$ it is easy to bound by using Lemma \ref{A1}--2, then, in order to bound $\mathcal{G}_3[u-s,a]$ in $H^3$, we take three derivatives to obtain
\begin{align*}
&\pa^3\mathcal{G}_1[u-s,a]=\frac{1}{4\pi}\inti\intpi \cos(\alpha')(1+a^2u)F_\rho(s-s')u'_\alpha \pa^3\log(A') d\al' ds'\\
&+\frac{1}{4\pi}\inti\intpi \left(\pa^3\left(\log(A')\cos(\alpha')(1+a^2u)F_\rho(s-s')u'_\alpha\right)-(1+a^2u)F_\rho(s-s')u'_\alpha \pa^3\log(A')\right) d\alpha' ds'.
\end{align*}

\begin{rem}Actually we also need to take mixed derivatives, as for example $\pa^2_\alpha\pa_\rho$. We will not compute explicitly these mixed derivatives but similar arguments with small modifications apply to them.
\end{rem}
To estimate the first term in the right hand side of the previous expression we use lemma \ref{A5} and to bound the second one we use \ref{A4}.

We now bound $\mathcal{G}_3$. The estimate of the $L^2-$norm follows from lemma \ref{l2logfinal}. Then we are concerned with the bound in $L^2$ of $\pa^3\mathcal{G}_3$,
\begin{align*}
&\pa^3\mathcal{G}_3[u-s,a]=\frac{1}{4\pi}\inti\intpi F_\rho(s-s')\sin(\alpha')\frac{1}{a^2}\pa^3\log(A')d\alpha'ds'\\
&+\frac{1}{4\pi}\inti\intpi \left(\pa^3\left(F_\rho(s-s')\frac{1}{a^2}\log\left(\frac{A'}{A'[0]}\right)\right)-F_\rho(s-s')\frac{1}{a^2}\pa^3\log(A')\right)\sin(\al')d\al'ds'.
\end{align*}
To estimate the fist term in the right hand side of the previous expression we can use lemma \ref{A5}. To bound the second one we can use lemmas \ref{l2logfinal} and \ref{A4}. We remark that we could use the cancellation given by the factor $\sin(\alpha')$ in the integrand of $\mathcal{G}_3$ but actually we do not need it to prove the estimate of $\mathcal{G}_3$ in $H^3$.

In addition we need to prove that $\mathcal{G}[\rti,a]$ is continuous with respect to $a$ and $\rti$ in $(-a_0,a_0)\times V^\delta$. In the next section we will prove that, in fact, it is $C^1$ in this domain.

Therefore, in order to prove that, $$\mathcal{G}\,:\, V\times (-1,1)\to H^{3,3}_{3, \text{odd}}(\Omega_a)$$
we just need to show that if \begin{align*} \tilde{r}(-\alpha,\rho) & = \tilde{r}(\alpha,\rho)\end{align*}
and
\begin{align*}
\tilde{r}\left(\alpha+\frac{2n\pi}{m},\rho\right) & = \tilde{r}(\alpha,\rho)
\end{align*} for $n\in \N$, then
\begin{align*}
\mathcal{G}(-\alpha,\rho) & = -\mathcal{G}(\alpha,\rho)
\end{align*}
and
\begin{align*}
\mathcal{G}\left(\alpha+\frac{2n\pi}{m},\rho\right) & = \mathcal{G}(\alpha,\rho)
\end{align*}
for $n\in \N$. These two properties are easy to check.

The last part of this section will be to check that the hypothesis 1 in the C-R theorem holds. This fact is a consequence of radial functions being stationary solutions of the Euler equation but let us check it on \eqref{G}. If we take $\tilde{r} = 0$, the only term in \eqref{G} that is not trivially zero is $\mathcal{G}_{3}[0,a]$. If $a \neq 0$, then it is immediate following from the oddness of the integrand in $\alpha'$. If $a = 0$,

\begin{align*}
\lim_{a\to 0}\mathcal{G}_{3}[0,a] = \frac{1}{4\pi}\int_{-\infty}^{\infty}\int_{-\pi}^{\pi}F_{\rho}(s-s')\frac{(\pa_{a^{2}}A')[0,0]}{A'[0,0]} \sin(\al-\al')d\al'ds',
\end{align*}

where the integrand is also odd in $\alpha'$.

\end{proof}

\subsection{Step 2. The derivatives}

\begin{lemma}
The derivative $\pa_{a}\mathcal{G}[u-s,a]$ is given by:

\begin{align}\label{daG}
\pa_{a}\mathcal{G}[u-s,a] & = \frac{d\lambda}{da}(a)\left(1+a^2u\right)u_\alpha + 2a\lambda(a)u u_\alpha \nonumber\\
&+ \frac{au}{2\pi} \inti \int_{-\pi}^\pi F_\rho(s-s')\log(A')\cos(\al')u'_\al d\al' ds' \nonumber\\
&+ \frac{1+a^2u}{4\pi} \inti \int_{-\pi}^\pi F_\rho(s-s')\frac{\pa_a A'}{A'}\cos(\al')u'_{\al} d\al' ds' \nonumber\\
&- \frac{u_\alpha}{4\pi} \inti \int_{-\pi}^\pi F_\rho(s-s')\frac{\pa_a A'}{A'}\cos(\al')(1+a^2u') d\al' ds' \nonumber\\
&- \frac{a u_{\al}}{2\pi} \inti \int_{-\pi}^\pi F_\rho(s-s')\log(A')\cos(\al')u' d\al' ds' \nonumber\\
& - \frac{1}{4\pi} \inti \int_{-\pi}^\pi F_\rho(s-s')\pa_a\left(\frac{1}{a^2}\log\left(\frac{A'}{A'[a=0]}\right)\right)\sin(\al') d\al' ds',\nonumber\\
&+ \frac{1}{4\pi} \inti \int_{-\pi}^\pi F_\rho(s-s')\frac{\pa_a A'}{A'}\sin(\al') (u+u')  d\al' ds',\nonumber\\
&+ \frac{a^2}{4\pi} \inti \int_{-\pi}^\pi F_\rho(s-s')\frac{\pa_a A'}{A'}\sin(\al') (uu'+u_\alpha u'_\alpha)  d\al' ds',\nonumber\\
&+\frac{a}{2\pi}\inti\intpi F_\rho(s-s')\log(A')\sin(\al') (uu'+u_\alpha u'_\alpha)  d\al' ds',
\end{align}
Moreover this derivative is continuous in $H^3$.
\end{lemma}
\begin{proof}
We will give the details of the differentiability with respect to $a$ of the term
\begin{align*}
\mathcal{G}_{11}[u-s,a]=\frac{1}{4\pi}\inti\intpi F_s(s-s')\cos(\alpha')\log(A') u'_\alpha d\alpha' ds'
\end{align*} and of $\mathcal{G}_3$
which are the most singular ones. First we notice that for $a\neq 0$ the formula
\begin{align*}
\pa_a \mathcal{G}_{11}[u-s,a]=\frac{1}{4\pi}\inti\intpi F_s(s-s')\cos(\alpha') \frac{\pa_a A'}{A'} u'_\alpha d\alpha' ds'
\end{align*}
holds. In addition, by using lemmas \ref{dA4} and \ref{dA5} we obtain that, for $a\neq 0$, $\pa_a\mathcal{G}_{11}\in H^3$ and that $\lim_{a\to 0^\pm} \pa_a \mathcal{G}_{11}[u-s,a]=0$ in $H^3$. It  remains to prove that $\lim_{h\to 0^\pm}\frac{\mathcal{G}_{11}[u-s,h]-\mathcal{G}_{11}[u-s,0]}{h}=0$ in $H^3$. This is a consequence of lemmas \ref{A4bis} and \ref{A5bis}.

Next we compute the derivative with respect to $a$ of $\mathcal{G}_3$ at the point $a=0$. First we show that
\begin{align*}
\lim_{a \to 0}\mathcal{G}_3[u-s,a]=\frac{1}{4\pi}\inti\intpi F_\rho(s-s')(u+u')\sin(\al')d\alpha'ds'
\end{align*}
This is done by using lemmas \ref{B15} and \ref{B16}.

Thus
\begin{align*}
\mathcal{G}_3[u-s,a]-\mathcal{G}_3[u-s,0]=\frac{1}{4\pi }\inti\intpi F_\rho(s-s')\p{\frac{1}{a^2}\log\p{\frac{A'[a]}{A'[0]}}-(u+u')}\sin(\alpha')d\alpha'ds'\end{align*}

And we can check that $$\lim_{a\to 0}\frac{\mathcal{G}_3[u-s,a]-\mathcal{G}_3[u-s,0]}{a}=0$$ in $H^3$ by using again lemmas \ref{B15} and \ref{B16}.

In addition we have that
\begin{align*}
\frac{1}{4\pi} \inti \int_{-\pi}^\pi F_\rho(s-s')\pa_a\left(\frac{1}{a^2}\log\left(\frac{A'}{A'[a=0]}\right)\right)\sin(\al') d\al' ds'
\end{align*}
is in $H^3$ for every $a\neq 0$ and that
\begin{align*}
\lim_{a\to 0}\frac{1}{4\pi} \inti \int_{-\pi}^\pi F_\rho(s-s')\pa_a\left(\frac{1}{a^2}\log\left(\frac{A'}{A'[a=0]}\right)\right)\sin(\al') d\al' ds'=0
\end{align*}
in $H^3$ by lemma \ref{B18} and \ref{B19}.

Finally we will prove the continuity with respect to $u$. Let us focus on the term
\begin{align*}
\pa_a \mathcal{G}_{11}[u-s,a]=\frac{1}{4\pi}\inti\intpi F_s(s-s')\cos(\alpha')\frac{\pa A'[u]}{A'[u]}u'_\alpha d\alpha'ds'.
\end{align*}
We want to estimate the difference $||\pa_a \mathcal{G}_{11}[u-s,a]-\pa_a \mathcal{G}_{11}[v-s,a]||_{H^3}$, with $u$, $v\in H^{4,3}$. Here, the most singular terms are
\begin{align*}
&J_1=\ab{\ab{\frac{1}{4\pi}\inti\intpi F_s(s-s')\cos(\alpha')\p{\frac{\pa_a A'[u]}{A'[u]}\pa^3u'_\alpha-\frac{\pa_a A'[v]}{A'[v]}\pa^3v'_\alpha} d\alpha'ds'}}_{L^2}\\
&J_2=\ab{\ab{\frac{1}{4\pi}\inti\intpi F_s(s-s')\cos(\alpha')\p{\pa^3\p{\frac{\pa_a A'[u]}{A'[u]}}u'_\alpha-\pa^3\p{\frac{\pa_a A'[v]}{A'[v]}}v'_\alpha} d\alpha'ds'}}_{L^2}.
\end{align*}
In $J_1$ the most singular terms come from the differences
\begin{align*}
\frac{\pa^3 \pa_a A'[u]}{A'[u]}u'_\alpha-\frac{\pa^3 \pa_a A'[v]}{A'[v]}v'_\alpha
\end{align*}
and
\begin{align*}
\pa_a A'[u]u'_\alpha \pa^3\frac{1}{A'[u]}-\pa_a A'[v]v'_\alpha \pa^3\frac{1}{A'[v]}.
\end{align*}
We will focus on
\begin{align*}
J_{21}=\ab{\ab{\frac{1}{4\pi}\inti\intpi F_s(s-s')\cos(\alpha')\p{\frac{\pa^3\pa_a A'[u]}{A'[u]}u'_\alpha-\frac{\pa^3\pa_a A'[v]}{A'[v]}v'_\alpha} d\alpha'ds'}}_{L^2}.
\end{align*}
We will deal with this term by splitting
\begin{align*}
&\frac{u'_\alpha}{A'[u]}\pa^3\pa_a A'[u]-\frac{v'_\alpha}{A'[v]}\pa^3\pa_a A'[v]\\
&=\frac{u'_\alpha}{A'[u]}\p{\pa^3\pa_aA'[u]-\pa^3\pa_a A'[v]}+\pa^3\pa_a A'[u]\p{\frac{u'_\alpha}{A'[u]}-\frac{v'_\alpha}{A'[v]}}
\end{align*}
We will give the details of the bound of
\begin{align*}
J_{211}=\ab{\ab{\frac{1}{4\pi}\inti\intpi F_s(s-s')\cos(\alpha')\p{\frac{u'_\alpha}{A[u]}\p{\pa^3\pa_aA'[u]-\pa^3\pa_a A'[v]}} d\alpha'ds'}}_{L^2}.
\end{align*}

We notice that
\begin{align*}
&\pa^3\pa_a A'[u]-\pa^3\pa_a A'[v]=8a \pa^3(u+u')\sinc{\al'}+4a^3\p{4\pa^3(uu')\sinc{\al'}+\pa^3(u-u')^2}\\
&8a \pa^3(v+v')\sinc{\al'}+4a^3\p{4\pa^3(vv')\sinc{\al'}+\pa^3(v-v')^2},
\end{align*}
and the most singular terms here are
\begin{align*}
&8a \pa^3(u-v+u'-v')\sinc{\al'}\\&+4a^3\p{4\p{u'\pa^3u-v'\pa^3v+ u\pa^3u'-v\pa^3v'}\sinc{\al'}+2(u-u')\pa^3(u-u')-2(v-v')\pa^3(v-v')},
\end{align*}
and it is enough to consider the terms
\begin{align*}
&8a \pa^3(u-v+u'-v')\sinc{\al'}\\&+4a^3\p{4\p{u'\pa^3(u-v)+u\pa^3(u'-v')}\sinc{\al'}+2(u-u')\pa^3(u-v-(u'-v'))}.
\end{align*}
Then we have to estimate
\begin{align*}
&\left|\frac{u'_\alpha}{A'[u]}\left(8a \pa^3(u-v+u'-v')\sinc{\al'}\right.\right.\\&\left.\left.+4a^3\p{4\p{u'\pa^3(u-v)+u\pa^3(u'-v')}\sinc{\al'}+2(u-u')\pa^3(u-v-(u'-v'))}\right)\right|\\
&\leq C a \frac{\sinc{\al'}}{\D}\p{|\pa^3(u-v)|+|\pa^3(u'-v')|}+Ca^3\frac{\sinc{\alpha'}+s'^2}{\D}\p{|\pa^3(u-v)|+|\pa^3(u'-v')|},
\end{align*}
where we have used lemma \ref{A1}. Then by applying lemma \ref{A2} we can obtain a suitable estimate for $J_{211}$. The term $J_2$ can be bounded in a similar way.

This concludes the proof of lemma \ref{daG}.

\end{proof}

\begin{lemma}\label{du}
The derivative $\pa_{\rti} \mathcal{G}[\rti,a]g=\pa_u \mathcal{G}[u-s,a]$ is given by:

\begin{align*}
\pa_{u} \mathcal{G}[u-s,a]g & = a^2 \lambda(a) g\,u_\alpha +\lambda(a)\left(1+a^2u\right)g_\alpha \nonumber \\
&+ \frac{a^2g}{4\pi} \int_{-\infty}^\infty \int_{-\pi}^\pi F_\rho(s-s')\log(A')\cos(\al')u'_\al d\al' ds' \nonumber\\
&+ \frac{1+a^2u}{4\pi} \inti \int_{-\pi}^\pi F_\rho(s-s')\frac{(\pa_{u} A')[g]}{A'}\cos(\al')u'_{\al} d\al' ds' \nonumber\\
&+ \frac{1+a^2u}{4\pi} \inti \int_{-\pi}^\pi F_\rho(s-s')\log(A')\cos(\al')g'_{\al} d\al' ds' \nonumber\\
&- \frac{g_{\al}}{4\pi} \inti \int_{-\pi}^\pi F_\rho(s-s')\log(A')\cos(\al')(1+a^2u') d\al' ds' \nonumber\\
&- \frac{u_{\al}}{4\pi} \inti \int_{-\pi}^\pi F_\rho(s-s')\frac{(\pa_{u} A')[g]}{A'}\cos(\al')(1+a^2u') d\al' ds' \nonumber\\
&- \frac{u_{\al}}{4\pi} \inti \int_{-\pi}^\pi F_\rho(s-s')\log(A')\cos(\al')a^2g' d\al' ds' \nonumber\\
& + \frac{1}{4\pi a^2} \inti \int_{-\pi}^\pi F_\rho(s-s')\frac{(\pa_{u} A')[g]}{A'}\sin(\al') d\al' ds',\nonumber\\
&+ \frac{1}{4\pi} \inti \int_{-\pi}^\pi F_\rho(s-s')\frac{(\pa_{u} A')[g]}{A'}\sin(\al') \left(u+u'\right)  d\al' ds',\nonumber\\
&+ \frac{1}{4\pi} \inti \int_{-\pi}^\pi F_\rho(s-s')\log(A')\sin(\al') \left(g+g'\right)  d\al' ds',\nonumber\\
& + \frac{a^2}{4\pi} \inti \int_{-\pi}^\pi F_\rho(s-s')\frac{(\pa_{u} A')[g]}{A'}\sin(\al')\left(uu'+u_\al u_\al'\right)  d\al' ds'\nonumber \\
& + \frac{a^2}{4\pi} \inti \int_{-\pi}^\pi F_\rho(s-s')\log(A')\sin(\al')\left(gu'+g'u+g_\al u_\al'+g'_\al u_\al\right)  d\al' ds'.\nonumber \\
\end{align*}
Moreover this derivative is continuous in $H^3$.
\end{lemma}
\begin{proof} We will focus in the derivatives of the terms $\mathcal{G}_1$ and $\mathcal{G}_3$ which are the most difficult ones.

In order to differentiate $\mathcal{G}_1$ with respect to $u$ we notice that thanks to lemma \ref{GXY} in order to prove Lemma \ref{du} is enough to compute the derivative of the term
\begin{align*}
\mathcal{G}_{11}[u-s,a]=\frac{1}{4\pi}\inti\intpi F_\rho(s-s')\log\left(A'\right)\cos(\alpha')u'_\alpha d\alpha'ds'.
\end{align*}
Thus we need to show that
\begin{align*}
&\lim_{t\to 0} \frac{\mathcal{G}_{11}[u+t g-s,a]-\mathcal{G}_{11}[u-s,a]}{t}=\frac{1}{4\pi}\inti\intpi F_{\rho}(s-s')\log\left(A'[u]\right)\cos(\al')g'_\alpha d\alpha'ds'\\
&+\frac{1}{4\pi}\inti\intpi F_{\rho}(s-s')\frac{\pa_u A'[g]}{A'}\cos(\alpha')u'_\alpha d\al' ds'
\end{align*}
is in $H^3$ for $g\in H^3$, with $||g||_{H^3}=1$, and that this derivative is continuous. We have that
\begin{align*}
&\mathcal{G}_{11}[u+t g -s, a]-\mathcal{G}_{11}[u-s,a]\\&=\frac{1}{4\pi}\inti\intpi F_{\rho}(s-s')\cos(\alpha')\left(\log\left(A'[u+tg]\right)(u'+tg')_\alpha-\log\left(A'[u]\right)u'_\alpha\right)d\alpha'ds'.
\end{align*}
Taking 3 derivatives yields
\begin{align*}
&\pa^3\left(\mathcal{G}_{11}[u+t g -s, a]-\mathcal{G}_{11}[u-s,a]\right)\\
&=\sum_{j=0}^3\frac{c_j}{4\pi}\inti\intpi \cos(\al') \pa^j\left(\log\left(\frac{A'[u+tg]}{A'[u]}\right)\right)\pa^{3-j}\p{F_\rho(s-s')u'_\alpha} d\alpha' ds'\\
&+t \sum_{j=0}^3 \frac{c_j}{4\pi}\inti\intpi \cos(\al') \pa^j\left(\log\left(\frac{A'[u+tg]}{A'[u]}\right)\right)\pa^{3-j}\p{F_\rho(s-s')g'_\alpha} d\al'ds'\\
&+t \sum_{j=0}^3 \frac{c_j}{4\pi}\inti\intpi \cos(\al') \pa^j \left(\log\left(A'[u]\right)\right)\pa^{3-j}\p{F_\rho(s-s')g'_\alpha} d\alpha'd s'.
\end{align*}
Then it is enough to prove that
\begin{align}\label{limite1}
&\lim_{t\to0}\left|\left|\inti\intpi \cos(\alpha')\pa^j\left(\left(\frac{1}{t}\log\left(\frac{A'[u+tg]}{A'[u]}\right)-\frac{\pa_u A'[g]}{\A}\right)\right)\pa^{3-j}\p{F_\rho(s-s')u'_\alpha} d\al'ds'\right|\right|_{L^2}=0,\\
&\lim_{t\to 0}\left|\left|\inti\intpi \cos(\alpha')\pa^j\left(\log\left(\frac{A'[u+tg]}{A'[u]}\right)\right)\pa^{3-j}\p{F_\rho(s-s')u'_\alpha} d\al'ds'\right|\right|_{L^2}=0\label{limite2}
\end{align}
for all $j$.
These equalities \eqref{limite1} and \eqref{limite2} are a consequence of  lemmas \ref{maindu1} and \ref{maindu1bis}.

For $\mathcal{G}_3$ we have that
\begin{align*}
\mathcal{G}_3[u+tg-s,a]-\mathcal{G}_3[u-s,a]=\frac{1}{4\pi}\inti\intpi F_{\rho}(s-s')\frac{1}{b}\log\p{\frac{\Ag}{\A}}\sin(\alpha')d\alpha'ds'
\end{align*}
and then we have to show that
\begin{align*}
\lim_{t\to 0}\ab{\ab{\inti\intpi \pa^{3-j}F_{\rho}(s-s')\frac{1}{b}\pa^j\p{\frac{1}{t}\log\p{\frac{\Ag}{\A}}-\frac{\pa_u A'[g]}{\A}}\sin(\alpha')d\alpha'ds'}}_{L^2}=0.
\end{align*}
for $j=0,1,2,3$. This is also a consequence of lemmas \ref{maindu2} and \ref{maindu2bis}.
\end{proof}

The proof of the the continuity in $H^3$ follows similar steps to those in lemma \ref{daG}.

\begin{lemma}\label{dudaG}

The derivative $\pa_{u} \pa_{a}\mathcal{G}[u-s,a]g$ is given by:

\begin{align}
\label{deriv_ar_G}
\pa_{u} \pa_{a}\mathcal{G}[u-s,a]g & =
a^2\frac{d\lambda}{da}(a)g u_\alpha  + \frac{d\lambda}{da}(a)(1+a^2u)g_\alpha +2a\lambda(a)g u_\alpha +2a \lambda(a)u g_\alpha \nonumber\\
&+ \frac{ag}{2\pi} \inti \int_{-\pi}^\pi F_\rho(s-s')\log(A')\cos(\al')u'_{\al} d\al' ds' \nonumber\\
&+ \frac{au}{2\pi} \inti \int_{-\pi}^\pi F_\rho(s-s')\frac{(\pa_{u}A')[g]}{A'}\cos(\al')u'_{\al} d\al' ds' \nonumber\\
&+ \frac{au}{2\pi} \inti \int_{-\pi}^\pi F_\rho(s-s')\log(A')\cos(\al')g'_{\al} d\al' ds' \nonumber\\
&+ \frac{a^2g}{4\pi} \inti \int_{-\pi}^\pi F_\rho(s-s')\frac{\pa_a A'}{A'}\cos(\al')u'_{\al} d\al' ds' \nonumber\\
&+ \frac{1+a^2u}{4\pi} \inti \int_{-\pi}^\pi F_\rho(s-s')\left(\frac{(\pa_{u}\pa_a)A'[g]}{A'}-\frac{\pa_a A(\pa_{u}A)[g]}{A'^{2}}\right)\cos(\al')u'_{\al} d\al' ds' \nonumber\\
&+ \frac{1+a^2u}{4\pi} \inti \int_{-\pi}^\pi F_\rho(s-s')\frac{\pa_a A'}{A'}\cos(\al')g'_{\al} d\al' ds' \nonumber\\
&- \frac{g_{\al}}{4\pi} \inti \int_{-\pi}^\pi F_\rho(s-s')\frac{\pa_aA'}{A'}\cos(\al')(1+a^2u') d\al' ds' \nonumber\\
&- \frac{u_{\al}}{4\pi} \inti \int_{-\pi}^\pi F_\rho(s-s')\left(\frac{(\pa_{u}\pa_a A')[g]}{A'}-\frac{\pa_aA'(\pa_{u}A')[g]}{A'^{2}}\right)\cos(\al')(1+a^2u') d\al' ds' \nonumber\\
&- \frac{u_{\al}}{4\pi} \inti \int_{-\pi}^\pi F_\rho(s-s')\frac{\pa_aA'}{A'}\cos(\al')a^2g' d\al' ds' \nonumber\\
&- \frac{ag_{\al}}{2\pi} \inti \int_{-\pi}^\pi F_\rho(s-s')\log(A')\cos(\al')u' d\al' ds' \nonumber\\
&- \frac{au_{\al}}{2\pi} \inti \int_{-\pi}^\pi F_\rho(s-s')\frac{(\pa_{u}A')[g]}{A'}\cos(\al')u' d\al' ds' \nonumber\\
&- \frac{au_{\al}}{2\pi} \inti \int_{-\pi}^\pi F_\rho(s-s')\log(A')\cos(\al')g' d\al' ds' \nonumber\\
&-\frac{1}{4\pi}\inti \intpi F_\rho(s-s')\pa_a\left(\frac{1}{a^2}\frac{(\pa_uA')[g]}{A'}\right)\sin(\al')d\al'ds'\nonumber\\
&+\frac{1}{4\pi}\inti\intpi  F_\rho(s-s')\left(\frac{(\pa_u\pa_a A')[g]}{A'}-\frac{\pa_a A' \pa_uA'[g]}{A'^2}\right)\sin(\al')(u+u')d\al'ds'\nonumber\\
&+\frac{1}{4\pi}\inti\intpi  F_\rho(s-s')\frac{\pa_a A'}{A'}\sin(\al')(g+g')d\al' ds'\nonumber\\
&+\frac{a^2}{4\pi}\inti\intpi  F_\rho(s-s')\left(\frac{(\pa_u\pa_a A')[g]}{A'}-\frac{\pa_a A' (\pa_u)A'[g]}{A'^2}\right)\sin(\al')(uu'+u_\alpha u'_\al)d\al' ds'\nonumber\\
&+\frac{a^2}{4\pi}\inti\intpi  F_\rho(s-s')\frac{\pa_a A'}{A'}\sin(\al')(gu'+g'u+g_\alpha u'_\alpha+g'_\alpha u)d\al'ds'\nonumber\\
&+\frac{a}{2\pi}\inti\intpi  F_\rho(s-s')\log(A')\sin(\al')(ug'+gu'+u_\alpha g'_\alpha+ u'_\alpha g_\alpha)d\alpha' ds'\nonumber\\
&+\frac{a}{2\pi}\inti\intpi  F_\rho(s-s')\frac{\pa_u A'[g]}{A'}\sin(\al')(uu'+u_\alpha u'_\alpha)d\alpha'd s'.
\end{align}
Moreover this derivative is continuous in $H^3$.
\end{lemma}
\begin{proof} We will give the detail of the differentiation with respect to $u$ of the terms
\begin{align*}
\pa_a \mathcal{G}_1[u-s,a]=\inti \int_{-\pi}^\pi F_\rho(s-s')\frac{\pa_a A'}{A'}\cos(\al')u'_{\al} d\al' ds'
\end{align*}
and

\begin{align*}
\pa_a \mathcal{G}_2[u-s,a]=\int_{-\infty}^\infty \int_{-\pi}^\pi F_\rho(s-s')\pa_a\p{\frac{1}{a^2}\log\left(\frac{A'}{A'[a=0]}\right)}\sin(\alpha') d\al' ds'.
\end{align*}

We will split $\pa_a \mathcal{G}_2$ into two terms $\pa_a\mathcal{G}_2=H_1 +H_2$ with
\begin{align*}
H_1=\inti\intpi F_\rho(s-s')\p{-\frac{2}{a^3}\log\p{\frac{A'[a]}{A'[0]}}+\frac{2}{a}(u+u')}\sin{\al'}d\alpha' ds'
\end{align*}

and
\begin{align*}
H_2=\inti\intpi F_\rho(s-s')\p{\frac{1}{a^2}\frac{\pa_aA'[a]}{A'[a]}-\frac{2}{a}(u+u')}\sin{\al'}d\alpha' ds'.
\end{align*}

Next we differentiate $H_2$ with respect to $u$
\begin{align*}
&\frac{H_2[u+tg]-H_2[u]}{t}\\
&=\inti\intpi F_\rho(s-s')\p{\frac{1}{a^2t}\frac{\pa_a \Ag}{\Ag}-\frac{1}{a^2t}\frac{\pa_a \A}{\A}-\frac{2}{a}(g+g')}\sin(\al') d\alpha'ds'
\end{align*}
In order to prove the lemma we need to show that
\begin{align*}
I\equiv &\inti\intpi F_\rho(s-s')\p{\frac{1}{a^2t}\frac{\pa_a \Ag}{\Ag}-\frac{1}{a^2t}\frac{\pa_a \A}{\A}-\frac{2}{a}(g+g')}\sin(\al') d\alpha'ds'\\
 &-\inti\intpi F_\rho(s-s')\p{\frac{1}{a^2}\frac{\pa_u\pa_a A'[g]}{A'[u]}-\frac{1}{a^2}\frac{\pa_a A'[g]\pa_uA'[g]}{\A}-\frac{2}{a}(g+g')}\sin(\al') d\alpha'ds'
\end{align*}
tends to 0 in $H^3$. In order to check this fact we write
\begin{align*}
I=&\inti\intpi F_\rho(s-s')\frac{1}{a^2}\p{\frac{1}{t}\frac{\pa_a \Ag-\pa \A}{\Ag}-\frac{\pa_u\pa_aA'[g]}{\A}}\sin(\alpha')ds'd\al'\\
&+\inti\intpi F_\rho(s-s')\frac{1}{a^2}\p{\frac{\pa_a\A}{t}\p{\frac{1}{\Ag}-\frac{1}{\A}}+\frac{\pa_a A'[g]\pa_u A'[g]}{\A^2}}\sin(\al')d\alpha'ds'\\
&\equiv H_{21}+ H_{22}
\end{align*}
Then
\begin{align*}
H_{21}=&\inti\intpi F_\rho(s-s')\frac{1}{a^2}\p{\frac{1}{t}\frac{\pa_a \Ag-\pa_a \A-t\pa_u\pa_a A'[g]}{\Ag}}\sin(\alpha')d\alpha ds'\\
&\inti\intpi F_\rho(s-s')\frac{1}{a^2}\frac{\pa_u\pa_a A'[g]}{\Ag \A}\p{\A-\Ag}\sin(\alpha')d\al'ds'.
\end{align*}
Then we can apply lemma \ref{paradifH21andH22} in order to prove that $H_{21}\to 0$ when $t\to 0$ in $H^3$. To show an analog result for $H_{22}$ we write
\begin{align*}
H_{22}=&\inti\intpi F_\rho(s-s') \frac{1}{a^2}\frac{\pa_a A'[g]}{tA'[u]\Ag}\p{\A-\Ag+t\pa_u A'[g]}\sin(\alpha')d\alpha' ds'\\
&+\inti\intpi F_\rho(s-s') \frac{1}{a^2}\frac{\pa_a A'[g]\pa_u A'[g]}{\Ag \A^2}\p{\Ag-\A}\sin(\alpha') d\al' ds'.
\end{align*}
and apply lemma \ref{paradifH21andH22}.
\end{proof}

In order to differentiate $H_1$ we proceed as follows

\begin{align*}
&\frac{H_1[u+tg]-H_1[u]}{t}=\inti\intpi F_s(s-s')\p{-\frac{2}{a^3}\frac{1}{t}\log\p{\frac{\Ag}{\A}}+\frac{2}{a}(g+g')}\sin(\alpha')d\al' ds'
\end{align*}
And we would like to show that
\begin{align*}
\inti\intpi F_s(s-s')\p{-\frac{2}{a^3}\frac{1}{t}\log\p{\frac{\Ag}{\A}}+\frac{2}{a^3}\frac{\pa_u \A}{\A}}\sin(\alpha')d\al' ds'
\end{align*}
tends to 0, as $t$ goes to 0 in $H^3$. In order to do it we apply lemma \ref{paradifH1}.

Finally we differentiate $\pa_a G_1[u-s,a]$ with respect to $u$. Let us call $J[u]\equiv \pa_a G_1[u-s,a]$. We have that
\begin{align*}
\frac{J[u+tg]-J[u]}{t}=&\frac{1}{t}\inti\intpi F_s(s-s')\frac{\pa_a \Ag}{\Ag}\cos(\al')(u'_\alpha+tg'_\alpha)d\alpha'ds'\\
&-\frac{1}{t}\inti\intpi F_s[s-s']\frac{\Ag}{\A}\cos(\alpha')d\alpha' ds'\\
&=\inti\intpi F_s(s-s')\frac{1}{t}\left(\frac{\pa_a\Ag}{\Ag}-\frac{\pa_a \A}{\A}\right)\cos(\alpha')(u'_\alpha+tg'_\alpha)d\alpha'ds'\\
&+\inti\intpi F_s(s-s')\frac{\pa_a \Ag}{\Ag}\cos(\alpha')g'_\alpha d\alpha'ds'.
\end{align*}

We need to show that
\begin{align*}
&\frac{J[u+tg]-J[u]}{t}-\inti \intpi F_{s}(s-s')\p{\frac{\pa_u\pa_a \A}{\A}-\frac{\pa_a A'\pa_u \A}{\A^2}}\cos(\alpha') u'_\alpha d\alpha' ds'\\
&-\inti\intpi F_s(s-s')\frac{\pa_a \A}{\A}\cos(\alpha') g'_\alpha d\alpha' ds',
\end{align*}
tends to 0, when $t$ goes to 0 in $H^3$. This means that
\begin{align*}
&\inti\intpi F_s(s-s')\p{\frac{1}{t}\p{\frac{\pa_a \Ag}{\Ag}-\frac{\pa\A}{\A}}-\frac{\pa_u\pa_a \A}{\A}+\frac{\pa_a \A \pa_u \A}{\A^2}}\cos(\alpha')u'_\alpha d\alpha' ds'\\
&\inti\intpi F_s(s-s')\p{\frac{\pa_a\Ag}{\Ag}-\frac{\pa_a \A}{\A}}\cos(\al')g'_\alpha d\alpha' ds'
\end{align*}
tends to 0 when $t$ goes to 0 in $H^3$. Thus, it is enough to prove that, in $H^3$,
\begin{align}\label{guayaba}
\lim_{t\to 0} \inti\intpi F_s(s-s') \p{\frac{1}{t}\p{\frac{\pa_a \Ag}{\Ag}-\frac{\pa\A}{\A}}-\frac{\pa_u\pa_a \A}{\A}+\frac{\pa_a \A \pa_u \A}{\A^2}}\cos(\alpha')u'_\alpha d\alpha' ds'=0.
\end{align}
In order to prove \eqref{guayaba} we will consider two terms separately
\begin{align*}
J_1=&\inti\intpi F_s(s-s')\p{\frac{1}{t}\frac{\pa_a \Ag-\pa_a \A}{\Ag}-\frac{\pa_u\pa_a\A}{\A}}\cos(\alpha')u'_\alpha d\alpha'ds'
\end{align*}
and
\begin{align*}
J_2=\inti\intpi F_s(s-s') \pa_a A'[u]\p{\frac{1}{t}\p{\frac{1}{\Ag}-\frac{1}{\A}}+\frac{\pa_u \A}{\A^2}}
\end{align*}
In addition we will split $J_1$ into two terms
\begin{align*}
J_1=&\inti\intpi F_s(s-s')\frac{1}{t}\p{\frac{\pa_a \Ag-\pa_a \A -t \pa_u\pa_a \A}{A[u+tg]}}\cos(\alpha')u'_\alpha d\alpha'ds'\\
&+ \inti\intpi F_s(s-s')\pa_u\pa_a \Ag \p{\frac{1}{\Ag}-\frac{1}{\A}}\cos(\alpha')u'_\alpha d\alpha'ds'\\
\equiv &J_{11}+J_{12}.
\end{align*}
The term $J_{11}$ and $J_{12}$ can be controlled by using \ref{paradifelultimo}.

In order to deal with $J_2$ we also split it into two parts
\begin{align*}
J_2=&\inti\intpi f_s(s-s')\pa_a \A \frac{1}{t}\frac{\A-\Ag+t\pa_u \A}{\Ag\A}\cos(\alpha')u'_\alpha d\alpha'ds'\\
&+\inti\intpi F_s(s-s')\frac{\pa_a\A\pa_u \A}{\A} \p{\frac{1}{\Ag}-\frac{1}{\A}}\cos(\alpha')u'_\alpha d\alpha' ds'\\
\equiv & J_{21}+J_{22}.
\end{align*}
We control the terms $J_{21}$ and $J_{22}$ by using the lemma \ref{paradifelultimo}.

The proof of the the continuity in $H^3$ follows similar steps to those in lemma \ref{daG}.

This concludes the proof of lemma \ref{dudaG}.

\subsection{Step 3. Analysis of the linear part.}
\label{subsectionlineareuler}
\subsubsection{One dimensionality of the Kernel of the linear operator.}

We will start computing a nontrivial element of the kernel of $\pa_{\tilde{r}} \mathcal{G}[0,0]g = 0$, which is given by:

\begin{align*}
 &  \pa_{\tilde{r}} \mathcal{G}[0,0]g = \lambda(0) g_{\al}(\al,s) \\
& + \frac{1}{4\pi} \int_{-\pi}^{\pi}\int_{-1}^{1} F_\rho(s') \log(2-2\cos(\al-\al'))\cos(\al-\al')g_{\al}(\al',s')d\al'ds' \\
& - \frac{g_{\al}(\al,s)}{4\pi} \int_{-\pi}^{\pi}\int_{-1}^{1} F_\rho(s') \log(2-2\cos(\al-\al'))\cos(\al-\al')d\al'ds' \\
& + \frac{1}{4\pi} \int_{-\pi}^{\pi}\int_{-1}^{1} F_\rho(s') \sin(\al-\al')g(\al',s')d\al'ds' \\
& + \frac{1}{4\pi} \int_{-\pi}^{\pi}\int_{-1}^{1} F_\rho(s') \sin(\al-\al')\log(2-2\cos(\al-\al'))g(\al',s')d\al'ds' \\
\end{align*}

We make the decomposition

\begin{align*}
g(\al,s) = \sum_{n=1}^{\infty} g_{n}(s) \sin(n\al),
\end{align*}

and project onto the $m$-th mode. For each $m$ we get the following system of equations:

%\todo{faltan tildes?}

\begin{align}
0 & = - m g_{m}(s)\lambda(0)\sin(m\al) \nonumber\\
& - \frac{m}{4\pi} \int_{-\pi}^{\pi}\int_{-1}^{1} F_\rho(s') \log(2-2\cos(\al-\al'))\cos(\al-\al')g_{m}(s')\sin(m\al') d\al'ds' \nonumber\\
& + \frac{m g_{m}(s)\sin(m\al)}{4\pi} \int_{-\pi}^{\pi}\int_{-1}^{1} F_\rho(s') \log(2-2\cos(\al-\al'))\cos(\al-\al')d\al'ds' \nonumber\\
& + \frac{1}{4\pi} \int_{-\pi}^{\pi}\int_{-1}^{1} F_\rho(s') \sin(\al-\al')g_{m}(s')\cos(m\al')d\al'ds' \nonumber\\
& + \frac{1}{4\pi} \int_{-\pi}^{\pi}\int_{-1}^{1} F_\rho(s') g_{m}(s')\sin(\al-\al')\log(2-2\cos(\al-\al'))g_{m}(s')\cos(m\al')d\al'ds' \nonumber\\
& = - m g_{m}(s)\lambda(0)\sin(m\al) \nonumber\\
& - \frac{m}{4\pi} \int_{-1}^{1} F_\rho(s') g_{m}(s') \int_{-\pi}^{\pi}\log(2-2\cos(\al-\al'))\cos(\al-\al')\sin(m\al') d\al'ds' \nonumber\\
& + \frac{m g_{m}(s)\sin(m\al)}{4\pi} \int_{-1}^{1} F_\rho(s') \int_{-\pi}^{\pi}\log(2-2\cos(\al-\al'))\cos(\al-\al')d\al'ds' \nonumber\\
& + \frac{1}{4\pi} \int_{-1}^{1}F_\rho(s') g_{m}(s')\int_{-\pi}^{\pi}\sin(\al-\al')\cos(m\al')d\al'ds' \nonumber\\
& + \frac{1}{4\pi}\int_{-1}^{1} F_\rho(s') g_{m}(s')\int_{-\pi}^{\pi}\sin(\al-\al')\log(2-2\cos(\al-\al'))\cos(m\al')d\al'ds' \nonumber\\
& = - m g_{m}(s)\lambda(0)  - \frac{m}{4\pi} (-2\pi)\frac{m}{m^{2}-1}\int_{-1}^{1} F_\rho(s') g_{m}(s') ds'  \nonumber\\
& + \frac{m g_{m}(s)}{4\pi} (-2\pi)\int_{-1}^{1} F_\rho(s')ds'  + \frac{1}{4\pi}(-2\pi)\frac{1}{m^{2}-1}\int_{-1}^{1} F_\rho(s') g_{m}(s')ds' \nonumber\\
& = -m g_{m}(s)\lambda(0)  + \frac{1}{2}\int_{-1}^{1} F_\rho(s') g_{m}(s') ds'  - \frac{m g_{m}(s)}{2}\int_{-1}^{1} F_\rho(s')ds'. \label{kerneleuler}
\end{align}

We will use now the following lemma:

\begin{lemma}
Let $m \in \N$ and $F_\rho(s)$ be such that $\int_{-1}^{1}F_\rho(s)ds = -1$. Then, for $k \in \N$, the equation

\begin{align*}
g_{k}(s) & = -\frac{m}{k}\int_{-1}^{1} F_\rho(s') g_{k}(s') ds'
\end{align*}

has solutions:

\begin{align*}
g_k(s) = \left\{
\begin{array}{cc}
c & \text{ if } k = m \\
0 & \text{ if } k \neq m \\
\end{array}
\right.
\end{align*}

where $c$ is any real number.
\end{lemma}

\begin{proof}
Since the RHS does not depend on $s$, the only possible solutions are the constant ones. Setting $g_{k}(s) = c$ and substituting into the equation yields the result.
\end{proof}

Rearranging terms and using the fact that $\lambda(0) = \frac{m-1}{2m}$ and that $F_\rho$ has integral equal to -1, we are left to solve

\begin{align*}
g_{k}(s) & = -\frac{m}{k}\int_{-1}^{1} F_\rho(s') g_{k}(s') ds'. \\
\end{align*}

By the lemma, there exists a nontrivial solution only whenever $k = m$, and in that case, it is given by a constant function $c$. This shows the existence of a nontrivial kernel of dimension 1.

%
%Rearranging terms and using the fact that $\lambda(0) = \frac{m-1}{2m}$ and that $F$ has integral equal to -1, the last equation is equivalent to solve
%
%\begin{align*}
%g_{m}(\rho) & = -\int_{-1}^{1} F(\rho') g_{m}(\rho') d\rho', \\
%\end{align*}
%
%which can be easily seen to have only the constant functions as solution.
%
%\todo{Escribir como un lema?}
%
%
%We now show that there are no other linearly independent solutions. We have just seen it for the $m$-th mode, we now show that if $g_{k}(\rho)$ solves the equation \eqref{kerneleuler} for the $k$-th mode and $k \neq m$, then $g_{k}(\rho) = 0$. By doing the same calculations as before, we obtain that $g_{k}(\rho)$ has to solve:
%
%\begin{align*}
%0 & = -k g_{k}(\rho)\frac{m-1}{2m} + \frac{1}{2}\int_{-1}^{1} F(\rho') g_{k}(\rho') d\rho' + \frac{k g_{k}(\rho)}{2}. \\
%\end{align*}
%
%which is equivalent to
%
%
%\begin{align*}
%g_{k}(\rho) & = -\frac{m}{k}\int_{-1}^{1} F(\rho') g_{k}(\rho') d\rho'. \\
%\end{align*}
%
%Again, if there is a solution it can only be be a constant $c$. However, if $c \neq 0$, then it does not satisfy the equation and therefore $g_{k}(\rho) = 0$ is the only solution for $k \neq m$.

%ESCRIBIR K = 1? En realidad no hace falta porque $k$ sera siempre mayor o igual que $m$.

\subsubsection{Codimension of the image of the linear operator.}

We now characterize the image of $\pa_{\tilde{r}} \mathcal{G}[0,0]$. We have the following Lemma:

\begin{lemma}
Let

\begin{align*}
Z = \left\{q(s,\alpha) \in H^{3,3}_{m,odd}, q(s,\alpha) = \sum_{k=1}^{\infty}q_{k}(s)\sin(k\alpha), \int_{-1}^{1}F_\rho(s')q_{m}(s') ds' = 0\right\}.
\end{align*}

Then $Z = \text{Im}\left(\pa_{\tilde{r}} \mathcal{G}[0,0]\right)$.

\end{lemma}

\begin{proof}

We start proving that $\text{Im}\left(\pa_{\tilde{r}} \mathcal{G}[0,0]\right) \subset Z$. This follows easily from \eqref{kerneleuler} since the $m-$th mode of $\text{Im}\left(\pa_{\tilde{r}} \mathcal{G}[0,0]\right)g$ is given by

\begin{align*}
\frac{g_{m}(s)}{2}  + \frac{1}{2}\int_{-1}^{1} F_\rho(s') g_{m}(s') ds',
\end{align*}

yielding zero after multiplying by $F(s)$ and integrating over $[-1,1]$.

We now prove the other implication and show that $Z \subset \text{Im}\left(\pa_{\tilde{r}} \mathcal{G}[0,0]\right)$.

Let $q(s,\alpha) \in Z$. We want to show that there exists a $g(s,\alpha) \in H^{4,3}_{m,even}$ such that $\pa_{\tilde{r}} \mathcal{G}[0,0]g = q$. It is enough to show that $g(s,\alpha) \in H^{4,3}$ since the $m$-fold property and the evenness follow trivially. Let us project $g$ into Fourier modes as

\begin{align*}
g(s,\alpha) = \sum_{k=m, k | m}^{\infty}g_{k}(s)\cos(k\alpha), \quad q(s,\alpha) = \sum_{k=m, k | m}^{\infty}q_{k}(s)\sin(k\alpha).
\end{align*}

This yields the following system of equations:

\begin{align}
\label{imagen}
q_{k}(s) & = k g_{k}(s)\frac{1-m}{2m}  + \frac{1}{2}\int_{-1}^{1} F_\rho(s') g_{k}(s') ds' + \frac{k g_{k}(s)}{2},
\end{align}

which implies that $g_{k}(s)$ has to be of the form $g_{k}(s) = A_{k}q_{k}(s) + B_{k}$ for some constants $A_k,B_k$. We distinguish two cases.

If $k = m$, then equation \eqref{imagen} yields:

\begin{align*}
q_{m}(s) & = g_{m}(s)\frac{1}{2}  + \frac{1}{2}\int_{-1}^{1} F_\rho(s') g_{m}(s') ds' \\
& = \frac{A_{m}}{2}q_{m}(s) + \frac12\int_{-1}^{1} F_\rho(s') q_{m}(s'),
\end{align*}

and the only solution is $A_{m} = 2$, $B_m \in \RR$, under the additional constraint

\begin{align*}
\int_{-1}^{1} F_\rho(s') q_{m}(s') = 0,
\end{align*}

and no solution otherwise. If $k \neq m$, then we have to solve

\begin{align*}
q_{k}(s) & = q_{k}(s)\frac{kA_{k}}{2m} + \frac{kB_{k}}{2m} + \frac{A_{k}}{2}\int_{-1}^{1} F_\rho(s') q_{k}(s') ds' - \frac{B_{k}}{2},
\end{align*}

which has a unique solution

\begin{align*}
A_{k} & = \frac{2m}{k} , \quad B_{k} = \frac{2m^{2}}{k(m-k)}\int_{-1}^{1} F_\rho(s') q_{k}(s') ds',
\end{align*}

for any $q_{k}(s)$. This shows the existence and the uniqueness. We discuss the regularity now. To do so, we compute the $H^{4,3}$ norm of $g(s,\al)$:

\begin{align*}
\|g\|_{H^{4,3}}^{2} & \lesssim \sum_{k=1}^{\infty}\|g_{k}\|_{H^{3}}^{2}(1+k^{4})^{2}
\leq \sum_{k=1}^{\infty}\|q_{k}\|_{H^{3}}^{2}\frac{4m^2}{k^2}(1+k^{4})^{2} + C\sum_{k=1, k \neq m}^{\infty} \frac{m^{4}}{k^2(m-k)^2}\|q_{k}\|_{L^{2}}^{2}(1+k^{4})^{2} \\
& \lesssim \sum_{k=1}^{\infty}\|q_{k}\|_{H^{3}}^{2}(1+k^{3})^{2} \lesssim \|q\|_{H^{3,3}}^{2}.
\end{align*}

This concludes that $Z = \text{Im}(\pa_{\tilde{r}} \mathcal{G}[0,0])$ and in particular shows that the codimension of the image of $\pa_{\tilde{r}} \mathcal{G}[0,0]$ is 1, as we needed.

\end{proof}

\subsection{Step 4. The transversality property 4.}

This step is devoted to show the transversality condition. We start writing out the calculations since everything is explicit, including the characterization of the image done in the previous subsection. Based on that, we only need to compute the $m$-th Fourier coefficient of $\pa_{\rti} \pa_{a}\mathcal{G}[\rti,a]g$.

By setting $a = 0, \tilde{r} = 0$ in \eqref{deriv_ar_G}, one obtains:

\begin{align*}
\pa_{\rti} \pa_{a}\mathcal{G}[\rti,a]g|_{\rti=0,a=0} & = \frac{d\lambda}{da}(0)g_\alpha \\
\end{align*}

We now take $g(\al,s) = \cos(m\al)$ and extract the contribution to the $m$-th mode. We get that it is equal to

\begin{align*}
-m\frac{d\lambda}{da}(0)\sin(m\al)
\end{align*}

Testing the coefficient in front of $\sin(m\al)$ against $F_\rho(s)$ and integrating in $[-1,1]$ we obtain

\begin{align*}
m\frac{d\lambda}{da}(0),
\end{align*}

which is different than 0 as long as $\frac{d\lambda}{da}(0) \neq 0$, thus satisfying the transversality condition.

\section*{Acknowledgements}

AC, DC and JGS were partially supported by the grant MTM2014-59488-P (Spain) and ICMAT Severo Ochoa project SEV-2015-556. AC was partially supported by the Ram\'on y Cajal program RyC-2013-14317 and ERC grant 307179-GFTIPFD. JGS was partially supported by an AMS-Simons Travel Grant. Part of this work was done while JGS was visiting ICMAT, to which he is grateful for its support. We thank Jacob Bedrossian and Vlad Vicol for helpful conversations.

\appendix

\section{Basic integrals}
\label{appendixbasicintegrals}

In this section we will outline the basic integral identities used in the spectral analysis of subsection \ref{subsectionlineareuler}.

The first three lemmas follow easily from \cite[Lemma A.3]{Castro-Cordoba-GomezSerrano:analytic-vstates-ellipses}

\begin{lemma}
\label{lemmaA1}
\begin{align*}
\int_{-\pi}^{\pi} \log(2-2\cos(\al'))\cos(\al')d\al' = -2\pi
\end{align*}
\end{lemma}

\begin{lemma}
Let $m$ be a positive integer. Then:
\begin{align*}
\int_{-\pi}^{\pi} \log(2-2\cos(\al'))\cos(\al')\cos(m\al')d\al' =
\left\{
\begin{array}{cc}
(-2\pi)\frac{m}{m^2-1} & \text{ if } m \neq 1 \\
-\frac{\pi}{2} & \text{ if } m = 1 \\
\end{array}
\right.
\end{align*}
\end{lemma}

\begin{lemma}
Let $m$ be a positive integer. Then:
\begin{align*}
\int_{-\pi}^{\pi}\log(2-2\cos(\al'))\sin(\al')\sin(m\al')d\al' =
\left\{
\begin{array}{cc}
(-2\pi)\frac{1}{m^2-1} & \text{ if } m \neq 1 \\
\frac{\pi}{2} & \text{ if } m = 1 \\
\end{array}
\right.
\end{align*}
\end{lemma}

The next two integrals are obvious but we state them here for convenience:

\begin{lemma}
Let $m$ be a positive integer. Then:
\begin{align*}
\int_{-\pi}^{\pi}\cos(\al')\cos(m\al') d\al' =
\left\{
\begin{array}{cc}
0 & \text{ if } m \neq 1 \\
\pi & \text{ if } m = 1
\end{array}
\right.
\end{align*}
\end{lemma}

\begin{lemma}
Let $m$ be a positive integer. Then:
\begin{align*}
\int_{-\pi}^{\pi}\sin(\al')\sin(m\al') d\al' =
\left\{
\begin{array}{cc}
0 & \text{ if } m \neq 1 \\
\pi & \text{ if } m = 1
\end{array}
\right.
\end{align*}
\end{lemma}

%The last two lemmas follow from \cite[Lemma A.1]{Castro-Cordoba-GomezSerrano:analytic-vstates-ellipses}.
%
%\begin{lemma}
%Let $m$ be a positive integer. Then:
%\begin{align*}
%\int_{-\pi}^{\pi}\frac{\sin(\al')\sin(m\al')}{2-2\cos(\al')} d\al' = \pi
%%\left\{
%%\begin{array}{cc}
%%0 & \text{ if } m \neq 1 \\
%%\pi & \text{ if } m = 1
%%\end{array}
%%\right.
%\end{align*}
%\end{lemma}
%
%
%\begin{lemma}
%\label{lemmaA7}
%Let $m$ be a positive integer. Then:
%\begin{align*}
%\int_{-\pi}^{\pi}\frac{\sin(\al')\cos(\al')\sin(m\al')}{2-2\cos(\al')} d\al' =
%\left\{
%\begin{array}{cc}
%\pi & \text{ if } m \neq 1 \\
%\frac{\pi}{2} & \text{ if } m = 1
%\end{array}
%\right.
%\end{align*}
%\end{lemma}

\section{Estimates on $A'$}\label{Aestimations}

In this section we will assume the following:
\begin{enumerate}
\item The function $\rti=u-s\in V^\delta$, i.e. $||\rti||_{H^{4,3}}\leq \delta$ .
\item $C(\delta)$ will denote a finite  positive constant that just depends on $\delta$, that increases with $\delta$ and such that $C(0)=0.$ Also $c(\delta)$ will be a constant that just depends on $\delta$ and such that $c(\delta)\geq \overline{c}>0$ for all $\delta\in [0,\delta_0]$ and $\delta_0$ small enough.
\item The parameter $a\in [-a_0,a_0],$ for some  small enough $a_0$ and $\delta \in [0,\delta_0]$ for some small enough $\delta_0$.
\item $C$ and $c$ will be constants such that $C<\infty$ and $c>0$.
\item $\partial$ means either $\pa_\alpha$ or $\pa_s$.
\end{enumerate}

We also recall that given any function $f(\al,s)$ we will write
\begin{align*}
f=&f(\alpha,s)\\
f'=&f(\alpha-\alpha',s-s').
\end{align*}

We want to obtain estimates on $\mathcal{G}[u-s,a]$ in \eqref{G}. In order to do it we will set $b=a^2$.

\begin{lemma}\label{A1} The function $$A'=4\sinc{\alpha'}+4b(u+u')\sinc{\alpha'}+b^2\left(4uu'\sinc{\alpha'}+(u-u')^2\right)$$ satisfies
\begin{enumerate}
\item $$A'\geq c\left(\sinc{\alpha'}+b^2 s'^2\right),$$
\item $$\log(A')\leq \log(\sinc{\al'}+b^2s'^2)+C,$$
\item $$\pa A'\leq  Cb\left(\sinc{\alpha'}+b s'^2\right),$$
\item $$\pa^2 A'\leq Cb \left(\sinc{\al'}+b\left(\left|\sin\left(\frac{\al'}{2}\right)\right|+|s'|+s'^2\right)\right)$$
\end{enumerate}
\end{lemma}
\begin{proof}
Using that $||\rti||_{H^{4,3}}\leq \delta$ and the Sobolev embedding we have that

\begin{align*}
|A'|\geq4\sinc{\alpha'}-b C(\delta)\sinc{\alpha'}-b^2C(\delta)\sinc{\alpha'}+b^2(u-u')^2.
\end{align*}
In addition,
\begin{align*}
&(u-u')^2=(s'+\rti-\rti')^2 \geq c s'^2-C(\rti-\rti')^2\\&\geq  cs'^2-C(\delta)(\sinc{\alpha'}+s'^2)\geq c(\delta)s'^2-C(\delta)\sinc{\alpha'}.
\end{align*}
Then
\begin{align*}
|A'|\geq c\left(\sinc{\alpha'}+b^2s'^2\right).
\end{align*}

In order to show the second statement we notice that
\begin{align*}
|\log(A')|&\leq |\log(B)| \quad \text{if}\quad 0\leq B \leq A'\leq 1\\
|\log(A')|&\leq |\log(B)|\quad \text{if} \quad 1\leq A'\leq B
\end{align*}
Since \begin{align*}
|A'|&\leq C\\
|A'|&\geq c\left(\sinc{\al'}+b^2s'^2\right)
\end{align*}
we have that

\begin{align*}
|\log(A')|\leq \log\left(\sinc{\al'}^2+b^2s'^2\right)+ C.
\end{align*}
Just a computation shows that
$$\pa A'= 4 b (\pa u +\pa u')\sinc{\alpha'}+b^2\left(4\left(u'\pa u+ u\pa u'\right)\sinc{\alpha'}+ 2(u-u')(\pa u -\pa u')\right).$$

Taking modulus we have that

\begin{align*}
|\pa A'|&\leq C(\delta)b \sinc{\alpha'}+C(\delta)a^2\sinc{\alpha'}+ C(\delta) b^2 (\sinc{\alpha'}+s'^2)\\
&\leq Cb\left(\sinc{\alpha'}+b s'^2\right).
\end{align*}
Taking two derivatives on $A$ yields
\begin{align*}
\pa^2 A'= & 4 b \pa^2(u+u')\sinc{\alpha'}\\ & +b^2\left(4\pa^2(uu')\sinc{\al'}+ 2(u-u')\pa^2(u-u')+2(\pa u -\pa u')^2\right)
\end{align*}
Taking modulus we obtain
\begin{align*}
|\pa^2 A'|\leq C(\delta) b \sinc{\alpha'}+C(\delta)b^2\sinc{\al'}+C(\delta)b^2\left(\left|\sin\left(\frac{\al'}{2}\right)\right|+|s'|\right)+C(\delta)b^2\left(\sinc{\alpha'}+s'^2\right).
\end{align*}

\end{proof}

\begin{lemma}\label{A2}The following estimate holds for $b>0$,
$$
\int_{-\pi}^\pi \int_{-1}^1 \frac{\left|\sin\left(\frac{\al}{2}\right)\right|^m |s|^k}{\left(\sinc{\alpha}+b^2s^2\right)^l}d\alpha ds \leq
\left\{\begin{array}{cc} C b^{m+1-2l} & m<2l-1\\ C \log\left(\frac{1}{b}\right) & m=2l-1 \\ C  & m > 2l-1\end{array}\right.,$$
with $l\geq 1$, $k,m\geq0$ and $k+m+1-2l\geq 0$.
\end{lemma}

\begin{proof}
First we notice that
$$
I\equiv \int_{-\pi}^\pi \int_{-1}^1 \frac{\left|\sin\left(\frac{\al}{2}\right)\right|^m |s|^k}{\left(\sinc{\alpha}+b^2s^2\right)^l}d\alpha ds =
C\int_{0}^{\frac{\pi}{2}}\int_{0}^1\frac{\sin(\al)^m s^k}{(\sin^2(\al)+b^2s^2)^l}d\al ds $$ and using that $c\alpha\leq \sin(\al)\leq \alpha$ in $[0,\pi/2]$,  we have that
$$I\leq C\int_{0}^\frac{\pi}{2}\int_{0}^1\frac{\al^m s^k}{(\al^2+b^2s^2)^l}d\al ds=\frac{C}{b^{k+1}}\int_{0}^\frac{\pi}{2}\int_{0}^b\frac{\al^m s^k}{(\al^2+s^2)^l}d\al ds.$$
Making the change of variable \begin{align*}
s=& r\cos(\theta)\\
\al=& r\sin(\theta)
\end{align*}
yields
\begin{align*}
&I\leq \frac{C}{b^{k+1}} \left( \int_{0}^{\arctan\left(\frac{\pi}{2b}\right)}\int_{0}^\frac{b}{\cos(\theta)}+\int_{\arctan\left(\frac{\pi}{2b}\right)}^\frac{\pi}{2}\int_{0}^\frac{\pi}{2\sin(\theta)}
\right)r^{k+m+1-2l} \sin^m(\theta)\cos^k(\theta)dr d\theta.
\end{align*}
On one hand
\begin{align*}
&\int_{0}^{\arctan\left(\frac{\pi}{2b}\right)}\int_{0}^\frac{b}{\cos(\theta)}r^{k+m+1-2l} \sin^m(\theta)\cos^k(\theta)drd\theta
\\ & =\frac{b^{k+m+2-2l}}{k+m+2-2l}\int_{0}^{\arctan\left(\frac{\pi}{2b}\right)}\tan^m(\theta)\cos^{2l-2}(\theta) d\theta =Cb^{k+m+2-2l}\int_{0}^{\arctan\left(\frac{\pi}{2b}\right)}
\frac{\tan^m(\theta)}{(1+\tan^2(\theta))^{l-1}}d\theta\\
&=Cb^{k+m+2-2l}\int_{0}^\frac{1}{b}\frac{y^m}{(1+y^2)^{l-1}(1+y^2)} dy\leq b^{k+m+2-2l}\left( 1+\int_{1}^\frac{\pi}{2b}y^{m-2l} dy\right).
\end{align*}
Thus
\begin{align*}
&\int_{0}^{\arctan\left(\frac{\pi}{2b}\right)}\int_{0}^\frac{b}{\cos(\theta)}r^{k+m+1-2l} \sin^m(\theta)\cos^k(\theta)drd\theta\\&
\leq \left\{\begin{array}{cc}C b^{k+m+2-2l}(1+\log\left(\frac{\pi}{2b})\right) & \text{ if $m=2l-1$}\\
b^{k+m+2-2l}\left(1+\frac{1}{m-2l+1}\left(\left(\frac{\pi}{2b}\right)^{m-2l+1}-1\right)\right) & \text{if $m\neq 2l-1$}
\end{array}\right.
\\ &\leq \left\{\begin{array}{ccc}C b^{k+1}\log\left(\frac{\pi}{2b})\right) & \text{ if $m=2l-1$}\\
C b^{k+1} & \text{if $m> 2l-1$}\\
C b^{k+m+2-2l} &\text{if $m<2l-1$}.
\end{array}\right.
\end{align*}
On the other hand
\begin{align*}
&\int_{\arctan\left(\frac{\pi}{2b}\right)}^\frac{\pi}{2}\int_{0}^\frac{\pi}{2\sin(\theta)}
r^{k+m+1-2l} \sin^m(\theta)\cos^k(\theta)dr d\theta\\
&=C \int_{\arctan\left(\frac{\pi}{2b}\right)}^\frac{\pi}{2} \cot^k(\theta)\sin^{2l-2}(\theta)d\theta\leq C \int_{\arctan\left(\frac{\pi}{2a}\right)}^\frac{\pi}{2} \cot^k(\theta) d\theta\\
&=C\int_{\frac{\pi}{2b}}^\infty \frac{1}{y^k}\frac{1}{1+y^2}dy \leq C \int_{\frac{\pi}{2b}}^\infty y^{-k-2}dy=Cb^{k+1}.
\end{align*}
\end{proof}

\begin{lemma}\label{A3}The following estimate holds for $b>0$
\begin{align*}
\int_{-\pi}^\pi\int_{-1}^1\left|\log\left(\sinc{\alpha'}+b^2s'^2\right)\right|d\alpha'ds'\leq C
\end{align*}
\end{lemma}
\begin{proof} We just have to notice that
\begin{align*}
\left|\log(\sinc{\alpha'}+b^2s'^2)\right|\leq \left| \log\left(\sinc{\al'}\right)\right|+C.
\end{align*}

\end{proof}

Next we will use lemmas \ref{A1}, \ref{A2} and \ref{A3} to prove the following
\begin{lemma}\label{A4}Let $f(\alpha,s)\in L^2(\Omega)$ with supp$(f)\subset \T\times [-1,1]$ and $b>0$. Then
\begin{enumerate}
\item \label{A4log} $$\left|\left|\inti\intpi \log(A')f'd\al'ds'\right|\right|_{L^2}\leq C ||f||_{L^2}.$$
\item \label{A4dA} $$\left|\left|\inti\intpi \frac{\pa A'}{A'} f'd\al'ds'\right|\right|_{L^2}\leq C b||f||_{L^2}.$$
\item \label{A4d2A} $$\left|\left|\inti\intpi \left(\frac{\pa^2 A'}{A'}-\frac{(\pa A')^2}{A^2}\right) f'd\al'ds'\right|\right|_{L^2}\leq Cb ||f||_{L^2}.$$
\item \label{A4d3A} $$\left|\left|\inti\intpi \left(\pa^3 \log(A')-\frac{\pa^3A'}{A'}\right) f'd\al'ds'\right|\right|_{L^2}\leq C b||f||_{L^2}.$$
\end{enumerate}
where we remark  that $C$ does not depend on $b$.
\end{lemma}
\begin{proof}
To prove \ref{A4log} we first use Lemma \ref{A1}--2, then Minkowski inequality and finally Lemma \ref{A3} to obtain that
\begin{align*}
\left|\left|\inti\intpi \log(A')f'd\al'ds'\right|\right|_{L^2}\leq C ||f||_{L^2}\int_{-\pi}^\pi\int_{-2}^2\left(|\log(\sinc{\al'}+b^2 s'^2)|+C\right)d\alpha' ds' \leq C.
\end{align*}
To prove \ref{A4dA} we first use Lemmas \ref{A1}--1 and \ref{A1}--3, then Minkowski inequality and finally Lemma \ref{A2} to get
\begin{align*}\left|\left|\inti\intpi \frac{\pa A'}{A'} f'd\al'ds'\right|\right|_{L^2}\leq C||f||_{L^2}\int_{
\pi}^\pi\int_{-2}^2 \frac{b\left(\sinc{\alpha'}+b s'^2\right)}{\sinc{\al'}+b^2s'^2}d\al' ds'\leq C||f||_{L^2}.\end{align*}
To prove \ref{A4d2A} we first use Lemmas \ref{A1}--1, \ref{A1}--3 and \ref{A1}--4 to obtain that
\begin{align}
&\left|\frac{\pa^2A'}{A'}\right|\leq Cb\frac{\sinc{\al'}+b\left(\left|\sin\left(\frac{\alpha'}{2}\right)\right|+|s'|+s^2\right)}{\sinc{\al'}+b^2s'^2}\label{pa1entrea1}\\
&\left|\frac{\left(\pa A'\right)^2}{A'^2}\right|\leq C b^2\frac{\sin^4\left(\frac{\al'}{2}\right)+b^2 s'^4}{(\sinc{\alpha'}+b^2 s'^2)^2}\label{pa1entrea2}.
\end{align}
Then we can use Minkowski inequality and Lemma \ref{A2} to yield \ref{A4d2A}.

Finally we prove \ref{A4d3A}. First we notice that

\begin{align*}
\pa^3\log(A')-\frac{\pa^3A'}{A'}=-3\frac{\pa A' \pa^2A'}{A'^2}-\frac{(\pa A')^3}{A'^3}
\end{align*}
and that by lemmas \ref{A1}--1, \ref{A1}--3, \ref{A1}--4, we have that
\begin{align*}
&\left|\frac{\pa A' \pa^2 A'}{A'^2}\right|\\&\leq Cb^2\frac{\sin^4\left(\frac{\alpha'}{2}\right)+b\left|\sin\left(\frac{\alpha'}{2}\right)\right|^3+b
\sinc{\al'}(|s'|+s'^2)+b\sinc{\al'}s'^2+b^2s'^2\left|\sin\left(\frac{\alpha'}{2}\right)\right|+b^2s'^2(|s'|+s'^2)}{(\sinc{\al'}+b^2s'^2)^2}\\
+&\left|\frac{(\pa A')^3}{A'^3}\right|\leq Cb^3\frac{\sin^6\left(\frac{\al'}{2}\right)+b^3s'^6}{(\sinc{\al'}+b^2s'^2)^3}.
\end{align*}
Thus we can apply Minkowski inequality and lemma \ref{A2} to prove \ref{A4d3A}.
\end{proof}

\begin{lemma}\label{A5}Let $f(\alpha,s)\in L^\infty(\Omega)$ with supp$(f)\subset \T\times [-1,1]$ and $b>0$. Then $$\left|\left|\inti\intpi \frac{\pa^3 A'}{A'}f'd\al'ds'\right|\right|_{L^2}\leq C b||f||_{L^\infty}.$$
\end{lemma}
\begin{proof}
Taking three derivatives on $A'$ yields
\begin{align*}
\pa^3 A'= &4b\pa^3(u+u')\sinc{\al'}+\\
&+b^2\left(4\pa^3(uu')\sinc{\al'}+6\pa^2(u-u')\pa(u-u')+(u-u')(\pa^3u-\pa^3u')\right)
\end{align*}
Then
\begin{align}
\left|\frac{\pa^3A'}{A'}\right|&\leq Cb\frac{\sinc{\alpha'}}{\sinc{\alpha'}+b^2s'^2}(\pa^3 u +\pa^3 u')+Cb^2\frac{\sinc{\al'}}{\sinc{\al'}+b^2s'^2}\pa^3(uu')\nonumber\\
&+Cb^2\frac{\left|\sin\left(\frac{\alpha'}{2}\right)\right|+|s'|}{\sinc{\al'}+b^2 s'^2}+Cb^2 \frac{\left|\sin\left(\frac{\alpha'}{2}\right)\right|+|s'|}{\sinc{\alpha'}+b^2s'^2}|\pa^3u-\pa^3u'|. \label{quote1}
\end{align}
Then by applying lemma \ref{A2} we achieve the conclusion on the lemma.
\end{proof}

\begin{lemma} \label{l2log1}The following estimate holds:
\begin{align*}
\left|\log\left(\frac{A'}{A'[0]}\right)\right|\leq \left|\log\left(\frac{(4-Cb)\sinc{\al'}+c b^2s'^2}{4\sinc{\al'}}\right)\right|
+\left|\log\left(\frac{(4+Cb+Cb^2)\sinc{\alpha'}+Cb^2s'^2}{4\sinc{\al'}}\right)\right|
\end{align*}
\end{lemma}
\begin{proof}
We notice as in Lemma \ref{A1}--2 that
\begin{align*}
\left|\log\left(\frac{A'}{A'[0]}\right)\right|\leq \left|\log\left(\frac{(4-C(\delta)b-C(\delta)b^2)\sinc{\al'}+c b^2s'^2}{4\sinc{\al'}}\right)\right|
+\left|\log\left(\frac{(4+Cb+Cb^2)\sinc{\alpha'}+Cb^2s'^2}{4\sinc{\al'}}\right)\right|
\end{align*}
since \begin{align*}
|A'|\geq (1-C(\delta)b-C(\delta)b^2)\sinc{\al'}+c b^2s'^2
\end{align*}  as we checked in the proof of lemma \ref{A1}--1 and
\begin{align*}
A'\leq C\left((1+b+b^2)\sinc{\alpha'}+b^2s'^2\right).
\end{align*}

\end{proof}

\begin{lemma}\label{l2log2}The following inequalities hold for $b>0$.
\begin{align*}
\intpi\int_{-2}^2 \left|\log\left(\frac{(1-b)\sinc{\al'}+b^2s'^2}{\sinc{al'}}\right)\right|d\al ds'\leq C b\\
\intpi\int_{-2}^2 \left|\log\left(\frac{(1+b)\sinc{\al'}+b^2s'^2}{\sinc{al'}}\right)\right|d\al ds'\leq C b
\end{align*}
\end{lemma}
\begin{proof}
We notice that
\begin{align*}
\left|\log\left(1-b+\frac{b^2s'^2}{\sinc{\al'}}\right)\right|\leq |\log(1-b)|+\left|\log\left(1+\frac{b^2s'^2}{\sinc{\al'}}\right)\right|.
\end{align*}
Then the estimate follows from
\begin{align*}
\intpi \log\left(1+\frac{b^2}{\sinc{\al'}}\right) d\al'=4\pi \arcsinh(b).
\end{align*}
In addition, the second inequality follows from
$$\intpi\log\left(1+b+\frac{b^2}{\sinc{\al'}}\right)d\alpha'=4\pi\log\left(b+\sqrt{1+b+b^2}\right).$$
\end{proof}
\begin{lemma}\label{l2logfinal}Let $f(\al,s)\in L^2$ with $\supp(f)\in \T\times [-1,1].$ Then
\begin{align*}
\left|\left|\inti\intpi \frac{1}{b}\log\left(\frac{A'}{A'[0]}\right)f'd\alpha' ds'\right|\right|_{L^2}\leq C ||f||_{L^2}.
\end{align*}
\end{lemma}
\begin{proof}
This lemma is a consequence of Lemmas \ref{l2log1} and \ref{l2log2} and Minkowski inequality.
\end{proof}

The following lemmas will be used to show the differentiability of the functional $\mathcal{G}$ and the continuity of its derivative.

Here we recall that
\begin{align*}
\pa_a A'= 8a(u+u')\sinc{\al'}+ 4a^3\left(uu'\sinc{\al'}+(u-u')^2\right)
\end{align*}

\begin{lemma}\label{dA1}The following estimates hold:
\begin{enumerate}
\item $$\left|\pa_a A'\right|\leq C |a| \left((1+a^2)\sinc{\al'}+a^2 s'^2\right).$$
\item $$\left|\pa\pa_a A'\right|\leq C |a| \left((1+a^2)\sinc{\al'}+a^2 s'^2\right).$$
\item $$\left|\pa^2\pa_a A'\right|\leq C |a| \left((1+a^2)\sinc{\al'}+ a^2\left|\sin\left(\frac{\al'}{2}\right)\right|+a^2\left(|s'|+s'^2\right)\right).$$
\end{enumerate}
\end{lemma}
\begin{proof}
The proof of this lemma follows similar steps as the proof of Lemma \ref{A1}.
\end{proof}

\begin{lemma}\label{dA4}Let $f(\alpha,s)\in L^2(\Omega)$ with supp$(f)\subset \T\times [-1,1]$ and $b>0$. Then
\begin{enumerate}
\item \label{dA4log} $$\left|\left|\inti\intpi \frac{\pa_a A'}{A'}f'd\al'ds'\right|\right|_{L^2}\leq C |a|||f||_{L^2}.$$
\item \label{dA4dA} $$\left|\left|\inti\intpi \pa\frac{\pa_a A'}{A'} f'd\al'ds'\right|\right|_{L^2}\leq C |a|||f||_{L^2}.$$
\item \label{dA4d2A} $$\left|\left|\inti\intpi \pa^2\frac{\pa_a A'}{A'} f'd\al'ds'\right|\right|_{L^2}\leq C|a| ||f||_{L^2}.$$
\item \label{dA4d3A} $$\left|\left|\inti\intpi \left(\pa^3 \left(\frac{\pa_aA'}{A'}\right)-\frac{\pa^3\pa_a A'}{A'}+ \frac{\pa_{a}A' \pa^{3}A'}{A'^{2}}\right) f'd\al'ds'\right|\right|_{L^2}\leq C |a|||f||_{L^2}.$$
\end{enumerate}
\end{lemma}
\begin{proof}
In order to prove 1 we notice that by Lemmas \ref{A1}-1 and \ref{dA1}-1
\begin{align}\label{quote2}
\left|\frac{\pa_a A'}{A'}\right|\leq C |a|\frac{\sinc{\al'}+a^2s'^2}{\sinc{\al'}+a^4s'^2}\leq C |a| +|a|^3\frac{s'^2}{\sinc{\al'}+a^4s'^2}
\end{align}
Therefore we obtain the first inequality in lemma \ref{dA1} by applying Minkowski's inequality and Lemma \ref{A2}.

To obtain 2 we proceed as follows
\begin{align*}
\pa\frac{\pa_a A'}{A'}=\frac{\pa \pa_a A'}{A'}-\frac{\pa_aA' \pa A'}{A'^2}.
\end{align*}
Then, by applying Lemmas \ref{A1}-1, \ref{A1}-3, \ref{dA1}-1 and \ref{dA1}-2, we have that
\begin{align*}
\left|\pa\frac{\pa_a A'}{A'}\right|\leq C|a|\frac{\sinc{\al'}+a^2 s'^2}{\sinc{\al'}+a^4s'^2}+C|a|^3\frac{\sin^4\left(\frac{\al'}{2}\right)+a^4s'^4}{\left(\sinc{\al'}+a^4s'^2\right)^2}
\end{align*}
Then we can use again Minkowski's inequality and Lemma \ref{A2} to yield 2.

Taking one more derivative yields
\begin{align*}
\pa^2 \frac{\pa_a A'}{A'}= \frac{\pa^2\pa_a A'}{A'}-2\frac{\pa\pa_a A'\pa A'}{A'^2}-\frac{\pa_aA'\pa^2A'}{A'^2}+2\frac{\pa_aA' (\pa A')^2}{A'^3}.
\end{align*}
By using Lemmas \ref{A1}-1, \ref{A1}-3 and \ref{dA1}-3 we can estimate
\begin{align*}
\left|\frac{\pa\pa_a A'\pa A'}{A'^2}\right|\leq C |a|^3\frac{\sin^4\left(\frac{\al'}{2}\right)+a^4s'^4}{(\sinc{\al'}+a^4 s'^2)^2}.
\end{align*}
By using Lemmas \ref{A1}-1, \ref{dA1}-1 and \ref{A1}-3 we can estimate
\begin{align*}
\left|\frac{\pa_a A'(\pa A')^2}{A'^3}\right|\leq C |a|^5\frac{(\sinc{\al'}+a^2s'^2)^3}{(\sinc{\al'}+a^4 s'^2)^3}\leq C |a|^7\frac{\sin^6(\al')+a^6s'^6}{(\sinc{\al'}+a^4 s'^2)^3}.
\end{align*}
By using Lemmas \ref{A1}-1 and \ref{dA1}-3 we can estimate
\begin{align*}
\left|\frac{\pa^2\pa_a A'}{A'}\right|\leq C |a| \frac{\left(\sinc{\al'}+ a^2\left|\sin\left(\frac{\al'}{2}\right)\right|+a^2\left(|s'|+s'^2\right)\right)}{\sinc{\al'}+a^4s'^2}.
\end{align*}
Finally, by using Lemmas \ref{A1}-1, \ref{A1}-3 and \ref{dA1}-1 we have that
\begin{align*}
&\left|\frac{\pa_a \pa^2 A'}{A'^2}\right|\leq  C|a|^3\frac{(\sinc{\al'}+a^2 s'^2)\left(\sinc{\al'}+ a^2\left|\sin\left(\frac{\al'}{2}\right)\right|+a^2\left(|s'|+s'^2\right)\right)}{(\sinc{\al'}+a^4s'^2)^2}\\
&\leq C|a|^3 \frac{\sin^4\left(\frac{\alpha'}{2}\right)+a^2 \left|\sin\left(\frac{\alpha'}{2}\right)\right|^3+a^2 \sinc{\alpha'}(|s'|+s'^2)+a^2\sinc{\al'}s'^2+ a^4 \left|\sin\left(\frac{\alpha'}{2}\right)\right|s'^2+a^4 s'^2(|s'|+s'^2)}{(\sinc{\alpha'}+a^4 s'^2)^2}.
\end{align*}
Then by applying Minkowski's inequality and Lemma \ref{A2} we achieve 3.

In order to prove 4, we start taking 3 derivatives:

\begin{align*}
&\pa^{3}\frac{\pa_{a}A'}{A'} - \frac{\pa^{3}\pa_{a}A'}{A'} + \frac{\pa_{a}A' \pa^{3}A'}{A'^{2}}
 = - 3 \frac{\pa^{2}\pa_{a}A' \pa A'}{A'^{2}} - 3 \frac{\pa \pa_{a} A'\pa^{2} A'}{A'^{2}}\\&
 + 6 \frac{\pa \pa_{a} A' (\pa A')^{2}}{A'^{3}} + 6 \frac{\pa_{a} A' \pa^{2} A' \pa A'}{A'^{3}}  - 6 \frac{\pa_{a} A' (\pa A')^{3}}{A'^{4}}
\end{align*}

By using Lemmas \ref{A1}-1, \ref{A1}-3 and \ref{dA1}-3 we can estimate

\begin{align*}
&\left|\frac{\pa^{2} \pa_{a} A' \pa A'}{A'^{2}}\right| \\
 &\leq C|a|^{3} \frac{\left(\sinc{\al'} + a^2s'^{2}\right)\left(\sinc{\al'} + a^{2}\left|\sin\left(\frac{\al'}{2}\right)\right|+a^{2}\left(|s'|+s'^{2}\right)\right)}{(\sinc{\al'}+a^{4}s'^{2})^{2}} \\
 &\leq C|a|^{3} \frac{\sin^{4}\left(\frac{\al'}{2}\right) + a^{2}\left|\sin\left(\frac{\al'}{2}\right)\right|^{3} + a^{2}\sinc{\al'}(|s'|+s'^{2}) + a^{2}s'^{2}\sinc{\al'} + a^{4}s'^{2}\left|\sin\left(\frac{\al'}{2}\right)\right| + a^{4}s'^{2}(|s'|+s'^{2})}{(\sinc{\al'}+a^{4}s'^{2})^{2}} \\
\end{align*}

By using Lemmas \ref{A1}-1, \ref{A1}-4 and \ref{dA1}-2 we can estimate

\begin{align*}
\left|\frac{\pa \pa_{a} A' \pa^{2} A'}{A'^{2}}\right| & \leq C|a|^{3} \frac{\left(\sinc{\al'} + a^2s'^{2}\right)\left(\sinc{\al'} + a^{2}\left|\sin\left(\frac{\al'}{2}\right)\right|+a^{2}\left(|s'|+s'^{2}\right)\right)}{(\sinc{\al'}+a^{4}s'^{2})^{2}},
\end{align*}

and this term is done as the previous one.

By using Lemmas \ref{A1}-1, \ref{A1}-3 and \ref{dA1}-2 we can estimate

\begin{align*}
\left|\frac{\pa \pa_{a} A' (\pa A')^{2}}{A'^{3}}\right| & \leq C|a|^{5} \frac{\left(\sinc{\al'} + a^2s'^{2}\right)^{3}}{(\sinc{\al'}+a^{4}s'^{2})^{3}} \\
& \leq C|a|^{5}\frac{\left(\sin^{6}\left(\frac{\al'}{2}\right) + a^{6}s'^{6}\right)}{(\sinc{\al'}+a^{4}s'^{2})^{3}}.
\end{align*}

By using Lemmas \ref{A1}-1, \ref{A4}-3, \ref{A1}-4 and \ref{dA1}-1 we can estimate

\begin{align*}
&\left|\frac{\pa_{a} A' \pa^{2} A' \pa A'}{A'^{3}}\right| \\
 &\leq C|a|^{5} \frac{\left(\sinc{\al'} + a^2s'^{2}\right)^{2}\left(\sinc{\al'} + a^{2}\left|\sin\left(\frac{\al'}{2}\right)\right|+a^{2}\left(|s'|+s'^{2}\right)\right)}{(\sinc{\al'}+a^{4}s'^{2})^{3}} \\
 &\leq C|a|^{5} \frac{\left(\sin^{4}\left(\frac{\al'}{2}\right) + a^4s'^{4}\right)\left(\sinc{\al'} + a^{2}\left|\sin\left(\frac{\al'}{2}\right)\right|+a^{2}\left(|s'|+s'^{2}\right)\right)}{(\sinc{\al'}+a^{4}s'^{2})^{3}} \\
 &\leq C|a|^{3} \frac{\sin^{6}\left(\frac{\al'}{2}\right) + a^{2}\left|\sin\left(\frac{\al'}{2}\right)\right|^{5} + a^{2}\sin^{4}\left(\frac{\al'}{2}\right)(|s'|+s'^{2}) + a^{4}s'^{4}\sinc{\al'} + a^{6}s'^{4}\left|\sin\left(\frac{\al'}{2}\right)\right| + a^{6}s'^{4}(|s'|+s'^{2})}{(\sinc{\al'}+a^{4}s'^{2})^{3}} \\
\end{align*}

By using Lemmas \ref{A1}-1, \ref{A4}-3 and \ref{dA1}-1 we can estimate

\begin{align*}
\left|\frac{\pa_{a} A' (\pa A')^{3}}{A'^{4}}\right| & \leq C|a|^{7} \frac{\left(\sinc{\al'} + a^2s'^{2}\right)^{4}}{(\sinc{\al'}+a^{4}s'^{2})^{4}} \\
& \leq C|a|^{7}\frac{\left(\sin^{8}\left(\frac{\al'}{2}\right) + a^{8}s'^{8}\right)}{(\sinc{\al'}+a^{4}s'^{2})^{4}}.
\end{align*}

Finally, by applying Minkowski's inequality and Lemma \ref{A2} we conclude 4.

\end{proof}

\begin{lemma}\label{dA5}Let $f(\alpha,s)\in L^\infty(\Omega)$ with supp$(f)\subset \T\times [-1,1]$ and $b>0$. Then:

$$\left|\left|\inti\intpi \frac{\pa^3 \pa_{a} A'}{A'}f'd\al'ds'\right|\right|_{L^2}\leq C |a| ||f||_{L^\infty}.$$

and

$$\left|\left|\inti\intpi \frac{\pa^3 A' \pa_{a} A'}{A'^{2}}f'd\al'ds'\right|\right|_{L^2}\leq C |a| ||f||_{L^\infty}.$$
\end{lemma}

\begin{proof}
We start computing

\begin{align*}
\pa^{3} \pa_{a} A' = 8a \pa^{3}(u+u')\sinc{\al'} + 4a^{3}\left(\pa^{3}(uu')\sinc{\al'} + 6 \pa(u-u')\pa^{2}(u-u') + 2(u-u')\pa^{3}(u-u')\right).
\end{align*}

Then, by using Lemma \ref{A1}-1:

\begin{align*}
\frac{|\pa^{3} \pa_{a} A'|}{|A|} & \leq C|a| \frac{\sinc{\al'}}{(\sinc{\al'}+a^{4}s'^{2})}|\pa^{3}(u+u')|
 + C|a|^{3} \frac{\sinc{\al'}}{(\sinc{\al'}+a^{4}s'^{2})}|\pa^{3}(uu')| \\
 & + C|a|^{3} \frac{\left|\sin\left(\frac{\al'}{2}\right)\right| + |s'|}{(\sinc{\al'}+a^{4}s'^{2})}
 + C|a|^{3} \frac{\left|\sin\left(\frac{\al'}{2}\right)\right| + |s'|}{(\sinc{\al'}+a^{4}s'^{2})}|\pa^{3} u - \pa^{3} u'|,
\end{align*}

and we can conclude as in Lemma \ref{A5}. This shows the first estimate. In order to prove the second one, we use \eqref{quote1} and \eqref{quote2}.
\end{proof}

\begin{lemma}\label{A4bis}Let $f(\alpha,s)\in L^2(\Omega)$ with supp$(f)\subset \T\times [-1,1]$ and $h\neq 0$. Then
\begin{enumerate}
\item \label{A4logbis} $$\left|\left|\inti\intpi \frac{1}{h} \log\left(\frac{A'[h]}{A'[0]}\right)f'd\al'ds'\right|\right|_{L^2}\leq C |h| ||f||_{L^2}.$$
\item \label{A4dAbis} $$\left|\left|\inti\intpi \frac{1}{h} \frac{\pa A'[h]}{A'[h]} f'd\al'ds'\right|\right|_{L^2}\leq C |h|||f||_{L^2}.$$
\item \label{A4d2Abis} $$\left|\left|\inti\intpi \frac{1}{h} \left(\frac{\pa^2 A'[h]}{A'[h]}-\frac{(\pa A'[h])^2}{A'^2[h]}\right) f'd\al'ds'\right|\right|_{L^2}\leq C|h| ||f||_{L^2}.$$
\item \label{A4d3Abis} $$\left|\left|\inti\intpi \frac{1}{h}\left(\pa^3 \log(A'[h])-\frac{\pa^3A'[h]}{A'[h]}\right) f'd\al'ds'\right|\right|_{L^2}\leq C |h|||f||_{L^2}.$$
\end{enumerate}
\end{lemma}
\begin{proof}
The inequalities \ref{A4dAbis}, \ref{A4d2Abis} and \ref{A4d3Abis} follow directly from lemma \ref{A4} by making $b=h^2$. Then it  remains to show \ref{A4logbis}. In order to do it we notice that
\begin{align*}
\frac{1}{h}\log\left(\frac{A'[h]}{A'[0]}\right)=h\int_{0}^1\frac{4(u+u')\sinc{\alpha'}+4h^2u u' \sinc{\alpha'}+ h^2(u-u')^2}
{4\sinc{\al}+4\mu h^2 (u+u')\sinc{\al'}+4\mu h^4 u u' \sinc{\al}+\mu h^4 (u-u')^2}d\mu.
\end{align*}
Similar computations that those ones to prove \ref{A1}-1 yield
\begin{align}\label{conmu}
4\sinc{\al}+4\mu h^2 (u+u')\sinc{\al'}+4\mu h^4 u u' \sinc{\al}+\mu h^4 (u-u')^2\geq c\left(\sinc{\alpha'}+\left(\sqrt{\mu}h^2\right)^2s'^2\right)
\end{align}
And we can estimate
\begin{align*}
\left|\frac{1}{h}\log\left(\frac{A'[h]}{A'[0]}\right)\right|\leq C h +Ch^3\int_{0}^1\frac{s'^2}{\sinc{\al'}+(\sqrt{\mu}h^2)^2s'^2} d\mu
\end{align*}
Then we achieve \ref{A4logbis} by using Minkowski's inequality and lemma \ref{A2}.
\end{proof}

\begin{lemma}\label{A5bis}Let $f(\alpha,s)\in L^\infty(\Omega)$ with supp$(f)\subset \T\times [-1,1]$ and $b>0$. Then $$\left|\left|\inti\intpi \frac{1}{h}\frac{\pa^3 A'}{A'}f'd\al'ds'\right|\right|_{L^2}\leq C |h|||f||_{L^\infty}.$$
\end{lemma}
\begin{proof} This lemma follows from lemma \ref{A5} by making again $b=h^2$.
\end{proof}

\begin{lemma}\label{otropopurri}The following estimates hold
\begin{enumerate}
\item
\begin{align*}
&\ab{\frac{1}{b}\log\p{\frac{A[a]}{A[0]}}-(u+u')}\leq C b\int_{0}^1\frac{\sinc{\al'}+s'^2}{\sinc{\alpha'}+b^2\mu s'^2}d\mu.
\end{align*}
\item \begin{align*}
\ab{\pa\frac{1}{b}\log\p{\frac{A[a]}{A[0]}}-\pa(u+u')}\leq C b \frac{\sinc{\alpha'}+s'^2}{\D}
\end{align*}
\item \begin{align*}
\ab{\pa^2\p{\frac{1}{b}\log\p{\frac{A[a]}{A[0]}}-(u+u')}}\leq & Cb \frac{\sinc{\alpha'}+s'^2+\ab{\sin\p{\frac{\alpha'}{2}}}+|s'|}{\D}\\
&+ Cb \frac{\p{\sinc{\alpha'}+bs'^2}^2}{\p{\D}^2}.
\end{align*}
\item \begin{align*}
&\ab{\pa^3\p{\frac{1}{b}\log\p{\frac{A[a]}{A[0]}}-(u+u')}-\p{\frac{1}{b}\frac{\pa^3A[a]}{A[a]}-\pa^3(u+u')}}\\
&\leq Cb \frac{\p{\sinc{\alpha'}+bs'^2}\p{\sinc{\alpha'}+b\p{|\sin\p{\frac{\alpha'}{2}}+s'^2+|s'|}}}{\p{\D}^2}\\
&+Cb^2\frac{\p{\sinc{\alpha'}+bs'^3}^3}{\p{\D}^3}.
\end{align*}
\item \begin{align*}
&\ab{\frac{1}{b}\frac{\pa^3A[a]}{A[a]}-\pa^3(u+u')}\leq Cb \frac{\sinc{\alpha'}+bs'^2}{\D}\ab{\pa^3(u+u')}\\
& C b \frac{\sinc{\alpha'}}{\D}\ab{\pa^3(uu')}+Cb \frac{\ab{\sin\p{\frac{\alpha'}{2}}}+|s'|}{\D}\ab{\pa^3(u-u')}
+Cb \frac{\ab{\sin\p{\frac{\alpha'}{2}}}+|s'|}{\D}.
\end{align*}
\end{enumerate}

\end{lemma}
\begin{proof}
\begin{align*}
&\frac{1}{b}\log\p{\frac{A[a]}{A[0]}}=\frac{1}{b}\int_{0}^1\frac{d}{d\mu}\log\p{4\sinc{\al'}+4b\mu(u+u')\sinc{\al'}+b^2\mu \p{4uu'\sinc{\al'}+(u-u')^2}}d\mu\\
&=\int_{0}^1 \frac{4(u+u')\sinc{\al'}+b\p{4uu'\sinc{\al'}+(u-u')^2}}{4\sinc{\al'}+4b\mu(u+u')\sinc{\al'}+b^2\mu\p{4uu'\sinc{\al'}+(u-u')^2}}d\mu.
\end{align*}
Then
\begin{align*}
&\frac{1}{b}\log\p{\frac{A[a]}{A[0]}}-(u+u')=b\int_{0}^1\frac{4uu'\sinc{\al'}+(u-u')^2+4\mu(u+u')^2\sinc{\al'}}{4\sinc{\al'}+4b\mu (u+u')\sinc{\alpha'}+b^2\mu \p{4uu'\sinc{\alpha'}+(u-u')^2}}d\mu\\
&+b^2\int_{0}^1\frac{\mu\p{u+u'}\p{4uu'\sinc{\alpha'}+(u-u')^2}}{4\sinc{\alpha'}+4b\mu(u+u')\sinc{\alpha'}+b^2\mu \p{4uu'\sinc{\alpha'}+(u-u')^2}}d\mu.
\end{align*}
Therefore
\begin{align*}
&\ab{\frac{1}{b}\log\p{\frac{A[a]}{A[0]}}-(u+u')}\leq C b\int_{0}^1\frac{\sinc{\al'}+s'^2}{\sinc{\alpha'}+b^2\mu s'^2}d\mu.
\end{align*}

Taking one derivative we have that
\begin{align*}
&\pa\p{\frac{1}{a^2}\log\p{\frac{A[a]}{A[0]}}-(u+u')}=\frac{1}{b}\frac{\pa A[a]}{A[a]}-\pa(u+u')\\
&=\frac{4\pa(u+u')\sinc{\al'}+b\p{4\pa(uu')\sinc{\al'}+2(u-u')\pa(u-u')}}{4\sinc{\alpha'}+4b(u+u')\sinc{\alpha'}+b^2\p{4uu'\sinc{\alpha'}+(u-u')^2}}-\pa(u+u')\\
&=\frac{b\p{4\pa(uu')\sinc{\alpha'}+2(u-u')\pa(u-u')}-\pa(u+u')\p{4b(u+u')\sinc{\alpha'}+b^2\p{4uu'\sinc{\al'}+(u-u')^2}}}{4\sinc{\alpha'}+4b(u+u')\sinc{\alpha'}+b^2\p{4uu'\sinc{\alpha'}+(u-u')^2}}
\end{align*}
Therefore
\begin{align*}
\ab{\pa\frac{1}{b}\log\p{\frac{A[a]}{A[0]}}-\pa(u+u')}\leq C b \frac{\sinc{\alpha'}+s'^2}{\D}
\end{align*}

Taking two derivatives we obtain
\begin{align*}
\pa^2\p{\frac{1}{b}\log\p{\frac{A[a]}{A[0]}}-(u+u')}=\frac{1}{b}\frac{\pa^{2} A[a]}{A[a]}-\frac{1}{b}\frac{\p{\pa A[a]}^2}{A[a]^2}-\pa^2(u+u')
\end{align*}
On one hand
\begin{align*}
&\frac{1}{b}\frac{\pa^2A[a]}{A[a]}-\pa^2 (u+u')\\
&=b\frac{4\pa^2(uu')\sinc{\alpha'}+2(u-u')\pa^2(u-u')+2(\pa(u-u'))^2}{A[a]}\\
&+\frac{4\pa^2(u+u')\sinc{\al'}}{4\sinc{\alpha'}+4b(u+u')\sinc{\alpha'}+b^2(4 uu'\sinc{\alpha'}+(u-u')^2)}-\pa^2(u-u')\\
&=b\frac{4\pa^2(uu')\sinc{\alpha'}+2(u-u')\pa^2(u-u')+2(\pa(u-u'))^2}{A[a]}\\
&-b\frac{\p{4(u+u')\sinc{\alpha'}+b(4uu'\sinc{\alpha'}+(u-u')^2)}\pa^2(u+u')}{A[a]}
\end{align*}
And we have that
\begin{align*}
\ab{\frac{1}{b}\frac{\pa^2A[a]}{A[a]}-\pa^2 (u+u')}\leq C b \frac{\sinc{\alpha'}+s'^2+\ab{\sin\p{\frac{\alpha'}{2}}}+|s'|}{\D}
\end{align*}
On the other hand
\begin{align*}
\frac{(\pa A[a])^2}{A[a]^2}=b\frac{\p{4\pa(u+u')\sinc{\alpha'}+b\p{4\pa(uu')\sinc{\alpha'}+2(u-u')\pa(u-u')}}^2}{A[a]^2}
\end{align*}
and then
\begin{align*}
\ab{\frac{(\pa A[a])^2}{A[a]^2}}\leq Cb \frac{\p{\sinc{\alpha'}+bs'^2}^2}{\p{\D}^2}.
\end{align*}
Taking three derivatives yields
\begin{align*}
\pa^3\p{\frac{1}{a^2}\log\p{\frac{A[a]}{A[0]}}-(u+u')}=\frac{1}{b}\frac{\pa^3A[a]}{A[a]}-\pa^3(u+u')-\frac{3}{b}\frac{\pa^2 A[a]\pa A[a]}{A[a]^2}+\frac{2}{b}\frac{\p{\pa A[a]}^3}{A[a]^3}.
\end{align*}

The inequalities \ref{A1}-3 and \ref{A1}-4 yield
\begin{align*}
\ab{\frac{1}{b}\frac{\pa A[a]\pa^2 A[a]}{A[a]^2}}\leq Cb \frac{\p{\sinc{\alpha'}+bs'^2}\p{\sinc{\alpha'}+b\p{|\sin\p{\frac{\alpha'}{2}}+s'^2+|s'|}}}{\p{\D}^2}
\end{align*}
and
\begin{align*}
\ab{\frac{1}{b}\frac{(\pa A[a])^3}{A[a]^3}}\leq Cb^2\frac{\p{\sinc{\alpha'}+bs'^3}^3}{\p{\D}^3}
\end{align*}
And \ref{otropopurri}-4 is then proven. To prove \ref{otropopurri}-5 we notice that
\begin{align*}
&\frac{1}{b}\frac{\pa^3 A[a]}{A[a]}-\pa^3(u+u')=\frac{4\pa^3(u+u')\sinc{\alpha'}}{A[a]}-\pa^3(u+u')\\
&+b\frac{4\pa^3(uu')\sinc{\alpha'}+2(u-u')\pa^3(u-u')}{A[a]}\\
&+b\frac{6\pa(u-u')\pa^2(u-u')}{A[a]}.
\end{align*}
And we can compute
\begin{align*}
\frac{4\pa^3(u+u')\sinc{\alpha'}}{A[a]}-\pa^3(u+u')=-b\frac{4(u+u')\sinc{\alpha'}+b\p{4uu'\sinc{\alpha'}+(u-u')^2}}{A[a]}\pa^3(u+u').
\end{align*}
Thus
\begin{align*}
\ab{\frac{4\pa^3(u+u')\sinc{\alpha'}}{A[a]}-\pa^3(u+u')}\leq Cb \frac{\sinc{\alpha'}+bs'^2}{\D}\ab{\pa^3(u+u')}.
\end{align*}
Therefore
\begin{align*}
&\ab{\frac{1}{b}\frac{\pa^3A[a]}{A[a]}-\pa^3(u+u')}\leq Cb \frac{\sinc{\alpha'}+bs'^2}{\D}\ab{\pa^3(u+u')}\\
& C b \frac{\sinc{\alpha'}}{\D}\ab{\pa^3(uu')}+Cb \frac{\ab{\sin\p{\frac{\alpha'}{2}}}+|s'|}{\D}\ab{\pa^3(u-u')}
+Cb \frac{\ab{\sin\p{\frac{\alpha'}{2}}}+|s'|}{\D}.
\end{align*}

\end{proof}

\begin{lemma}\label{B15} Let $g\in H^{4,3}$ with $||g||_{H^{4,3}}=1$ and  $f(\alpha,s)\in L^2(\Omega)$ with supp$(f)\subset \T\times [-1,1]$. Then the following estimates hold:

\begin{align*}
\ab{\ab{\inti\intpi \pa^j\p{\frac{1}{b}\log\p{\frac{A[a]}{A[0]}}-(u+u')} \ms f'd\al'ds' }}_{L^2}\leq Cb\log\p{\frac{1}{b}},
\end{align*}
for $j=0,1,2$.
\item \begin{align*}\ab{\ab{\inti\intpi \p{\pa^3\p{\frac{1}{b}\log\p{\frac{A[a]}{A[0]}}-(u+u')}-Z(s,s',\alpha,\alpha')}\ms f'd\alpha'ds'}}_{L^2}\leq Cb
\end{align*}
where
\begin{align*}
&Z(s,s',\al,\alpha')=\frac{1}{b}\frac{\pa^3A[a]}{A[a]}-\pa^3(u+u')
\end{align*}

\end{lemma}

\begin{proof}
The proof follows from Lemma \ref{otropopurri}, Lemma \ref{A2} and Minkowski's inequality.
\end{proof}

\begin{lemma} \label{B16} Let  $f(\alpha,s)\in L^\infty(\Omega)$ with supp$(f)\subset \T\times [-1,1]$. Then the following estimates hold:
\begin{align*}
\ab{\ab{\inti\intpi Z(\alpha,\alpha',s,s') \ms f'd\alpha'ds'}}_{L^2}\leq Cb\log\p{\frac{1}{b}}
\end{align*}
where
\begin{align*}
&Z(\alpha,\alpha',s,s')=\frac{1}{b}\frac{\pa^3A[a]}{A[a]}-\pa^3(u+u')
\end{align*}
\end{lemma}

\begin{proof}
The proof follows from Lemma \ref{otropopurri}, Lemma \ref{A2} and Minkowski's inequality.
\end{proof}

\begin{lemma} \label{B17}
The following estimates hold:
\begin{enumerate}
\item
\begin{align*}&\ab{\pa_a \p{\frac{1}{a^2}\log\p{\frac{A[a]}{A[0]}}}}\\
&\leq C a \int_{0}^1\frac{\mu \sinc{\alpha'}}{\sinc{\alpha'}+\p{\sqrt{\mu}b}^2s'^2}d\mu + C a^3\int_{0}^1 \frac{\mu s'^2}{\sinc{\alpha'}+\p{\sqrt{\mu}b}^2s'^2}d\mu\\
&+C a \frac{\sinc{\alpha'}}{\sinc{\alpha'}+b^2s'^2} + C a^3\frac{ s'^2}{\sinc{\alpha'}+b^2s'^2}\\
&+ C a \int_{0}^1\frac{\sinc{\alpha'}+s'^2}{\sinc{\alpha'}+(\sqrt{\mu}b)^2s'^2}d\mu\\
&+C a \frac{\sinc{\alpha'}+s'^2}{\sinc{\alpha'}+b^2s'^2}
\end{align*}
\item \begin{align*}
\ab{\pa\pa_a \p{\frac{1}{a^2}\log\p{\frac{A[a]}{A[0]}}}}\leq C a\frac{\sinc{\alpha'}+s'^2}{\D}+C a \frac{\p{\sinc{\alpha'}+bs'^2}^2}{\p{\D}^2}.
\end{align*}
\item
\begin{align*}
&\ab{\pa^2\pa_a \p{\frac{1}{a^2}\log\p{\frac{A[a]}{A[0]}}}}\leq C a \frac{\p{\sinc{\alpha'}+\ms+s'^2+|s'|}}{\D}\\
&+C a \frac{\p{\sinc{\alpha'}+bs'^2}^2}{\p{\D}^2}+ C a \frac{\p{\sinc{\alpha'}+b s'^2}\p{\sinc{\alpha'}+b\p{\ms+s'^2+|s'|}}}{\p{\D}^2}\\
&+C a^3\frac{\p{\sinc{\alpha'}+bs'^2}^3}{\p{\D}^3}.
\end{align*}
\item \begin{align*}
&\ab{\pa^3\pa_a \p{\frac{1}{a^2}\log\p{\frac{A[a]}{A[0]}}}+\frac{2}{a^2}\frac{\pa^3A[a]}{A[a]}-\frac{1}{a^2}\frac{\pa^3\pa_aA[a]}{A[a]}+\frac{1}{a^2}\frac{\pa A[a]\pa^3 A[a]}{A[a]^2}}\\
&\leq  C a\frac{\p{\sinc{\alpha'}+bs'^2}\p{\sinc{\alpha'}+b\p{\ms+s'^2+|s'|}}}{\p{\D}^2}\\
&+C a^3 \frac{\p{\sinc{\alpha'}+bs'^2}^3}{\p{\D}^3} +C a\frac{\p{\sinc{\alpha'}+bs'^2}^2\p{\sinc{\alpha'}+b\p{\ms+s'^2+|s'|}}}{\p{\D}^3}\\ &+ C a^5\frac{\p{\sinc{\alpha'}+bs'^2}^4}{\p{\D}^4}\\
\end{align*}
\item \begin{align*}
&\ab{\frac{2}{a^2}\frac{\pa^3A[a]}{A[a]}-\frac{1}{a^2}\frac{\pa^3\pa_aA[a]}{A[a]}+\frac{1}{a^2}\frac{\pa A[a]\pa^3 A[a]}{A[a]^2}}\leq  C a \frac{\sinc{\al'}}{\D}\ab{\pa^3(uu')}\\
&+C a \frac{\ms +|s'|}{\D}+Ca \frac{\ms+|s'|}{\D}|\pa^3(u-u')|.
\end{align*}
\end{enumerate}

\end{lemma}
\begin{proof}
First we notice that
\begin{align*}
\pa_a \p{\frac{1}{a^2}\log\p{\frac{A[a]}{A[0]}}}=-\frac{2}{a^3}\log\p{\frac{A[a]}{A[0]}}+\frac{1}{a^2}\frac{\pa_a A[a]}{A[a]}
\end{align*}
We have that
\begin{align*}
&-\frac{2}{a^3}\log\p{\frac{A[a]}{A[0]}}=-\frac{2}{a^3}\int^1_{0}\frac{4b (u+u')\sinc{\alpha'}+b^2\p{4uu'\sinc{\alpha'}+(u-u')^2}}{4\sin{\alpha'}+4b\mu (u+u')\sinc{\alpha'}+b^2\mu\p{4uu'\sinc{\alpha'}+(u-u')^2}}d\mu\\
&=\int_{0}^1\frac{-8\frac{1}{a} (u+u')\sinc{\alpha'}-2a\p{4uu'\sinc{\alpha'}+(u-u')^2}}{4\sinc{\alpha'}+4b\mu (u+u')\sinc{\alpha'}+b^2\mu\p{4uu'\sinc{\alpha'}+(u-u')^2}}d\mu.
\end{align*}
In addition
\begin{align*}
\frac{1}{a^2}\frac{\pa_a A[a]}{A[a]}=\frac{8\frac{1}{a}(u+u')\sinc{\alpha'}+4a\p{4uu'\sinc{\alpha'}+(u-u')^2}}{A[a]}
\end{align*}
Thus
\begin{align*}
&-\frac{2}{a^3}\log\p{\frac{A[a]}{A[0]}}+\frac{1}{a^2}\frac{\pa_a A[a]}{A[a]}\\
&=\frac{8}{a}\sinc{\alpha'}\int_0^1\p{\frac{1}{4\sinc{\alpha'}+4b\mu (u+u')\sinc{\alpha'}+b^2\mu\p{4uu'\sinc{\alpha'}+(u-u')^2}}-\frac{1}{A[a]}}d\mu\\
&+\int_{0}^1\frac{-2a\p{4uu'\sinc{\alpha'}+(u-u')^2}}{4\sinc{\alpha'}+4b\mu (u+u')\sinc{\alpha'}+b^2\mu\p{4uu'\sinc{\alpha'}+(u-u')^2}}d\mu\\
&+\frac{4a\p{4uu'\sinc{\alpha'}+(u-u')^2}}{A[a]}
\end{align*}
We will split the first term of the right hand side of the previous expression in the following way
\begin{align*}
&\frac{8}{a}\sinc{\alpha'}\int_0^1\p{\frac{1}{4\sinc{\alpha'}+4b\mu (u+u')\sinc{\alpha'}+b^2\mu\p{4uu'\sinc{\alpha'}+(u-u')^2}}-\frac{1}{A[a]}}d\mu\\
&=\frac{8}{a}\sinc{\alpha'}\int_0^1\p{\frac{1}{4\sinc{\alpha'}+4b\mu (u+u')\sinc{\alpha'}+b^2\mu\p{4uu'\sinc{\alpha'}+(u-u')^2}}-\frac{1}{4\sinc{\alpha'}}}d\mu\\
&+\frac{8}{a}\sinc{\alpha'}\p{\frac{1}{4\sinc{\alpha'}}-\frac{1}{A[a]}}.
\end{align*}
And we can estimate
\begin{align*}
&\ab{\frac{8}{a}\sinc{\alpha'}\int_0^1\p{\frac{1}{4\sinc{\alpha'}+4b\mu (u+u')\sinc{\alpha'}+b^2\mu\p{4uu'\sinc{\alpha'}+(u-u')^2}}-\frac{1}{4\sinc{\alpha'}}}d\mu}\\&
\leq 2\ab{\int_{0}^1\frac{4a\mu (u+u')\sinc{\alpha'}+a^3\mu\p{4uu'\sinc{\alpha'}+(u-u')^2}}{4\sinc{\alpha'}+4b\mu (u+u')\sinc{\alpha'}+b^2\mu\p{4uu'\sinc{\alpha'}+(u-u')^2}}d\mu}\\
&\leq C a \int_{0}^1\frac{\mu \sinc{\alpha'}}{\sinc{\alpha'}+\p{\sqrt{\mu}b}^2s'^2}d\mu + C a^3\int_{0}^1 \frac{\mu\p{\sinc{\alpha'}+s'^2}}{\sinc{\alpha'}+\p{\sqrt{\mu}b}^2s'^2}d\mu\\
&\leq C a \int_{0}^1\frac{\mu \sinc{\alpha'}}{\sinc{\alpha'}+\p{\sqrt{\mu}b}^2s'^2}d\mu + C a^3\int_{0}^1 \frac{\mu s'^2}{\sinc{\alpha'}+\p{\sqrt{\mu}b}^2s'^2}d\mu,
\end{align*}
where we have used \eqref{conmu}. In a similar way
\begin{align*}
&\ab{\frac{8}{a}\sinc{\alpha'}\p{\frac{1}{4\sinc{\alpha'}}-\frac{1}{A[a]}}}\leq C a \frac{\sinc{\alpha'}}{\sinc{\alpha'}+b^2s'^2} + C a^3\frac{ s'^2}{\sinc{\alpha'}+b^2s'^2}.
\end{align*}

In addition we have that
\begin{align*}
&\ab{\int_{0}^1\frac{-2a\p{4uu'\sinc{\alpha'}+(u-u')^2}}{4\sinc{\alpha'}+4b\mu (u+u')\sinc{\alpha'}+b^2\mu\p{4uu'\sinc{\alpha'}+(u-u')^2}}d\mu}\\ &
\leq C a \int_{0}^1\frac{\sinc{\alpha'}+s'^2}{\sinc{\alpha'}+(\sqrt{\mu}b)^2s'^2}d\mu\\
&\ab{\frac{4a\p{4uu'\sinc{\alpha'}+(u-u')^2}}{A[a]}}\leq C a \frac{\sinc{\alpha'}+s'^2}{\sinc{\alpha'}+b^2s'^2},
\end{align*}
and we have already proven the first inequality of the lemma.

Taking one derivative we have that
\begin{align*}
\pa\pa_a \p{\frac{1}{a^2}\log\p{\frac{A[a]}{A[0]}}} =-\frac{2}{a^3}\frac{\pa A[a]}{A[a]}+\frac{1}{a^2}\frac{\pa\pa_a A[a]}{A[a]}-\frac{1}{a^2}\frac{\pa_a A[a]\pa A[a]}{A[a]^2}.
\end{align*}
On one hand, we can compute that
\begin{align*}
-\frac{2}{a^3}\frac{\pa A[a]}{A[a]}+\frac{1}{a^2}\frac{\pa\pa_a A[a]}{A[a]}=2a\frac{4\pa(uu')\sinc{\alpha'}+2(u-u')\pa(u-u')}{A[a]}
\end{align*}
and then
\begin{align*}
\ab{-\frac{2}{a^3}\frac{\pa A[a]}{A[a]}+\frac{1}{a^2}\frac{\pa\pa_a A[a]}{A[a]}}\leq C a\frac{\sinc{\alpha'}+s'^2}{\D}.
\end{align*}
On the other hand
\begin{align*}
\ab{\frac{1}{a^2}\frac{\pa_a A[a]\pa A[a]}{A[a]^2}}\leq C a \frac{\p{\sinc{\alpha'}+bs'^2}^2}{\p{\D}^2}.
\end{align*}
Therefore
\begin{align*}
\ab{\pa\pa_a \p{\frac{1}{a^2}\log\p{\frac{A[a]}{A[0]}}}}\leq C a\frac{\sinc{\alpha'}+s'^2}{\D}+C a \frac{\p{\sinc{\alpha'}+bs'^2}^2}{\p{\D}^2}.
\end{align*}
Taking two derivatives we have that
\begin{align*}
&\pa^2\pa_a \p{\frac{1}{a^2}\log\p{\frac{A[a]}{A[0]}}}=-\frac{2}{a^3}\frac{\pa^2 A[a]}{A[a]}+\frac{2}{a^3}\frac{\pa A[a]^2}{A[a]^2}+\frac{1}{a^2}\frac{\pa^2\pa_a A[a]}{A[a]}-\frac{2}{a^2}\frac{\pa \pa_a A[a]\pa A[a]}{A[a]^2}\\
&-\frac{1}{a^2}\frac{\pa_a A[a]\pa^2 A[a]}{A[a]^2}+ \frac{2}{a^2}\frac{\pa_a A[a]\pa A[a]^2}{A[a]^3}.
\end{align*}
And we can estimate
\begin{align*}
&\ab{-\frac{2}{a^3}\frac{\pa^2 A[a]}{A[a]}+\frac{1}{a^2}\frac{\pa^2\pa_a A[a]}{A[a]}}\leq C a \frac{\p{\sinc{\alpha'}+\ms+s'^2+|s'|}}{\D}\\
&\ab{\frac{2}{a^3}\frac{\pa A[a]^2}{A[a]^2}}\leq C a \frac{\p{\sinc{\alpha'}+bs'^2}^2}{\p{\D}^2},\\
&\ab{\frac{2}{a^2}\frac{\pa \pa_a A[a]\pa A[a]}{A[a]^2}}\leq C a \frac{\p{\sinc{\alpha'}+bs'^2}^2}{\p{\D}^2},\\
&\ab{\frac{1}{a^2}\frac{\pa_a A[a]\pa^2 A[a]}{A[a]^2}}\leq C a \frac{\p{\sinc{\alpha'}+b s'^2}\p{\sinc{\alpha'}+b\p{\ms+s'^2+|s'|}}}{\p{\D}^2},\\
&\ab{\frac{2}{a^2}\frac{\pa_a A[a]\pa A[a]^2}{A[a]^3}}\leq C a^3\frac{\p{\sinc{\alpha'}+bs'^2}^3}{\p{\D}^3}
\end{align*}
This proves the third inequality.

Taking three derivatives yields
\begin{align*}
&\pa^3\pa_a \p{\frac{1}{a^2}\log\p{\frac{A[a]}{A[0]}}}=-\frac{2}{a^3}\frac{\pa^3 A[a]}{A[a]}+\frac{6}{a^3}\frac{\pa^2A[a]\pa A[a]}{A[a]^2}-\frac{4}{a^3}\frac{\p{\pa A[a]}^3}{A[a]^3}+\frac{1}{a^2}\frac{\pa^3\pa_a A[a]}{A[a]}\\
&-\frac{3}{a^2}\frac{\pa^2\pa_a A[a]\pa A[a]}{A[a]^2}-\frac{3}{a^2}\frac{\pa \pa_a A[a]\pa^2 A[a]}{A[a]^2}+\frac{6}{a^2}\frac{\pa \pa_a A[a]\p{\pa A[a]}^2}{A[a]^3}-\frac{1}{a^2}\frac{\pa_a A[a]\pa^3 A[a]}{A[a]^2}\\
&+\frac{6}{a^2}\frac{\pa A[a]\pa^2A[a]\pa A[a]}{A[a]^3}-\frac{6}{a^2}\frac{\pa_a A[a]\p{\pa A[a]}^3}{A[a]^4}.
\end{align*}
And we can estimate
\begin{align*}
&\ab{\frac{6}{a^3}\frac{\pa^2A[a]\pa A[a]}{A[a]^2}}\leq C a\frac{\p{\sinc{\alpha'}+bs'^2}\p{\sinc{\alpha'}+b\p{\ms+s'^2+|s'|}}}{\p{\D}^2}\\
&\ab{\frac{4}{a^3}\frac{\p{\pa A[a]}^3}{A[a]^3}}\leq C a^3 \frac{\p{\sinc{\alpha'}+bs'^2}^3}{\p{\D}^3}\\
&\ab{\frac{3}{a^2}\frac{\pa^2\pa_a A[a]\pa A[a]}{A[a]^2}}\leq C a\frac{\p{\sinc{\alpha'}+bs'^2}\p{\sinc{\alpha'}+b\p{\ms+s'^2+|s'|}}}{\p{\D}^2}\\
&\ab{\frac{3}{a^2}\frac{\pa \pa_a A[a]\pa^2 A[a]}{A[a]^2}} \leq C a\frac{\p{\sinc{\alpha'}+bs'^2}\p{\sinc{\alpha'}+b\p{\ms+s'^2+|s'|}}}{\p{\D}^2}\\
&\ab{\frac{6}{a^2}\frac{\pa \pa_a A[a]\p{\pa A[a]}^2}{A[a]^3}} \leq C a^3 \frac{\p{\sinc{\alpha'}+bs'^2}^3}{\p{\D}^3}\\
&\ab{\frac{6}{a^2}\frac{\pa_a A[a]\pa^2A[a]\pa A[a]}{A[a]^3}}\leq  C a\frac{\p{\sinc{\alpha'}+bs'^2}^2\p{\sinc{\alpha'}+b\p{\ms+s'^2+|s'|}}}{\p{\D}^3}\\
&\ab{\frac{6}{a^2}\frac{\pa_a A[a]\p{\pa A[a]}^3}{A[a]^4}}\leq C a^5\frac{\p{\sinc{\alpha'}+bs'^2}^4}{\p{\D}^4}.
\end{align*}
Then we have proven the fourth inequality.

Finally we notice that
\begin{align*}
&\ab{-\frac{2}{a^3}\frac{\pa^3 A[a]}{A[a]}+\frac{1}{a^2}\frac{\pa^3\pa_a A[a]}{A[a]}}\leq C a \frac{\sinc{\al'}}{\D}\ab{\pa^3(uu')}\\
&+C a \frac{\ms +|s'|}{\D}+Ca \frac{\ms+|s'|}{\D}|\pa^3(u-u')|\\
&\ab{\frac{1}{a^2}\frac{\pa_a A[a]\pa^3 A[a]}{A[a]^2}}\\
&\ab{\frac{1}{a^2}\frac{\pa_a A [a]\pa^3 A[a]}{A[a]}}\leq C a \frac{\sinc{\al'}+bs'^2}{\D}\sinc{\al'}|\pa^3(u+u')|\\
&+ C a^3 \frac{\sinc{\al'}+bs'^2}{\D}\sinc{\al'}|\pa^3(uu')|+C a^3 \frac{\sinc{\al'}+bs'^2}{\D}\p{\ms+|s'|}\\
&+ C a^3 \frac{\sinc{\al'}+bs'^2}{\D} \p{\ms+|s'|}|\pa^3(u-u')|.
\end{align*}

\begin{align*}
\end{align*}
\end{proof}

\begin{lemma}\label{B18}
\begin{align*}
\ab{\ab{\int_{-\infty}^\infty\pa^j\p{\pa_a \frac{1}{a^2}\log\p{\frac{A[a]}{A[0]}}}\ms f'd\alpha'ds'}}_{L^2}\leq C a\log\p{\frac{1}{b}}
\end{align*}
for $j=0,1,2.$
\end{lemma}

\begin{proof}
The proof follows from Lemma \ref{B17}, Lemma \ref{A2} and Minkowski's inequality.
\end{proof}

\begin{lemma}\label{B19}
\begin{align*}
\ab{\ab{\int_{-\infty}^\infty\intpi\pa^3\p{\pa_a \frac{1}{a^2}\log\p{\frac{A[a]}{A[0]}}}\ms f'd\alpha'ds'}}_{L^2}\leq C a\log\p{\frac{1}{b}}
\end{align*}
\end{lemma}

\begin{proof}
The proof follows from Lemma \ref{B17}, Lemma \ref{A2} and Minkowski's inequality.
\end{proof}

\begin{lemma}\label{popurri}Let $g\in H^{4,3}$ with $||g||_{H^{4,3}}=1$ then the following estimates hold
\begin{enumerate}
\item $$\ab{A'[u+tg]}\geq c\p{\sinc{\al'}+b^2s'^2}.$$
\item $$\ab{\pa_uA'[g]}\leq Cb\p{\sinc{\al'}+bs'^2}.$$
\item  \begin{align*}
\left|A'[u+t g] -A'[u]\right|\leq C bt\p{\sinc{\alpha'}+bs'^2}.
\end{align*}
\item
\begin{align*}
 \left|\pa A'[u+t g] -\pa A'[u]\right|\leq  C bt\p{\sinc{\alpha'}+bs'^2}.
\end{align*}
\item \begin{align*}
 \left|\pa^2 A'[u+t g] -\pa^2 A'[u]\right|\leq  C bt\p{\sinc{\alpha'}+b\p{\ab{\sin\p{\frac{\alpha'}{2}}}+|s'|+s'^2}}.
\end{align*}
\item
\begin{align*}
\ab{A'[u+tg]-A'[u]-t\pa_u A[g]}\leq Ct^2b^2\p{\sinc{\al'}+s'^2}.
\end{align*}
\item
\begin{align*}
\ab{\pa A'[u+tg]-\pa A'[u]-t \pa \pa_u A[g]}\leq Ct^2b^2\p{\sinc{\al'}+s'^2}
\end{align*}
\item
\begin{align*}
\ab{\pa^2\p{A'[u+tg]-A'[u]-t\pa_uA'[u]}}\leq Ct^2b^2\p{\sinc{\al'}+s'^2+\ab{\sin\p{\frac{\al'}{2}}}+|s'|}
\end{align*}
\end{enumerate}
\end{lemma}
\begin{proof}
We recall that
\begin{align*}
A'[u+tg]=&4\sinc{\al'}+4b\left(u+u'+t(g+g')\right)\sinc{\al'}\\+&b^2\left(4(u+tg)(u'+tg')\sinc{\al'}+(u-u'+t(g-g'))^2\right)\\
=& A'[u]+t\left(b(g+g')\sinc{\al'}+4\left((ug'+u'g)\sinc{\alpha'}+2(u-u')(g-g')\right)b^2\right)\\
+& t^2b^2\left(4gg'\sinc{\al'}+(g-g')^2\right).
\end{align*}
Therefore
\begin{align*}
A'[u+tg]\geq A'[u]-tc (b+b^2)\sinc{\alpha'}-ct^2 b^2\left(\sinc{\alpha'}+s'^2\right).
\end{align*}
And for $t$ small enough ($b\in[0,1]$)
\begin{align*}
A'[u+tg]\geq c \left(\sinc{\al'}+b^2 s'^2\right).
\end{align*}
To prove the second inequality in the statement of the lemma we recall that
\begin{align*}
\pa_u A'[g]= b(g+g')\sinc{\al'}+4b^2\p{\p{ug'+u'g}\sinc{\al'}+2(u-u')(g-g')}
\end{align*}
and then
\begin{align*}
\ab{\pa_u A'}\leq C\p{b \sinc{\alpha'}+b^2s'^2}.
\end{align*}
The third, fourth and  fifth  inequalities are straightforward. Taking two derivatives we have that
\begin{align*}
&\pa^2 A'[u+tg]-\pa^2A'[u]\\&=tb \pa^2(g+g')\sinc{\al'}\\ &+t4b^2\p{\pa^2(ug'+u'g)\sinc{\al'}+2\pa^2(u-u')(g-g')+4\pa(u-u')\pa(g-g')+2(u-u')\pa^2(g-g')}\\
&+t^2b^2\p{4\pa^2(gg')\sinc{\al'}+2(\pa g -\pa g')^2+2(g-g')\pa^2(g-g')}
\end{align*}
and we can bound
\begin{align*}
&\ab{\pa^2 A'[u+tg]-\pa^2A'[u]}\leq t\p{b\sinc{\alpha'}+b^2\p{\sinc{\alpha'}+\left|\sin{\frac{\alpha'}{2}}\right|+s'^2+|s'|}}\\
&+t^2b^2\p{\sinc{\alpha'}+\left|\sin{\frac{\alpha'}{2}}\right|+s'^2+|s'|}.
\end{align*}
To prove 6, 7 and 8 we just notice that
\begin{align*}
A'[u+tg]-A'[u]-t\pa_u A'[g]=t^2b^2\p{4gg'\sinc{\al'}+(g-g')^2}.
\end{align*}
\end{proof}

\begin{lemma}\label{estimacionesdu}The following estimates hold:
\begin{enumerate}
\item \begin{align*}
&\ab{\frac{1}{t}\log\p{\frac{\Ag}{\A}}-\frac{\pa_u A'[g]}{\A}}\\
&\leq C\frac{b\sinc{\al'}+b^2s'^2}{\sinc{\al'}+b^2s'^2}
\p{t\frac{b\sinc{\al'}+b^2 s'^2}{\sinc{\al'}+b^2s'^2}+t^2b^2 \frac{\sinc{\al'}+s'^2}{\sinc{\al'}+b^2s'^2}}
+Cb^2 t \frac{\sinc{\al'}+s'^2}{\sinc{\al'}+b^2 s'^2}.
\end{align*}
\item \begin{align*}
&\ab{\frac{1}{t}\pa \log\p{\frac{\Ag}{\A}}-\pa\frac{\pa_u A[g]}{\A}}\\
&\leq  C t b^2\frac{\p{\sinc{\alpha'}+bs'^2}^2}{\p{\D}^2}+ Ctb^2 \frac{\sinc{\alpha'}+s'^2}{\D}\\
&+ Ct b^3 \frac{\p{\sinc{\alpha'}+bs'^2}^3}{\p{\D}^3}+ Ct b^3 \frac{\p{\sinc{\alpha'}+bs'^3}\p{\sinc{\alpha'}+s'^2}}{\p{\D}^2}.
\end{align*}

\item
\begin{align*}
&\ab{\frac{1}{t}\pa^2 \log\p{\frac{\Ag}{\A}}-\pa^2\frac{\pa_u A[g]}{\A}}\\
&\leq  Cb^2t \frac{\p{\sinc{\alpha'}+bs'^2}\p{\sinc{\alpha'}+b\p{\ab{\sin\p{\frac{\al'}{2}}}+|s'|+s'^2}}}{\p{\D}^2}\\
&+ Cb^2t \frac{\p{\sinc{\alpha'}+bs'^2}^2}{\p{\D}^2}+ Cb^3 t \frac{\p{\sinc{\alpha'}+bs'^2}^3}{\p{\D}^3}\\
&+Cb^2t \frac{\sinc{\alpha'}+\ms+s'^2+|s'|}{\D}+ Cb^3t \frac{\p{\sinc{\alpha'}+bs'^2}\p{\sinc{\alpha'}+s'^2}}{\p{\D}^2}\\
%&+C b^3 t \frac{\p{\sinc{\alpha'}+b'^2}^3}{\p{\D}^3}\\
&+C b^3t \p{\frac{\p{\sinc{\alpha'}+b\p{\ms+s'^2+|s'|}}}{\p{\D}^3}+b\frac{\p{\sinc{\alpha'}+bs'^2}^2}{\p{\D}^4}}\p{\sinc{\alpha'}+bs'^2}^2\\
%&+Cb^3t\frac{\p{\sinc{\alpha'}+b s'^2}\p{\sinc{\alpha'}+s'^2}}{\p{\D}^2}\\
&+ Cb^3t  \frac{\left(\sinc{\al'}+b\left(\left|\sin\left(\frac{\al'}{2}\right)\right|+|s'|+s'^2\right)\right)}{\p{\D}^2}\p{\sinc{\al'}+s'^2}\\
&+ Cb^4t
\frac{\p{\sinc{\al}+bs'^2}^2}{\p{\D}^3}\p{\sinc{\al'}+s'^2}
\end{align*}
\item
%%%%%%%%%%%%%%%%%%%%%%%%%%%%%%%%%%%%%%%%%%%%%%%%%%%%%%%%%%%%%%%%%%%%%%%%%%%%%%%%%%%%%%%%%%%%%%%%%%%%%%%%%%%%%%%%%%%%%%%%%%%%%%%%%%%%%%%%%%%%%%%%%%%%%%%%%
%%%%%%%%%%%%%%%%%%%%%%%%%%%%%%%%%%%%%%%%%%%%%%%%%%%%%%%%%%%%%%%%%%%%%%%%%%%%%%%%%%%%%%%%%%%%%%%%%%%%%%%%%%%%%%%%%%%%%%%%%%%%%%%%%%%%%%%%%%%%%%%%%%%%%%%%%
\begin{align*}
&\ab{\frac{1}{t}\pa^3 \log\p{\frac{\Ag}{\A}}-\pa^3\frac{\pa_u A[g]}{\A}}\leq \\
&Cb\frac{\sinc{\alpha'}+bs'^2}{\p{\D}^2}\\
&\times \left( bt\sinc{\al'}\ab{\pa^3(g+g')}+Cb^2t \sinc{\al'}\ab{\pa^3(ug'+gu')}\right.\\
&+b^2t \p{\ms+|s'|}\ab{\pa^3(u-u')}+b^2t\p{\ms+|s'|}\ab{\pa^3(g-g')}\\&\left.+b^2t\p{\ms+|s'|}+b^2t\sinc{\al'}\ab{\pa^3(gg')}\right)\\
& +Cb^2 t\frac{\sinc{\alpha'}}{\D}|\pa^3(gg')|+Cb^2 t \frac{\ms+|s'|}{\D}\ab{\pa^3(g-g')}\\
&+Cb^2 t \dA^2\\
&\times \left(b\frac{\sinc{\alpha'}}{\p{\D}^3}\ab{\pa^3(u+u')}+b^2\frac{\sinc{\alpha'}}{\p{\D}^3}\ab{\pa^3\p{uu'}}\right.\\
&\left.+Cb^2\frac{\ms+|s'|}{\p{\D}^3}\ab{\pa^3(u-u')}\right)\\
&+Cb^2t\p{\sinc{\al'}+s'^2}\\
&\times\left( b\frac{\sinc{\alpha'}}{\p{\D}^2}\ab{\pa^3(u+u')}+b^2\frac{\sinc{\alpha'}}{\p{\D}^2}\ab{\pa^3\p{uu'}}\right.\\
&\left.+b^2\frac{\ms+|s'|}{\p{\D}^2}\ab{\pa^3(u-u')}\right)\\
&+ \text{KERNEL}_b(\alpha',s').
\end{align*}
where
\begin{align*}
||\text{KERNEL}_b ||_{L^1(\Omega_a)}\leq Cbt
\end{align*}
\end{enumerate}
\end{lemma}
\begin{proof}
First we focus on the inequality 1. We have that
\begin{align*}
\frac{dA'[u+tg]}{dt}=&b(g+g')\sinc{\al'}+4\left((ug'+u'g)\sinc{\alpha'}+2(u-u')(g-g')\right)b^2\\
+& 2tb^2\left(4gg'\sinc{\al'}+(g-g')^2\right)\\
=& \pa_u A'[g']+2tb^2\left(4gg'\sinc{\al'}+(g-g')^2\right).
\end{align*}
In addition
\begin{align*}
\log\left(\frac{A'[u+tg]}{A'[u]}\right)=\int_{0}^1 \frac{d}{d\mu}\log\left(A'[u+t\mu g]\right)d\mu
\end{align*}
and therefore
\begin{align*}
&\frac{1}{t}\log\left(\frac{A'[u+tg]}{A'[u]}\right)-\frac{\pa_u A'[g]}{A'[u]}\\
=&\int_{0}^1\left(\frac{\pa_u A'[g]+2b^2t\mu\left(4gg'\sinc{\al'}+(g-g')^2\right)}{A'[u+t\mu g]}-\frac{\pa_u A'[g]}{A'[u]}\right)d\mu\\
=&\frac{\pa_u A'[g]}{A'[u]}\int_{0}^1 \frac{A'[u]-A'[u+t\mu g]}{A'[u+t\mu g]}d\mu+ 2b^2t \left(4gg'\sinc{\al'}+(g-g')^2\right)\int_{0}^1\frac{\mu}{A'[u+\mu tg]}d\mu.
\end{align*}

By using lemma \ref{popurri}-1 and  \ref{popurri}-3 we can conclude

\begin{align*}
\left|\frac{A'[u+t\mu g] -A'[u]}{A'[u+t\mu g]}\right|\leq C \mu t \frac{b \sinc{\al'}+b^2s'^2}{\sinc{\al'}+b^2s'^2}+C\mu^2t^2b^2\frac{\sinc{\al'}+s'^2}{\sinc{\al'}+b^2s'^2}
\end{align*}
and therefore
\begin{align*}
\ab{\frac{\pa_u A'[g]}{A'[u]}\int_{0}^1 \frac{A'[u]-A'[u+t\mu g]}{A'[u+t\mu g]}d\mu}\leq C\frac{b\sinc{\al'}+b^2s'^2}{\sinc{\al'}+b^2s'^2}
\p{t\frac{b\sinc{\al'}+b^2 s'^2}{\sinc{\al'}+b^2s'^2}+t^2b^2 \frac{\sinc{\al'}+s'^2}{\sinc{\al'}+b^2s'^2}}.
\end{align*}
In addition
\begin{align*}
\ab{2b^2t\p{4gg'\sinc{\al'}+(g-g')^2}\int_{0}^1\frac{\mu}{A'[u+t\mu g]}d\mu}\leq Cb^2 t \frac{\sinc{\al'}+s'^2}{\sinc{\al'}+b^2 s'^2}
\end{align*}
Then by using lemma \ref{A2} we can conclude the inequality 1 of the lemma.

To prove  the inequality 2 we take one derivative
\begin{align}\label{macroderivada}
&\pa \frac{1}{t}\log\p{\frac{A'[u+tg]}{A'[u]}}-\pa \frac{\pa_u A'[g]}{A'[u]}=\frac{1}{t}\p{\frac{\pa A'[u+tg]}{A'[u+tg]}- \frac{\pa  A'[u]}{A'[u]}}-\pa \frac{\pa_u A'[g]}{A'[u]}\\
&=\frac{1}{t}\p{\frac{\pa A'[u+tg]-\pa A[u]}{A'[u+tg]}}-\frac{\pa \pa_u A'[u]}{A'[u]}+\frac{1}{t}\pa A[u] \p{\frac{1}{A'[u+tg]}-\frac{1}{A'[u]}}+\frac{\pa A'[u]}{A'[u]^2}\pa_u A'[g]\nonumber\\
&=\frac{1}{t}\p{\pa A'[u+tg]-\pa A'[u]}\p{\frac{1}{A'[u+tg]}-\frac{1}{A'[u]}}+\frac{1}{t}\frac{\pa A'[u+tg]-\pa A[u]-t \pa \pa_u A'[g]}{A'[u]}\nonumber\\
&+\frac{\pa A'[u]}{t}\p{\frac{1}{A'[u+tg]}-\frac{1}{A'[u]}+t \frac{\pa_u A'[g]}{A'[u]^2}}\nonumber\\
&=\frac{1}{t}\p{\pa A'[u+tg]-\pa A'[u]}\p{A'[u]-A'[u+tg]}\frac{1}{A'[u+tg]A'[u]}\nonumber\\
& +\frac{1}{t}\frac{\pa A'[u+tg]-\pa A[u]-t\pa\pa_uA'[g]}{A'[u]}\nonumber\\
&+\frac{1}{t}\frac{\pa A'[u]}{A'[u]^2A'[u+tg]}\p{A'[u]-A'[u+tg]}^2\nonumber \\ &+\frac{1}{t}\frac{\pa A'[u]}{A'[u]^2}\p{A'[u]-A'[u+tg]+t\pa_u A'[g]}\nonumber\\
&=J_1+J_2+J_3+J_4\nonumber.
\end{align}
And we can apply \ref{popurri} and \ref{A2} to obtain

\begin{align*}
&\ab{J_1}=\ab{\frac{1}{t}\frac{\Ag-\A}{\Ag\A}\pa(\Ag-\A)}\leq C t b^2\frac{\p{\sinc{\alpha'}+bs'^2}^2}{\p{\D}^2}\\
&\ab{J_2}=\ab{\frac{1}{t}\frac{\pa\Ag-\pa \A-t\pa \pa_u A'[g]}{\A}}\leq Ctb^2 \frac{\sinc{\alpha'}+s'^2}{\D}\\
&\ab{J_3}=\ab{\frac{\pa \A}{t}\frac{\p{\Ag-\A}^2}{\A^2\Ag}}\leq Ct b^3 \frac{\p{\sinc{\alpha'}+bs'^2}^3}{\p{\D}^3}\\
&\ab{J_4}=\ab{\frac{\pa A'[u]}{t}\frac{A'[u]-A'[u+tg]+t\pa_u A'[g]}{A'[u]^2}}\\
&\leq Ct b^3 \frac{\p{\sinc{\alpha'}+bs'^3}\p{\sinc{\alpha'}+s'^2}}{\p{\D}^2}.
\end{align*}
In order to prove the inequality 3 we compute two derivatives to obtain

\begin{align*}
&\pa J_1=\frac{1}{t}\frac{\pa^2\p{A'[u+tg]-A'[u]}\p{A'[u+tg]-A'[u]}}{A'[u+tg]A'[u]}+\frac{1}{t}\frac{\p{\pa A'[u+tg]-\pa A'[u]}^2}{A'[u+tg]A'[u]}\\
&-\frac{1}{t}\p{\pa A'[u+tg]-\pa A'[u]}\p{A'[u+tg]-A'[u]}\pa\frac{1}{\Ag\A}\\
&\equiv J_{11}+J_{12}+J_{13}.
\end{align*}
 Since
\begin{align*}
\pa\frac{1}{\Ag \A}=\frac{1}{\Ag\A}\p{\frac{\pa \Ag}{\Ag}+\frac{\pa \A}{\A}}
\end{align*}
and therefore
\begin{align}\label{paaga}
\ab{\pa\frac{1}{\Ag \A}}\leq Cb \frac{\sinc{\alpha'}+bs'^2}{\p{\D}^3}
\end{align}
we have that
\begin{align*}
&\ab{J_{11}}\leq Cb^2t \frac{\p{\sinc{\alpha'}+bs'^2}\p{\sinc{\alpha'}+b\p{\ab{\sin\p{\frac{\al'}{2}}}+|s'|+s'^2}}}{\p{\D}^2}\\
&\ab{J_{12}}\leq Cb^2t \frac{\p{\sinc{\alpha'}+bs'^2}^2}{\p{\D}^2}\\
&\ab{J_{13}}\leq Cb^3 t \frac{\p{\sinc{\alpha'}+bs'^2}^3}{\p{\D}^3}.
\end{align*}
In addition
\begin{align*}
&\pa J_2= \frac{1}{t} \frac{\pa^2 A'[u+tg]-\pa^2 A'[u]-t \pa^2 \pa_u A'[g]}{A'[u]}-\frac{\pa A'[u]}{t}\frac{\pa A'[u+tg]-\pa A'[u]-t \pa\pa_u A'[g]}{A'[u]^2}\\
&\equiv J_{21}+J_{22}.
\end{align*}
And we can estimate
\begin{align*}
&\ab{J_{21}}\leq Cb^2t \frac{\sinc{\alpha'}+\ms+s'^2+|s'|}{\D}\\
&\ab{J_{22}}\leq Cb^3t \frac{\p{\sinc{\alpha'}+bs'^2}\p{\sinc{\alpha'}+s'^2}}{\p{\D}^2}
\end{align*}
\begin{align*}
&\pa J_3=  (\Ag-\A)^2\pa \frac{\pa \A}{\Ag\A^2}+2 \frac{\pa A'[u]}{t}\frac{(A'[u]-A'[u+tg])(\pa A'[u]-\pa A'[u+tg])}{A'[u]^2A'[u+tg]}\\
&\equiv J_{32}+J_{31}
\end{align*}
Since
\begin{align*}
&\pa \frac{\pa \A}{\Ag\A^2}=\frac{\pa^2\A}{\Ag\A^2}-\frac{\pa\A\pa\Ag}{ \Ag^2\A^2}-2\frac{\pa\A^2}{ \Ag\A^3}\\
\end{align*}
we can estimate that
\begin{align}
&\ab{\pa \frac{\pa \A}{\Ag\A^2}}\leq Cb\frac{\p{\sinc{\alpha'}+b\p{\ms+s'^2+|s'|}}}{\p{\D}^3}\nonumber\\
&+Cb^2\frac{\p{\sinc{\alpha'}+bs'^2}^2}{\p{\D}^4}\label{otrapa}
\end{align}

And we have that
\begin{align*}
&\ab{J_{31}}\leq C b^3 t \frac{\p{\sinc{\alpha'}+b'^2}^3}{\p{\D}^3}\\
&\ab{J_{32}}\leq C b^3t \p{\frac{\p{\sinc{\alpha'}+b\p{\ms+s'^2+|s'|}}}{\p{\D}^3}+\frac{\p{\sinc{\alpha'}+bs'^2}^2}{\p{\D}^4}}\p{\sinc{\alpha'}+bs'^2}^2
\end{align*}

\begin{align*}
&\pa J_4 = \frac{\pa A'[u]}{t}\frac{\pa A'[u]-\pa A'[u+tg]+t\pa \pa_u A'[g]}{A'[u]^2}+\pa\p{\frac{\pa A'[u]}{A'[u]^2}}\frac{1}{t}\p{A'[u]-A'[u+tg]+t\pa_u A'[g]}\\
&\equiv J_{41}+J_{42},
\end{align*}
where $J_{41}$ is the same than $J_{22}$.
Since
\begin{align*}
\pa \p{\frac{\pa A'[u]}{A'[u]^2}}=\frac{\pa^2 A'[u]}{A'[u]^2}-2\frac{\p{\pa A'[u]}^2}{A'[u]^3}
\end{align*}
and then
\begin{align}
\ab{\pa \p{\frac{\pa A'[a]}{A'[a]^2}}}\leq  Cb \frac{\left(\sinc{\al'}+b\left(\left|\sin\left(\frac{\al'}{2}\right)\right|+|s'|+s'^2\right)\right)}{\p{\D}^2}+Cb^2\frac{\p{\sinc{\al}+bs'^2}^2}{\p{\D}^3}\label{pisco},
\end{align}

\begin{align*}
%&\ab{J_{41}}\leq Cb^3t\frac{\p{\sinc{\alpha'}+b s'^2}\p{\sinc{\alpha'}+s'^2}}{\p{\D}^2}\\
&\ab{J_{42}}\leq Cb^3t  \frac{\left(\sinc{\al'}+b\left(\left|\sin\left(\frac{\al'}{2}\right)\right|+|s'|+s'^2\right)\right)}{\p{\D}^2}\p{\sinc{\al'}+s'^2}\\&+ Cb^4t
\frac{\p{\sinc{\al}+bs'^2}^2}{\p{\D}^3}\p{\sinc{\al'}+s'^2}
\end{align*}
Finally we compute 3 derivatives. We have to differentiate $J_{11}, J_{12}, J_{13}, J_{21}, J_{22}, J_{31}, J_{32}$ and $J_{42}$.

For $J_{11}$ we have that

\begin{align*}
&\pa J_{11}=\frac{1}{t}\frac{\p{\pa^3 A'[u+tg]-\pa^3 A'[u]}\p{A'[u+tg]-A'[u]}}{A'[u+tg]A'[u]}+\frac{1}{t}\frac{\p{\pa^2 A'[u+tg]-\pa^2A'[u]}\p{\pa A'[u+tg]-\pa A'[u]}}{A'[u+tg]A'[u]}\\
&-\frac{1}{t} \p{\pa^2A'[u+tg]-\pa^2A'[u]}\p{A'[u+tg]-A'[u]}\pa\p{\frac{1}{\Ag\A}}\\
&\equiv J_{111}+J_{112}+J_{113}.
\end{align*}
In order to bound this terms we notice that
\begin{align*}
&\pa^3\Ag-\pa^3\A=4bt\pa^3(g+g')\sinc{\alpha'}\\&
+b^2 t \p{4\pa^3(ug'+gu')\sinc{\alpha'}+2(g-g')\pa^3(u-u')+2(u-u')\pa^3(g-g')}\\
+&b^2t\p{6\pa(u-u')\pa^2(g-g')+6\pa(g-g')\pa^2(u-u')}\\
&+b^2t^2\p{4\pa^3(gg')\sinc{\al'}+4(g-g')\pa^3(g-g')+12\pa\p{g-g'}\pa^2(g-g')}.
\end{align*}
Thus
\begin{align*}
&\ab{\pa^3\Ag-\pa^3\A}\leq Cbt\sinc{\al'}\ab{\pa^3(g+g')}+Cb^2t \sinc{\al'}\ab{\pa^3(ug'+gu')}\\
&+Cb^2t \p{\ms+|s'|}\ab{\pa^3(u-u')}+Cb^2t\p{\ms+|s'|}\ab{\pa^3(g-g')}\\&+Cb^2t\p{\ms+|s'|}+Cb^2t^2\sinc{\al'}\ab{\pa^3(gg')}
\end{align*}
And using lemma \ref{popurri} and \ref{paaga} we obtain
\begin{align*}
&\ab{J_{111}}\leq Cb\frac{\sinc{\alpha'}+bs'^2}{\p{\D}^2}\\
&\times \left( bt\sinc{\al'}\ab{\pa^3(g+g')}+Cb^2t \sinc{\al'}\ab{\pa^3(ug'+gu')}\right.\\
&+b^2t \p{\ms+|s'|}\ab{\pa^3(u-u')}+b^2t\p{\ms+|s'|}\ab{\pa^3(g-g')}\\&\left.+b^2t\p{\ms+|s'|}+b^2t^2\sinc{\al'}\ab{\pa^3(gg')}\right)\\
&\ab{J_{112}}\leq Cb^2t\frac{\p{\sinc{\alpha'}+bs'^2}\p{\sinc{\alpha'}+b\p{\ms+|s'|+s'^2}}}{\p{\D}^2}\\
&\ab{J_{113}}\leq  Cb^3t\frac{\p{\sinc{\alpha'}+bs'^2}^2\p{\sinc{\alpha'}+b\p{\ms+|s'|+s'^2}}}{\p{\D}^3}
\end{align*}

For $J_{12}$

\begin{align*}
&\pa J_{12}=\frac{2}{t}\frac{\p{\pa \Ag-\pa \A}\p{\pa^2 \Ag-\pa^2 \A}}{\Ag\A}-\frac{1}{t}\p{\pa \Ag-\pa \A}^2\pa \frac{1}{\Ag\A}\\
&\equiv J_{121}+J_{122},
\end{align*}
where $J_{121}$ is equivalent to $J_{112}$. Then we just need to bound $J_{122}$. By using lemma \ref{popurri} and \ref{paaga} we have that
\begin{align*}
\ab{J_{122}}\leq Cb^3t\frac{\p{\sinc{\al'}+bs'^2}^3}{\p{\D}^3}.
\end{align*}

For $J_{13}$

\begin{align*}
&\pa J_{13}=\frac{1}{t}\p{\pa^2 \Ag-\pa^2 \A}\p{\Ag-\A}\pa\frac{1}{\Ag\A}\\&+\frac{1}{t}\p{\pa \Ag-\pa \A}^2\pa\frac{1}{\Ag\A}\\
&+\frac{1}{t}\p{\pa \Ag-\pa\A}\p{\Ag-\A}\pa^2\frac{1}{\Ag\A}\\
&\equiv J_{131}+J_{132}+J_{133},
\end{align*}
where $J_{131}$ is equivalent to $J_{113}$ and $J_{132}$ is equivalent to $J_{122}.$ Then we just have to bound $J_{133}.$

Since
\begin{align*}
&\pa^2\frac{1}{\Ag\A}=-\pa\frac{1}{\Ag\A}\p{\frac{\pa \Ag}{\Ag}+\frac{\pa \A}{\A}}\\&-\frac{1}{\Ag\A}\p{\pa \frac{\pa\Ag}{\Ag}+\pa \frac{\pa\Ag}{\Ag}}
\end{align*}
Then using \eqref{pa1entrea1}, \eqref{pa1entrea2}, \eqref{paaga} and lemma \ref{A1} we obtain
\begin{align*}
&\ab{\pa^2\frac{1}{\Ag\A}}\\
&\leq Cb^2 \frac{\p{\sinc{\alpha'}+bs'^2}^2}{\p{\D}^4}+Cb\frac{\sinc{\alpha'}+b\p{\ms+|s'|+s'^2}}{\p{\D}^3}+Cb^2\frac{\left(\sinc{\alpha'}+bs'^2\right)^2}{\p{\D}^2}
\end{align*}
and thanks to lemma \ref{popurri}
\begin{align*}
&\ab{J_{133}}\leq Cb^2t\p{\sinc{\alpha'}+bs'^2}^2\\
&\times \p{b^2 \frac{\p{\sinc{\alpha'}+bs'^2}^2}{\p{\D}^4}+b\frac{\sinc{\alpha'}+b\p{\ms+|s'|+s'^2}}{\p{\D}^3}+b^2\frac{\left(\sinc{\alpha'}+bs'^2\right)^2}{\p{\D}^2}}.
\end{align*}

\begin{align*}
&\pa J_{21}=\frac{1}{t} \frac{\pa^3 \Ag - \pa^3 \A -t \pa^3\pa_u A[g]}{\A}\\&
+\frac{1}{t}\p{\pa^2 \Ag-\pa^2 \A-t\pa^2 \pa_u A'[g]}\pa \frac{1}{\A}\\
&\equiv J^*_{211}+J_{212}.
\end{align*}

We can estimate by using lemma \ref{popurri}
\begin{align*}
&\ab{J_{211}}\leq Cb^2 t\frac{\sinc{\alpha'}}{\D}|\pa^3(gg')|+Cb^2t \frac{\ms+|s'|}{\D}+Cb^2 t \frac{\ms+|s'|}{\D}\ab{\pa^3(g-g')}\\
&\ab{J_{212}}\leq Ct b^3\p{\sinc{\alpha'}+\ms+ s'^2+|s'|}\frac{\sinc{\alpha'}+bs'^2}{\p{\D}^2}
\end{align*}

For $J_{22}$ we have that

\begin{align*}
&\pa J_{22}=\frac{1}{t}\pa^2\frac{1}{\A} \p{\pa \Ag -\pa \A-t\pa \pa_u A'[g]}+\frac{1}{t}\pa \frac{1}{\A}\p{\pa^2 \Ag-\pa^2 \A-t\pa^2 \pa_u A'[g]}\\
&\equiv J_{221}+J_{222},
\end{align*}
where $J_{222}$ is equivalent to $J_{212}$. Then we have to bound $J_{221}$. From \ref{pisco} and lemma \ref{popurri} we have that
\begin{align*}
&\ab{J_{221}}\leq Ctb^2\p{\sinc{\alpha'}+bs'^2} \\&\times\p{b \frac{\left(\sinc{\al'}+b\left(\left|\sin\left(\frac{\al'}{2}\right)\right|+|s'|+s'^2\right)\right)}{\p{\D}^2}+b^2\frac{\p{\sinc{\al}+bs'^2}^2}{\p{\D}^3}}
\end{align*}

For $J_{31}$ we have that
\begin{align*}
&\pa J_{31}= \frac{2}{t}\pa\p{\frac{\pa \A}{\A^2\Ag}}(\A-\Ag)(\pa\A-\pa\Ag)\\&+\frac{1}{t}\pa^2\p{\frac{\pa \A}{\A^2\Ag}}(\A-\Ag)^2\\
&\equiv J_{311}+J_{312}.
\end{align*}
We notice that
\begin{align*}
&\pa^2 \frac{\pa \A}{\A^2\Ag}=\frac{\pa^3\A}{\A^2\Ag}-6\frac{\pa^2\A\pa\A}{\A^3\Ag}-2\frac{\pa^2\A\pa\Ag}{\A^2\Ag}+6\frac{\pa\A^3}{\A^3\Ag}\\
&+4\frac{\pa \A^2\pa\Ag}{\A^3\Ag^2}-\frac{\pa^2\Ag\pa\A}{\A^2\Ag^2}+2\frac{\pa\A\pa\Ag^2}{\A^2\Ag^3}
\end{align*}

By using lemma \ref{A1} we obtain that
\begin{align*}
&\ab{\frac{\pa^3 \A}{\Ag\A^2}}\leq Cb\frac{\sinc{\alpha'}}{\p{\D}^3}\ab{\pa^3(u+u')}+Cb^2\frac{\sinc{\alpha'}}{\p{\D}^3}\ab{\pa{uu'}}\\
&+Cb^2\frac{\ms+|s'|}{\p{\D}^2}+Cb^2\frac{\ms+|s'|}{\p{\D}^3}\ab{\pa^3(u-u')}.
\end{align*}
and
\begin{align*}
&\ab{\pa^2 \frac{\pa \A}{\A^2\Ag}-\frac{\pa^3 \A}{\Ag\A^2}}\leq Cb^2\frac{\ddA\dA}{\p{\D}^4}\\
&+Cb^2\frac{\ddA\dA}{\p{\D}^3}+Cb^3\frac{\dA^3}{\p{\D}^4}\\
&+Cb^3\frac{\dA^3}{\p{\D}^5}+Cb^2\frac{\ddA\dA}{\p{\D}^4}\\
&+Cb^3\frac{\dA\ddA}{\p{\D}^5}\\
&\leq Cb^2\frac{\ddA\dA}{\p{\D}^4}\\
&+Cb^3\frac{\dA^3}{\p{\D}^5}+Cb^2\frac{\ddA\dA}{\p{\D}^4}\\
&+Cb^2\frac{\ddA\dA}{\p{\D}^4}
\end{align*}
And we can estimate
\begin{align*}
&\ab{J_{311}}\leq C b^3 t \p{\frac{\p{\sinc{\alpha'}+b\p{\ms+s'^2+|s'|}}\p{\sinc{\al'}+bs'^2}^2}{\p{\D}^3}}\\
& + Cb^4 t \frac{\p{\sinc{\alpha'}+bs'^2}^4}{\p{\sinc{\al'}+b^2s'^2}^4}\\
&\ab{J_{312}}\leq Cb^2 t \dA^2\\
&\times \left(b\frac{\sinc{\alpha'}}{\p{\D}^3}\ab{\pa^3(u+u')}+b^2\frac{\sinc{\alpha'}}{\p{\D}^3}\ab{\pa^{3}\p{uu'}}\right.\\
&+b^2\frac{\ms+|s'|}{\p{\D}^2}+Cb^2\frac{\ms+|s'|}{\p{\D}^3}\ab{\pa^3(u-u')}\\
&+b^2\frac{\ddA\dA}{\p{\D}^4}\\
&+b^3\frac{\dA^3}{\p{\D}^5}+b^2\frac{\ddA\dA}{\p{\D}^4}\\
&\left.+b^2\frac{\ddA\dA}{\p{\D}^4}\right).
\end{align*}

For $J_{32}$ we have that
\begin{align*}
&\pa J_{32}=\frac{2}{t} \pa\p{\frac{\pa \A}{\A^2\Ag}}(\A-\Ag)(\pa\A-\pa\Ag)\\
&+\frac{2}{t}\frac{\pa \A}{\A^2\Ag}\p{\pa A -\pa \Ag}^2\\
&+\frac{2}{t}\frac{\pa \A}{\A^2\Ag}(\A-\Ag)(\pa^2 \A-\pa^2\Ag)\\
&\equiv J_{321}+J_{322}+J_{323},
\end{align*}
where $J_{321}=J_{311}$.  Then we need to estimate using lemma \ref{A1}
\begin{align*}
&\ab{J_{322}}\leq Cb^3t \frac{\p{\sinc{\alpha'}+bs'^2}^3}{\p{\D}^3}\\
&\ab{J_{323}}\leq Cb^3t \frac{\p{\sinc{\alpha'}+bs'^2}^2}{\p{\D}^3}\p{\sinc{\alpha'}+b\p{\ms +s'^2+|s'|}}\\
\end{align*}
Finally for $J_{42}$ we have that
\begin{align*}
&\pa J_{42}=\pa^3\p{\frac{1}{A}}\frac{1}{t}\p{A-\Ag+t\pa_u A[g]} +\pa^2\p{\frac{1}{\A}}\frac{1}{t}\p{\pa A-\pa Ag+t\pa_u A[g]}\\
&\equiv J_{421}+J_{422},
\end{align*}
where $J_{422}=J_{221}$. And then we just have to bound $J_{421}$. We notice that
\begin{align*}
\pa^3\frac{1}{\A}=\frac{\pa^3\A}{\A^2}-6\frac{\pa^2\A \pa \A}{\A^3}+6\frac{\pa \A ^3}{\A^4}
\end{align*}
And using lemma \ref{A1} we have that
\begin{align*}
&\frac{\pa^3\A}{\A^2}\leq Cb\frac{\sinc{\alpha'}}{\p{\D}^2}\ab{\pa^3(u+u')}+Cb^2\frac{\sinc{\alpha'}}{\p{\D}^2}\ab{\pa{uu'}}\\
&+Cb^2\frac{\ms+|s'|}{\p{\D}^2}+Cb^2\frac{\ms+|s'|}{\p{\D}^2}\ab{\pa^3(u-u')}\\
&\ab{\frac{\pa^2\A\pa \A}{\A^3}}\leq C b^2\frac{\p{\sinc{\al'}+b\p{\ms+s'^2+|s'|}\p{\sinc{\al'}+bs'^2}}}{\p{\D}^3}\\
&\ab{\frac{\pa\A^3}{\A^4}}\leq C b^3 \frac{\p{\sinc{\al'}+bs'^2}^3}{\p{\D}^4}.
\end{align*}
\begin{align*}
&\ab{J_{422}}\leq Cb^2t\p{\sinc{\al'}+s'^2}\\
&\times\left( b\frac{\sinc{\alpha'}}{\p{\D}^2}\ab{\pa^3(u+u')}+b^2\frac{\sinc{\alpha'}}{\p{\D}^2}\ab{\pa^{3}(uu')}\right.\\
&+b^2\frac{\ms+|s'|}{\p{\D}^2}+b^2\frac{\ms+|s'|}{\p{\D}^2}\ab{\pa^3(u-u')}\\
&b^2\frac{\p{\sinc{\al'}+b\p{\ms+s'^2+|s'|}\p{\sinc{\al'}+bs'^2}}}{\p{\D}^3}\\
&\left.b^3 \frac{\p{\sinc{\al'}+bs'^2}^3}{\p{\D}^4}\right).
\end{align*}

This finishes the proof of the lemma.
\end{proof}

\begin{lemma}\label{maindu1} Let $g\in H^{4,3}$ with $||g||_{H^{4,3}}=1$ and  $f(\alpha,s)\in L^2(\Omega)$ with supp$(f)\subset \T\times [-1,1]$. Then the following estimates hold:
\begin{enumerate}

\item\begin{align*}
\ab{\ab{\inti\intpi \pa^j \p{\frac{1}{t}\log\p{\frac{A'[u+tg]}{A'[u]}}-\frac{\pa_u A'[g]}{A'[u]}}f'd\al'ds' }}_{L^2}\leq Cbt,
\end{align*}
\item\begin{align*}
\ab{\ab{\inti\intpi \pa^j \p{\log\p{\frac{A'[u+tg]}{A'[u]}}}f'd\al'ds' }}_{L^2}\leq Cbt,
\end{align*}
for $j=0,1,2$.
\item \begin{align*}\ab{\ab{\inti\intpi \p{\pa^3\p{\frac{1}{t}\log\p{\frac{\Ag}{\A}}-\pa_u A'[g]}-Z(\al,\al',s,s')}f(\alpha',s')d\alpha'ds'}}_{L^2}\leq Cbt
\end{align*}
where
\begin{align*}
&Z(\alpha,\alpha',s,s')=\frac{1}{t}\frac{\p{\pa^3 \Ag-\pa^3\A}\p{\Ag-\A}}{\Ag\A}-\frac{1}{t}\frac{\pa^3\Ag-\pa^3\A-t\pa^3\pa_u A'[g]}{\A}\\
&+\frac{\pa^3\A}{t}\frac{\p{\A-\Ag}^2}{\A^2\Ag}.
\end{align*}
\item \begin{align*}
\ab{\ab{\inti\intpi  \p{\pa^3\log\p{\frac{\Ag}{\A}}-W(\alpha,\alpha',s,s')}d\alpha'ds'}}_{L^2}\leq Cbt
\end{align*}
where
\begin{align*}
W(\alpha, \alpha',s,s')=\pa^3 \A\p{\frac{1}{\Ag}-\frac{1}{\A}}+\frac{\p{\pa^3 \Ag-\pa^3\A}}{\Ag}
\end{align*}
\end{enumerate}
\end{lemma}

\begin{proof}
The proof follows from Lemma \ref{estimacionesdu}, Lemma \ref{A2} and Minkowski's inequality.
\end{proof}

\begin{lemma}\label{maindu1bis} Let $g\in H^{4,3}$ with $||g||_{H^{4,3}}=1$ and  $f(\alpha,s)\in L^\infty(\Omega)$ with supp$(f)\subset \T\times [-1,1]$. Then the following estimates hold:
\begin{align*}
\ab{\ab{\inti\intpi Z(\alpha,\alpha',s,s') f'd\alpha'ds'}}_{L^2}\leq Cbt
\end{align*}
where
\begin{align*}
&Z(\alpha,\alpha',s,s')=\frac{1}{t}\frac{\p{\pa^3 \Ag-\pa^3\A}\p{\Ag-\A}}{\Ag\A}-\frac{1}{t}\frac{\pa^3\Ag-\pa^3\A-t\pa^3\pa_u A'[g]}{\A}\\
&+\frac{\pa^3\A}{t}\frac{\p{\A-\Ag}^2}{\A^2\Ag}.
\end{align*}
\begin{align*}
\ab{\ab{\inti\intpi W(\alpha,\alpha',s,s') f(\alpha',s')d\alpha'ds'}}_{L^2}\leq Cbt
\end{align*}
where
\begin{align*}
W(\alpha, \alpha',s,s')=\pa^3 \A\p{\frac{1}{\Ag}-\frac{1}{\A}}+\frac{\p{\pa^3 \Ag-\pa^3\A}}{\Ag}.
\end{align*}
\end{lemma}

\begin{proof}
The proof follows from Lemma \ref{estimacionesdu}, Lemma \ref{A2} and Minkowski's inequality.
\end{proof}

\begin{lemma}\label{maindu2} Let $g\in H^{4,3}$ with $||g||_{H^{4,3}}=1$ and  $f(\alpha,s)\in L^2(\Omega)$ with supp$(f)\subset \T\times [-1,1]$. Then the following estimates hold:

\begin{align*}
\ab{\ab{\inti\intpi \pa^j \frac{1}{b}\p{\frac{1}{t}\log\p{\frac{A'[u+tg]}{A'[u]}}-\frac{\pa_u A'[g]}{A'[u]}}\left|\sin\p{\frac{\al'}{2}}\right|f'd\al'ds' }}_{L^2}\leq Cbt\log\left(\frac{1}{b}\right),
\end{align*}

for $j=0,1,2$.
 \begin{align*}\ab{\ab{\inti\intpi \frac{1}{b}\p{\pa^3\p{\frac{1}{t}\log\p{\frac{\Ag}{\A}}-\pa_u A'[g]}-Z(\al,\al',s,s')}\left|\sin\p{\frac{\alpha'}{2}}\right|f(\alpha',s')d\alpha'ds'}}_{L^2}\leq Cbt \log\left(\frac{1}{b}\right)
\end{align*}
where
\begin{align*}
&Z(\alpha,\alpha',s,s')=\frac{1}{t}\frac{\p{\pa^3 \Ag-\pa^3\A}\p{\Ag-\A}}{\Ag\A}-\frac{1}{t}\frac{\pa^3\Ag-\pa^3\A-t\pa^3\pa_u A'[g]}{\A}\\
&+\frac{\pa^3\A}{t}\frac{\p{\A-\Ag}^2}{\A^2\Ag}.
\end{align*}

\end{lemma}
\begin{proof}
The proof is as in lemma \ref{maindu1} but using the extra $\sin\p{\frac{\alpha'}{2}}$ of this term.
\end{proof}

\begin{lemma}\label{maindu2bis} Let $g\in H^{4,3}$ with $||g||_{H^{4,3}}=1$ and  $f(\alpha,s)\in L^\infty(\Omega)$ with supp$(f)\subset \T\times [-1,1]$. Then the following estimates hold:
\begin{align*}
\ab{\ab{\inti\intpi Z(\alpha,\alpha',s,s')\left|\sin\p{\frac{\al'}{2}}\right| f'd\alpha'ds'}}_{L^2}\leq Cbt\log\left(\frac{1}{b}\right)
\end{align*}
where
\begin{align*}
&Z(\alpha,\alpha',s,s')=\frac{1}{t}\frac{\p{\pa^3 \Ag-\pa^3\A}\p{\Ag-\A}}{\Ag\A}-\frac{1}{t}\frac{\pa^3\Ag-\pa^3\A-t\pa^3\pa_u A'[g]}{\A}\\
&+\frac{\pa^3\A}{t}\frac{\p{\A-\Ag}^2}{\A^2\Ag}.
\end{align*}

\end{lemma}
\begin{proof}
The proof is as in lemma \ref{maindu1bis} but using the extra $\sin\p{\frac{\alpha'}{2}}$ of this term.
\end{proof}

\begin{lemma}\label{paradifH1} The following estimate holds

\begin{align*}
\ab{\ab{\inti\intpi F_s(s-s')\p{-\frac{2}{a^3}\frac{1}{t}\log\p{\frac{\Ag}{\A}}+\frac{2}{a^3}\frac{\pa_u \A}{\A}}\left|\sin(\alpha')\right|d\al' ds'}}_{H^3}\leq C a\log\p{\frac{1}{b}}
\end{align*}
\end{lemma}
\begin{proof}
The proof is similar to that one in  lemmas \ref{maindu2} and \ref{maindu2bis} but you have to divide between $a$, because we are taking a derivative in $a$.
\end{proof}

\begin{lemma}\label{paradifH21andH22} The following estimates hold

\begin{align*}
\ab{\ab{\inti\intpi F_\rho(s-s')\frac{1}{a^2}\p{\frac{1}{t}\frac{\pa_a \Ag-\pa_a \A-t\pa_u\pa_a A[g]}{\Ag}}\left|\sin(\alpha')\right|d\alpha ds'}}\\
+\ab{\ab{\inti\intpi F_\rho(s-s')\frac{1}{a^2}\frac{\pa_u\pa_a A'[g]}{\Ag \A}\p{\A-\Ag}\left|\sin(\alpha')\right|d\al'ds'}}\leq Ca\log\p{\frac{1}{b}}
\end{align*}

\begin{align*}
&\ab{\ab{\inti\intpi F_\rho(s-s') \frac{1}{a^2}\frac{\pa_a A'[g]}{tA'[u]\Ag}\p{\A-\Ag+t\pa_u A'[g]}\left|\sin(\alpha')\right|d\alpha' ds'}}\\
&+\ab{\ab{\inti\intpi F_\rho(s-s') \frac{1}{a^2}\frac{\pa_a A[g]\pa_u A[g]}{\Ag \A^2}\p{\Ag-\A}\left|\sin(\alpha')\right| d\al' ds'}}\leq C a\log\p{\frac{1}{b}}.
\end{align*}

\end{lemma}
\begin{proof}
The proof is similar to that one of lemmas \ref{maindu2} and \ref{maindu2bis} but dividing by $a$ because we lose a derivative in $a$.
\end{proof}

\begin{lemma}\label{paradifelultimo} The following estimates hold
\begin{align*}
&\ab{\ab{\inti\intpi F_s(s-s')\frac{1}{t}\p{\frac{\pa_a \Ag-\pa_a \A -t \pa_u\pa_a \A}{A[u+tg]}}\cos(\alpha')u'_\alpha d\alpha'ds'}}_{H^3}\\
&+\ab{\ab{ \inti\intpi F_s(s-s')\pa_u\pa_a \Ag \p{\frac{1}{\Ag}-\frac{1}{\A}}\cos(\alpha')u'_\alpha d\alpha'ds'}}_{H^3}\leq C a\log\p{\frac{1}{b}}
\end{align*}
\begin{align*}
&\ab{\ab{\inti\intpi f_s(s-s')\pa_a \A \frac{1}{t}\frac{\A-\Ag+t\pa_u \A}{\Ag\A}\cos(\alpha')u'_\alpha d\alpha'ds'}}\\
&+\ab{\ab{\inti\intpi F_s(s-s')\frac{\pa_a\A\pa_u \A}{\A} \p{\frac{1}{\Ag}-\frac{1}{\A}}\cos(\alpha')u'_\alpha d\alpha' ds'}}\leq C a \log\p{\frac{1}{b}}.
\end{align*}
\end{lemma}
\begin{proof}
The proof of this lemma is similar to the proofs of lemmas \ref{maindu1} and \ref{maindu1bis} but dividing by $a$ since we are taking a derivative in this parameter.
\end{proof}

\textbf{Conflict of Interest:} The authors declare that they have no conflict of interest.

 \bibliographystyle{abbrv}
 \bibliography{references}

\begin{tabular}{l}
\textbf{Angel Castro} \\
  {\small Departamento de Matem\'aticas} \\
 {\small Universidad Aut\'onoma de Madrid} \\
 {\small Instituto de Ciencias Matem\'aticas-CSIC}\\
 {\small C/ Nicolas Cabrera, 13-15, 28049 Madrid, Spain} \\
  {\small Email: angel\_castro@icmat.es} \\
\\
\textbf{Diego C\'ordoba} \\
  {\small Instituto de Ciencias Matem\'aticas} \\
 {\small Consejo Superior de Investigaciones Cient\'ificas} \\
 {\small C/ Nicolas Cabrera, 13-15, 28049 Madrid, Spain} \\
  {\small Email: dcg@icmat.es} \\
\\
\textbf{Javier G\'omez-Serrano} \\
{\small Department of Mathematics} \\
{\small Princeton University}\\
{\small 610 Fine Hall, Washington Rd,}\\
{\small Princeton, NJ 08544, USA}\\
 {\small Email: jg27@math.princeton.edu} \\
  \\

\end{tabular}

\end{document}